\documentclass[a4paper,12pt]{amsart}
\RequirePackage{plautopatch}
\RequirePackage[l2tabu, orthodox]{nag}
\usepackage{bxtexlogo}
\usepackage{amsfonts}
\usepackage{bbm}
\usepackage{amssymb, amsmath,stmaryrd,mathrsfs,mathtools}
\usepackage{tikz,float}
\usepackage{tikz-cd}
\usepackage{amsthm}
\usepackage{enumerate}

\usetikzlibrary{calc}
\usepackage{url}
\usepackage[hidelinks]{hyperref}
\usepackage{graphicx}
\usepackage{xcolor}
\theoremstyle{definition}
\newtheorem{thm}{Theorem}

\newtheorem{lem}[thm]{Lemma}
\newtheorem{cor}[thm]{Corollary}
\newtheorem{exa}[thm]{Example}
\newtheorem{dfn}[thm]{Definition}
\newtheorem{rem}[thm]{Remark}
\newtheorem{que}[thm]{Question}
\newtheorem{fac}[thm]{Fact}
\newtheorem{con}[thm]{Construction}
\newtheorem{ass}[thm]{Assumption}
\newtheorem{mainlem}[thm]{Main Lemma}

\newtheorem{ques}{Question}
\newtheorem*{thmA}{Theorem A}
\newtheorem*{thmB}{Theorem B}
\newtheorem*{thmC}{Theorem C}
\newtheorem*{thmm}{Theorem}
\newtheorem*{acknowledgements}{Acknowledgements}

\newtheorem{dfnfac}[thm]{Definition and Fact}

\numberwithin{thm}{section} 
\numberwithin{equation}{section} 

\renewcommand{\a}{\alpha}
\renewcommand{\b}{\beta}

\newcommand{\p}{\mathbb{P}}

\newcommand{\br}{\mathbb{R}}

\newcommand{\qd}{\dot{\mathbb{Q}}}

\DeclareMathAlphabet{\mymathbb}{U}{BOONDOX-ds}{m}{n}
\newcommand{\zero}{\mymathbb{0}}

\newcommand{\oo}{{\omega}^\omega}
\newcommand{\ooo}{[{\omega}]^\omega}
\newcommand{\fin}{[\omega]^{<\omega}}
\newcommand{\seq}{\omega^{<\omega}}
\newcommand{\sq}{2^{<\omega}}

\newcommand{\on}{\mathpunct{\upharpoonright}}

\newcommand{\bb}{\mathfrak{b}}
\newcommand{\cc}{\mathfrak{c}}
\newcommand{\dd}{\mathfrak{d}}
\newcommand{\ee}{\mathfrak{e}}
\newcommand{\pre}{\mathfrak{pr}}
\newcommand{\hh}{\mathfrak{h}}
\newcommand{\mm}{\mathfrak{m}}
\newcommand{\pp}{\mathfrak{p}}
\newcommand{\rr}{\mathfrak{r}}
\newcommand{\spl}{\mathfrak{s}}

\newcommand{\m}{\mathcal{M}}
\newcommand{\n}{\mathcal{N}}

\DeclareMathOperator{\non}{non}
\DeclareMathOperator{\cov}{cov}
\DeclareMathOperator{\add}{add}
\DeclareMathOperator{\cof}{cof}

\newcommand{\covm}{\cov(\mathcal{M})}
\newcommand{\nonm}{\non(\mathcal{M})}
\newcommand{\addm}{\add(\mathcal{M})}
\newcommand{\cofm}{\cof(\mathcal{M})}
\newcommand{\covn}{\cov(\mathcal{N})}
\newcommand{\nonn}{\non(\mathcal{N})}
\newcommand{\addn}{\add(\mathcal{N})}
\newcommand{\cofn}{\cof(\mathcal{N})}
\newcommand{\covi}{\cov(I)}
\newcommand{\noni}{\non(I)}
\newcommand{\addi}{\add(I)}
\newcommand{\cofi}{\cof(I)}

\DeclareMathOperator{\cob}{COB}

\DeclareMathOperator{\lcu}{LCU}
\newcommand{\R}{\mathbf{R}}

\renewcommand{\lq}{\preceq_T}

\DeclareMathOperator{\dom}{dom}
\DeclareMathOperator{\ran}{ran}

\DeclareMathOperator{\cf}{cf}

\newcommand{\eeb}{\mathfrak{e}^*}
\newcommand{\preb}{\mathfrak{pr}^*}

\newcommand{\pr}{\mathbb{PR}}
\newcommand{\ev}{\mathbb{E}}

\newcommand{\bpd}{\sqsubset^\mathrm{bp}}

\newcommand{\Ad}{\mathbb{S}}

\DeclareMathOperator{\Pred}{Pred}

\title{Cicho\'n's maximum with evasion number}
\author{Takashi Yamazoe}
\address{Graduate School of System Informatics, Kobe University,
	Rokko--dai 1--1, Nada--ku, 657--8501 Kobe, Japan}
\email{212x502x@cloud.kobe-u.jp}
\thanks{This work was supported by JST SPRING, Japan Grant Number JPMJSP2148.}
\begin{document}

	\begin{abstract}
		We show that the evasion number $\ee$ can be added to Cicho\'n's maximum with a distinct value.
		More specifically, it is consistent that 
		$\aleph_1<\addn<\covn<\bb<\ee<\nonm<\covm<\dd<\nonn<\cofn<2^{\aleph_0}$ holds.
	\end{abstract}
	\maketitle

	\section{Introduction}\label{sec_int}
	\subsection{Background on Cicho\'n's maximum and evasion number}
	
	Cardinal invariants of the continuum are cardinals characterizing some structure of the continuum. Well-known examples are the \textit{bounding number} $\bb$ and the \textit{dominating number} $\dd$: 
	For $f,g\in\oo$, let $f\leq^*g$ be defined by: $ f(n)\leq g(n)$ for all but finitely many $n<\omega$. And we define $\bb$ and $\dd$ as follows:
	\begin{itemize}
		\item $\bb\coloneq\min\{|F|: F\subseteq\oo, \lnot(\exists f\in \oo~ \forall g\in F~ g\leq^*f)\}$.
		\item $\dd\coloneq\min\{|F|: F\subseteq\oo, \forall f\in \oo~ \exists g\in F~ f\leq^*g\}$.
	\end{itemize}
	Other examples are those related to an ideal on the reals:
	Let $X$ be a set and $I\subseteq\mathcal{P}(X)$ be an ideal containing all singletons. We define the following four numbers on $I$:
	\begin{itemize}
		\item $\addi:=\min\{|\mathcal{A}|:\mathcal{A}\subseteq I,~\bigcup\mathcal{A}\notin I\}$.
		\item $\covi:=\min\{|\mathcal{A}|:\mathcal{A}\subseteq I,~\bigcup\mathcal{A}=X\}$.
		\item $\noni:=\min{\{}|A|:A\subseteq X,~A\notin I\}$.
		\item $\cofi:=\min{\{}|\mathcal{A}|:\mathcal{A}\subseteq I~\forall B\in I~\exists A\in\mathcal{A}~B\subseteq A\}$.
	\end{itemize}
	The set of all Lebesgue null sets $\n$ and the set of all meager sets $\m$ are ideals on $X=\br$, so we can define the $2*4=8$ cardinal invariants. 
	
	The relationship of these $2+8=10$ cardinal invariants is illustrated in \textit{Cicho\'n's diagram} (see Figure \ref{fig_Cd}). It is said to be \textit{complete} in the sense that we cannot prove any more inequalities between two cardinal invariants in the diagram, in other words, no other arrows can be added to the diagram. Moreover, it is known that the diagram can be ``divided into two parts anywhere''. More precisely, any assignment of $\aleph_1$ and $\aleph_2$ to its numbers is consistent whenever it does not contradict the arrows and the two equations $\addm=\min\{\bb,\covm\}$ and $\cofm=\max\{\dd,\nonm\}$ (see \cite[Chapter 7]{BJ95}).
	
	\begin{figure}[ht]
		
		\centering
		\begin{tikzpicture}
			\tikzset{
				textnode/.style={text=black},
			}
			\tikzset{
				edge/.style={color=black,thin},
			}
			\newcommand{\w}{2.4}
			\newcommand{\h}{1.5}
			
			\node[textnode] (addN) at (0, 0) {$\addn$};
			\node[textnode] (covN) at (0, \h*2) {$\covn$};
			
			\node[textnode] (addM) at (\w, 0) {$\addm$};
			\node[textnode] (b) at (\w, \h) {$\bb$};
			\node[textnode] (nonM) at (\w, \h*2) {$\nonm$};
			
			\node[textnode] (covM) at (\w*2, 0) {$\covm$};
			\node[textnode] (d) at (\w*2, \h) {$\dd$};
			\node[textnode] (cofM) at (\w*2, \h*2) {$\cofm$};

			\node[textnode] (nonN) at (\w*3, 0) {$\nonn$};
			\node[textnode] (cofN) at (\w*3, \h*2) {$\cofn$};
			
			\node[textnode] (aleph1) at (-\w, 0) {$\aleph_1$};
			\node[textnode] (c) at (\w*4, \h*2) {$2^{\aleph_0}$};
			
			\draw[->,edge] (addN) to (covN);
			\draw[->,edge] (addN) to (addM);
			\draw[->,edge] (covN) to (nonM);	
			\draw[->,edge] (addM) to (b);
			\draw[->,edge] (b) to (nonM);
			\draw[->,edge] (addM) to (covM);
			\draw[->,edge] (nonM) to (cofM);
			\draw[->,edge] (covM) to (d);
			\draw[->,edge] (d) to (cofM);
			\draw[->,edge] (b) to (d);
			\draw[->,edge] (covM) to (nonN);
			\draw[->,edge] (cofM) to (cofN);
			\draw[->,edge] (nonN) to (cofN);
			\draw[->,edge] (aleph1) to (addN);
			\draw[->,edge] (cofN) to (c);
		\end{tikzpicture}
		\caption{Cicho\'n's diagram. An arrow $\mathfrak{x}\to\mathfrak{y}$ denotes that $\mathfrak{x}\leq\mathfrak{y}$ holds. Also $\addm=\min\{\bb,\covm\}$ and $\cofm=\max\{\dd,\nonm\}$ hold.}\label{fig_Cd}
	\end{figure}
	
	Since the separations of Cicho\'n's diagram with two values are well studied, we are naturally interested in the separation with more values and in this sense the ultimate question is the following:
	
	\begin{ques}
		Can we separate Cicho\'n's diagram with as many values as possible? 
		In other words, can we construct a model where all the cardinal invariants except for the two dependent numbers $\addm$ and $\cofm$ are pairwise different? 
	\end{ques}
	Such a model is called \textit{Cicho\'n's maximum} model.
	The question was positively solved by Goldstern, Kellner and Shelah \cite{GKS}, assuming four strongly compact cardinals. They constructed a model whose separation order is as in Figure \ref{fig_CM_GKS}.
	%Later, Kellner, Shelah and T{\u{a}}nasie \cite{KST} constructed Cicho\'n's maximum of another order\footnote{Thus, all the two possible models of point-symmetrical orders were constructed at that time.} illustrated in Figure \ref{fig_CM_KST}, assuming the same large cardinals.
	Later, they and Mej\'{\i}a \cite{GKMS} eliminated the large cardinal assumption and hence proved that Cicho\'n's maximum is consistent with ZFC.
	\begin{figure}
		\centering
		\begin{tikzpicture}
			\tikzset{
				textnode/.style={text=black}, 
			}
			\tikzset{
				edge/.style={color=black, thin, opacity=0.5}, 
			}
			\newcommand{\w}{2.4}
			\newcommand{\h}{1.5}
			
			\node[textnode] (addN) at (0,  0) {$\addn$};
			\node (t1) [fill=lime, draw, text=black, circle,inner sep=1.0pt] at (-0.25*\w, 0.5*\h) {$\theta_1$};
			
			\node[textnode] (covN) at (0,  2*\h) {$\covn$};
			\node (t2) [fill=lime, draw, text=black, circle,inner sep=1.0pt] at (0.25*\w, 1.5*\h) {$\theta_2$};

			\node[textnode] (addM) at (\w,  0) {$\cdot$};
			\node[textnode] (b) at (\w,  1*\h) {$\bb$};
			\node (t3) [fill=lime, draw, text=black, circle,inner sep=1.0pt] at (0.75*\w, 1.2*\h) {$\theta_3$};
			
			\node[textnode] (nonM) at (\w,  2*\h) {$\nonm$};
			\node (t4) [fill=lime, draw, text=black, circle,inner sep=1.0pt] at (1.25*\w, 2.4*\h) {$\theta_4$};
			
			\node[textnode] (covM) at (\w*2,  0) {$\covm$};
			\node (t5) [fill=lime, draw, text=black, circle,inner sep=1.0pt] at (1.75*\w, -0.4*\h) {$\theta_5$};
			
			\node[textnode] (d) at (\w*2,  1*\h) {$\dd$};
			\node (t6) [fill=lime, draw, text=black, circle,inner sep=1.0pt] at (2.25*\w, 0.8*\h) {$\theta_6$};
			
			\node[textnode] (cofM) at (\w*2,  2*\h) {$\cdot$};

			\node[textnode] (nonN) at (\w*3,  0) {$\nonn$};
			\node (t7) [fill=lime, draw, text=black, circle,inner sep=1.0pt] at (2.75*\w, 0.5*\h) {$\theta_7$};
			
			\node[textnode] (cofN) at (\w*3,  2*\h) {$\cofn$};
			\node (t8) [fill=lime, draw, text=black, circle,inner sep=1.0pt] at (3.25*\w, 1.5*\h) {$\theta_{8}$};
			
			\node[textnode] (aleph1) at (-\w,  0) {$\aleph_1$};
			\node[textnode] (c) at (\w*4,  2*\h) {$2^{\aleph_0}$};
			\node (tcc) [fill=lime, draw, text=black, circle,inner sep=1.0pt] at (3.7*\w, 2.5*\h) {$\theta_\cc$};

			\draw[->, edge] (addN) to (covN);
			\draw[->, edge] (addN) to (addM);
			\draw[->, edge] (covN) to (nonM);	
			\draw[->, edge] (addM) to (b);
			
			\draw[->, edge] (addM) to (covM);
			\draw[->, edge] (nonM) to (cofM);
			
			\draw[->, edge] (d) to (cofM);
			\draw[->, edge] (b) to (d);
			\draw[->, edge] (b) to (nonM);
			\draw[->, edge] (covM) to (nonN);
			\draw[->, edge] (covM) to (d);
			\draw[->, edge] (cofM) to (cofN);
			\draw[->, edge] (nonN) to (cofN);
			\draw[->, edge] (aleph1) to (addN);
			\draw[->, edge] (cofN) to (c);

			\draw[blue,thick] (-0.5*\w,1*\h)--(0.5*\w,1*\h);
			\draw[blue,thick] (2.5*\w,1*\h)--(3.5*\w,1*\h);

			\draw[blue,thick] (0.5*\w,1.5*\h)--(1.5*\w,1.5*\h);
			\draw[blue,thick] (1.5*\w,0.5*\h)--(2.5*\w,0.5*\h);

			\draw[blue,thick] (1.5*\w,-0.5*\h)--(1.5*\w,2.5*\h);
			%\draw[blue,thick] (0.5*\w,2.5*\h)--(0.5*\w,2.5*\h);
			\draw[blue,thick] (2.5*\w,-0.5*\h)--(2.5*\w,2.5*\h);

			\draw[blue,thick] (0.5*\w,-0.5*\h)--(0.5*\w,2.5*\h);
			%\draw[blue,thick] (2.5*\w,-0.5*\h)--(2.5*\w,0.5*\h);
			\draw[blue,thick] (-0.5*\w,-0.5*\h)--(-0.5*\w,2.5*\h);
			\draw[blue,thick] (3.5*\w,-0.5*\h)--(3.5*\w,2.5*\h);

		\end{tikzpicture}
		\caption{Cicho\'n's maximum constructed in \cite{GKS}. $\aleph_1<\theta_1<\cdots<\theta_8<\theta_\cc$ are regular cardinals. $\addm$ and $\cofm$ are omitted as dots ``$\cdot$'' since they have dependent values.}\label{fig_CM_GKS}
	\end{figure}

	Consequently, the following natural question arises:
	
	\begin{ques}
		\label{ques_2}
		Can we add to Cicho\'n's maximum other cardinal invariants with distinct values?
	\end{ques}
	
	We can, as has been shown, e.g., in \cite{GKMS21a} (where $\mm,\pp$ and $\hh$ are added), or \cite{GKMS21b} (where $\spl$ and $\rr$ are added).
	
	In this paper, we focus on the \textit{evasion number} $\ee$, which was first introduced by Blass \cite{Bla94}. %from a group-theoretic motivation.
	
	\begin{dfn}
		\begin{itemize}
			
			\item A pair $\pi=(D,\{\pi_n:n\in D\})$ is a predictor if $D\in\ooo$ and each $\pi_n$ is a function $\pi_n\colon\omega^n\to\omega$. $\Pred$ denotes the set of all predictors.
			\item $\pi\in \Pred$ predicts $f\in\oo$ if $f(n)=\pi_n(f\on n)$ for all but finitely many $n\in D$. 
			$f$ evades $\pi$ if $\pi$ does not predict $f$.
			\item The prediction number $\pre$ and the evasion number $\ee$ are defined as follows\footnote{While the name ``prediction number'' and the notation ``$\pre$'' are not common, we use them in this paper.}:
			\begin{gather*}
				\mathfrak{pr}\coloneq\min\{|\Pi|:\Pi\subseteq \Pred,\forall f\in\oo~\exists\pi\in \Pi~ \pi \text{ predicts } f\},\\
				\mathfrak{e}\coloneq\min\{|F|:F\subseteq\oo,\forall \pi\in \Pred~\exists f\in F~ f\text{ evades }\pi\}.
			\end{gather*}
			
		\end{itemize}
	\end{dfn}
	The two numbers are embedded % $\ee$ and $\pre$
	into Cicho\'n's diagram as in Figure \ref{fig_Cd+e}.

	\begin{figure}
		\begin{tikzpicture}
			\tikzset{
				textnode/.style={text=black}, 
			}
			\tikzset{
				edge/.style={color=black, thin}, 
			}
			\tikzset{cross/.style={preaction={-,draw=white,line width=9pt}}}
			\newcommand{\w}{2.4}
			\newcommand{\h}{1.8}
			
			\node[textnode] (addN) at (0,  0) {$\addn$};
			\node[textnode] (covN) at (0,  \h*2) {$\covn$};
			
			\node[textnode] (e) at (0.5*\w,  1.3*\h) {$\mathfrak{e}$}; 
			\node[textnode] (pr) at (2.5*\w,  0.7*\h) {$\mathfrak{pr}$}; 
			
			\node[textnode] (addM) at (\w,  0) {$\addm$};
			\node[textnode] (b) at (\w,  \h) {$\bb$};
			\node[textnode] (nonM) at (\w,  \h*2) {$\nonm$};
			
			\node[textnode] (covM) at (\w*2,  0) {$\covm$};
			\node[textnode] (d) at (\w*2,  \h) {$\dd$};
			\node[textnode] (cofM) at (\w*2,  \h*2) {$\cofm$};

			\node[textnode] (nonN) at (\w*3,  0) {$\nonn$};
			\node[textnode] (cofN) at (\w*3,  \h*2) {$\cofn$};
			
			\node[textnode] (aleph1) at (-\w,  0) {$\aleph_1$};
			\node[textnode] (c) at (\w*4,  \h*2) {$2^{\aleph_0}$};
			
			\draw[->, edge] (e) to (nonM);
			\draw[->, edge] (e) to (covM);
			\draw[->, edge] (addN) to (e);
			
			\draw[->, edge] (covM) to (pr);
			\draw[->, edge] (nonM) to (pr);
			\draw[->, edge] (pr) to (cofN);
			
			\draw[->, edge] (addN) to (covN);
			\draw[->, edge] (addN) to (addM);
			\draw[->, edge] (covN) to (nonM);	
			\draw[->, edge] (addM) to (b);
			\draw[->, edge] (b) to (nonM);
			\draw[->, edge] (addM) to (covM);
			\draw[->, edge] (nonM) to (cofM);
			\draw[->, edge] (covM) to (d);
			\draw[->, edge] (d) to (cofM);
			\draw[->, edge] (b) to (d);
			\draw[->, edge] (covM) to (nonN);
			\draw[->, edge] (cofM) to (cofN);
			\draw[->, edge] (nonN) to (cofN);
			\draw[->, edge] (aleph1) to (addN);
			\draw[->, edge] (cofN) to (c);

		\end{tikzpicture}
		\caption{Cicho\'n's diagram with evasion/prediction numbers.}\label{fig_Cd+e}
		
	\end{figure}

	\subsection{Main Results}

	We prove\footnote{Goldstern, Kellner, Mej\'{\i}a and Shelah \cite{GKMS_CM_and_evasion} stated that they proved the same separation result, but later found a gap in their proof. Moreover, our method is different from theirs and consequently the constellation of Figure \ref{fig_thmC} is a specific separation result of ours.\label{foot_MR}} that they are instances answering Question \ref{ques_2}.
	More concretely, we can add the prediction and evasion numbers to Cicho\'n's maximum with distinct values as in Figure \ref{fig_CM+e}:
	
	\begin{figure}
		\centering
		\begin{tikzpicture}
			\tikzset{
				textnode/.style={text=black}, 
			}
			\tikzset{
				edge/.style={color=black, thin, opacity=0.5}, 
			}
			\newcommand{\w}{2.4}
			\newcommand{\h}{1.8}
			
			\node[textnode] (addN) at (0,  0) {$\addn$};
			\node (t1) [fill=lime, draw, text=black, circle,inner sep=1.0pt] at (-0.25*\w, 0.4*\h) {$\theta_1$};
			
			\node[textnode] (covN) at (0,  2*\h) {$\covn$};
			\node (t2) [fill=lime, draw, text=black, circle,inner sep=1.0pt] at (0.15*\w, 1.5*\h) {$\theta_2$};

			\node[textnode] (addM) at (\w,  0) {$\cdot$};
			\node[textnode] (b) at (\w,  1*\h) {$\bb$};
			\node (t3) [fill=lime, draw, text=black, circle,inner sep=1.0pt] at (0.6*\w, 0.7*\h) {$\theta_3$};
			
			\node[textnode] (e) at (0.5*\w,  1.3*\h) {$\mathfrak{e}$}; 
			\node (t4) [fill=lime, draw, text=black, circle,inner sep=1.0pt] at (0.82*\w, 1.4*\h) {$\theta_4$};
			
			\node[textnode] (nonM) at (\w,  2*\h) {$\nonm$};
			\node (t5) [fill=lime, draw, text=black, circle,inner sep=1.0pt] at (1.2*\w, 2.5*\h) {$\theta_5$};
			
			\node[textnode] (covM) at (\w*2,  0) {$\covm$};
			\node (t6) [fill=lime, draw, text=black, circle,inner sep=1.0pt] at (1.8*\w, -0.5*\h) {$\theta_6$};
			
			\node[textnode] (pr) at (2.5*\w,  0.7*\h) {$\mathfrak{pr}$}; 
			\node (t7) [fill=lime, draw, text=black, circle,inner sep=1.0pt] at (2.18*\w, 0.6*\h) {$\theta_7$};
			
			\node[textnode] (d) at (\w*2,  1*\h) {$\dd$};
			\node (t8) [fill=lime, draw, text=black, circle,inner sep=1.0pt] at (2.4*\w, 1.3*\h) {$\theta_8$};
			
			\node[textnode] (cofM) at (\w*2,  2*\h) {$\cdot$};

			\node[textnode] (nonN) at (\w*3,  0) {$\nonn$};
			\node (t9) [fill=lime, draw, text=black, circle,inner sep=1.0pt] at (2.85*\w, 0.5*\h) {$\theta_9$};
			
			\node[textnode] (cofN) at (\w*3,  2*\h) {$\cofn$};
			\node (t10) [fill=lime, draw, text=black, circle,inner sep=1.0pt] at (3.25*\w, 1.6*\h) {$\theta_{10}$};
			
			\node[textnode] (aleph1) at (-\w,  0) {$\aleph_1$};
			\node[textnode] (c) at (\w*4,  2*\h) {$2^{\aleph_0}$};
			\node (tcc) [fill=lime, draw, text=black, circle,inner sep=1.0pt] at (3.7*\w, 2.5*\h) {$\theta_\cc$};

			\draw[->, edge] (addN) to (covN);
			\draw[->, edge] (addN) to (addM);
			\draw[->, edge] (covN) to (nonM);	
			\draw[->, edge] (addM) to (b);
			
			\draw[->, edge] (addM) to (covM);
			\draw[->, edge] (nonM) to (cofM);
			
			\draw[->, edge] (d) to (cofM);
			\draw[->, edge] (b) to (d);
			\draw[->, edge] (b) to (nonM);
			\draw[->, edge] (covM) to (nonN);
			\draw[->, edge] (covM) to (d);
			\draw[->, edge] (cofM) to (cofN);
			\draw[->, edge] (nonN) to (cofN);
			\draw[->, edge] (aleph1) to (addN);
			\draw[->, edge] (cofN) to (c);
			
			\draw[->, edge] (e) to (nonM);
			\draw[->, edge] (e) to (covM);
			\draw[->, edge] (addN) to (e);
			
			\draw[->, edge] (covM) to (pr);
			\draw[->, edge] (nonM) to (pr);
			\draw[->, edge] (pr) to (cofN);

			\draw[blue,thick] (-0.5*\w,1*\h)--(0.35*\w,1*\h);
			\draw[blue,thick] (2.65*\w,1*\h)--(3.5*\w,1*\h);

			\draw[blue,thick] (0.35*\w,1.65*\h)--(1.5*\w,1.65*\h);
			\draw[blue,thick] (1.5*\w,0.35*\h)--(2.65*\w,0.35*\h);
			
			\draw[blue,thick] (0.35*\w,1.15*\h)--(1.5*\w,1.15*\h);
			\draw[blue,thick] (1.5*\w,0.85*\h)--(2.65*\w,0.85*\h);

			\draw[blue,thick] (1.5*\w,-0.5*\h)--(1.5*\w,2.5*\h);
			%\draw[blue,thick] (0.5*\w,2.5*\h)--(0.5*\w,2.5*\h);
			\draw[blue,thick] (2.65*\w,-0.5*\h)--(2.65*\w,2.5*\h);

			\draw[blue,thick] (0.35*\w,-0.5*\h)--(0.35*\w,2.5*\h);
			%\draw[blue,thick] (2.5*\w,-0.5*\h)--(2.5*\w,0.5*\h);
			\draw[blue,thick] (-0.5*\w,-0.5*\h)--(-0.5*\w,2.5*\h);
			\draw[blue,thick] (3.5*\w,-0.5*\h)--(3.5*\w,2.5*\h);

		\end{tikzpicture}
		\caption{Cicho\'n's maximum with evasion/prediction numbers.}\label{fig_CM+e}
	\end{figure}
	
	\begin{thmA}(Theorem \ref{thm_p6fin})
		\label{thmA}
		The separation described in Figure \ref{fig_CM+e} consistently holds. %where $\aleph_1\leq\theta_1\leq\cdots\leq\theta_{10}$ are regular and $\theta_\cc\geq\theta_{10}$ satisfies $\theta_\cc^{\aleph_0}=\theta_\cc$.
	\end{thmA}
	
	%Let us review the history of Cicho\'n's maximum focusing on the methods.
	The construction of Cicho\'n's maximum consists of two steps: the first one is to separate the left side of the diagram with additional properties and the second one is to separate the right side (point-symmetrically) using these properties. In \cite{GKS}, the large cardinal assumption was used in the second step to apply \textit{Boolean Ultrapowers}. In \cite{GKMS}, they introduced the \textit{submodel} method instead, which is a general technique to separate the right side without large cardinals.
	
	Let us focus on the first step. The main work to separate the left side is \textit{to keep the bounding number $\bb$ small} through the forcing iteration, since the naive bookkeeping iteration to increase the cardinal invariants in the left side guarantees the smallness of the other numbers but not of $\bb$.
	To tackle the problem, in \cite{GKS} they used the \textit{ultrafilter-limit} %\footnote{While they did not use the name ``ultrafilter-limit'' but ``$\lim^D$'' for an ultrafilter $D$ on $\omega$, we use ``ultrafilter-limit'' in this paper.}
	method, which was first introduced by Goldstern, Mej\'{\i}a and Shelah \cite{GMS16} to separate the left side of the diagram. 
	%and in \cite{KST} they introduced the \textit{FAM-limit}\footnote{While they did not use the name ``FAM-limit'' but ``strong FAM-limit for intervals'', we use ``FAM-limit'' in this paper. ``FAM'' is short for ``finitely additive measure'' (on $\omega$). Also, the original idea of the notion is seen in \cite{She00}} method.

	We introduce a new limit method called \textit{closed-ultrafilter-limit} (Definition \ref{dfn_UF_linked}), which is a variant of ultrafilter-limit, and prove that it keeps $\ee$ small:
	\begin{thmB}(Main Lemma \ref{lem_Main_Lemma})
		Closed-ultrafilter-limits keep $\ee$ small.
	\end{thmB}
	We also prove that the two ultrafilter-limit methods can be mixed and obtain the separation model of Figure \ref{fig_CM+e} (Theorem A, Theorem \ref{thm_p6fin}).
	
	Moreover, we prove that we can control the values of the following variants of the evasion/prediction numbers:

	\begin{dfn}
		\label{dfn_variant}
		\begin{enumerate}
			\item A predictor $\pi$ \textit{bounding-predicts} $f\in\oo$ if $f(n)\leq\pi(f\on n)$ for all but finitely many $n\in D$.
			$\preb$ and $\eeb$ denote the prediction/evasion number respectively with respect to the bounding-prediction.
			\item Let $g\in\left(\omega+1\setminus2\right)^\omega$. ( ``$\setminus2$'' is required to exclude trivial cases.)
			%	\begin{itemize}
				%\item 
				\textit{$g$-prediction} is the prediction where the range of functions $f$ is restricted to $\prod_{n<\omega}g(n)$ and $\pre_g$ and $\ee_g$ denote the prediction/evasion number respectively with respect to the $g$-prediction.
				Namely,
				\begin{gather*}
					\mathfrak{pr}_g\coloneq\min\{|\Pi|:\Pi\subseteq \Pred,\forall f\in\textstyle{\prod_{n<\omega}g(n)}~\exists\pi\in \Pi~\pi \text{ predicts } f\},\\
					\mathfrak{e}_g\coloneq\min\{|F|:F\subseteq\textstyle{\prod_{n<\omega}g(n)},\forall \pi\in \Pred~\exists f\in F~ f\text{ evades }\pi\}.
				\end{gather*}
				Define:
				\begin{gather*}
					\mathfrak{pr}_{ubd}\coloneq\sup\left\{\mathfrak{pr}_g:g\in\left(\omega\setminus2\right)^\omega\right\},\\
					\ee_{ubd}\coloneq\min\left\{\ee_g:g\in\left(\omega\setminus2\right)^\omega\right\}.
				\end{gather*}

			\end{enumerate}
			
		\end{dfn}
		
		The new numbers can be embedded into the diagram as in Figure \ref{fig_Cd_all}. 
		We obtain the following separation result:
		\begin{thmC}(Theorem \ref{thm_p_fin_prime_GCH})
			\label{thmC}
			The separation described in Figure \ref{fig_thmC} consistently holds.
		\end{thmC}

		\begin{figure}
			\centering
			\begin{tikzpicture}
				\tikzset{
					textnode/.style={text=black}, 
				}
				\tikzset{
					edge/.style={color=black, thin}, 
				}
				\newcommand{\w}{2.4}
				\newcommand{\h}{2.0}
				
				\node[textnode] (addN) at (0,  0) {$\addn$};
				%\node (t1) [fill=lime, draw, text=black, circle,inner sep=1.0pt] at (-0.25*\w, 0.8*\h) {$\theta_1$};
				
				\node[textnode] (covN) at (0,  \h*3) {$\covn$};
				%\node (t2) [fill=lime, draw, text=black, circle,inner sep=1.0pt] at (0.15*\w, 3.3*\h) {$\theta_2$};

				\node[textnode] (addM) at (\w,  0) {$\cdot$};
				\node[textnode] (b) at (\w,  1.3*\h) {$\bb$};
				%\node (t3) [fill=lime, draw, text=black, circle,inner sep=1.0pt] at (0.68*\w, 1.2*\h) {$\theta_3$};
				
				\node[textnode] (nonM) at (\w,  \h*3) {$\nonm$};
				%\node (t7) [fill=lime, draw, text=black, circle,inner sep=1.0pt] at (1.65*\w, -0.3*\h) {$\theta_7$};

				\node[textnode] (covM) at (\w*2,  0) {$\covm$};
				
				\node[textnode] (d) at (\w*2,  1.7*\h) {$\dd$};
				%\node (t10) [fill=lime, draw, text=black, circle,inner sep=1.0pt] at (2.32*\w, 1.8*\h) {$\theta_{10}$};
				\node[textnode] (cofM) at (\w*2,  \h*3) {$\cdot$};

				\node[textnode] (nonN) at (\w*3,  0) {$\nonn$};
				%\node (t11) [fill=lime, draw, text=black, circle,inner sep=1.0pt] at (2.85*\w, -0.3*\h) {$\theta_{11}$};
				
				\node[textnode] (cofN) at (\w*3,  \h*3) {$\cofn$};
				%\node (t12) [fill=lime, draw, text=black, circle,inner sep=1.0pt] at (3.25*\w, 2.2*\h) {$\theta_{12}$};
				
				\node[textnode] (aleph1) at (-\w,  0) {$\aleph_1$};
				\node[textnode] (c) at (\w*4,  \h*3) {$2^{\aleph_0}$};
				%\node (t10) [fill=lime, draw, text=black, circle,inner sep=1.0pt] at (3.67*\w, 3.4*\h) {$\theta_\cc$};
				
				\node[textnode] (e) at (0.5*\w,  1.7*\h) {$\mathfrak{e}$}; 
				\node[textnode] (estar) at (\w,  1.7*\h) {$\ee^*$};
				%\node (t4) [fill=lime, draw, text=black, circle,inner sep=1.0pt] at (0.15*\w, 1.7*\h) {$\theta_4$};

				\node[textnode] (pr) at (2.5*\w,  1.3*\h) {$\mathfrak{pr}$}; 
				\node[textnode] (prstar) at (\w*2,  1.3*\h) {$\mathfrak{pr}^*$};
				%\node (t9) [fill=lime, draw, text=black, circle,inner sep=1.0pt] at (2.85*\w, 1.3*\h) {$\theta_9$};
				
				\node[textnode] (eubd) at (\w*0.75,  \h*2.45) {$\ee_{ubd}$};
				%\node[textnode] (nonE) at (\w*0.75,  \h*2.55) {$\non(\mathcal{E})$};
				%\node (t5) [fill=lime, draw, text=black, circle,inner sep=1.0pt] at (0.5*\w, 2.1*\h) {$\theta_5$};%0.5 2.1
				
				\node[textnode] (prubd) at (\w*2.25,  \h*0.55) {$\mathfrak{pr}_{ubd}$};
				%\node[textnode] (covE) at (\w*2.25,  \h*0.45) {$\cov(\mathcal{E})$};
				%\node (t8) [fill=lime, draw, text=black, circle,inner sep=1.0pt] at (2.47*\w, 0.9*\h) {$\theta_8$};

				\draw[->, edge] (addN) to (covN);
				\draw[->, edge] (addN) to (addM);
				\draw[->, edge] (covN) to (nonM);	
				\draw[->, edge] (addM) to (b);
				%\draw[->, edge] (b) to (nonM);
				\draw[->, edge] (addM) to (covM);
				\draw[->, edge] (nonM) to (cofM);
				%\draw[->, edge] (covM) to (d);
				\draw[->, edge] (d) to (cofM);
				\draw[->, edge] (b) to (prstar);
				\draw[->, edge] (covM) to (nonN);
				\draw[->, edge] (cofM) to (cofN);
				\draw[->, edge] (nonN) to (cofN);
				\draw[->, edge] (aleph1) to (addN);
				\draw[->, edge] (cofN) to (c);
				
				%\draw[->, edge] (e) to (nonM);
				\draw[->, edge] (e) to (covM);
				\draw[->, edge] (addN) to (e);
				
				\draw[->, edge] (covM) to (prstar);
				\draw[->, edge] (covM) to (prubd);
				\draw[->, edge] (covN) to (prubd);
				\draw[->, edge] (nonM) to (pr);
				\draw[->, edge] (pr) to (cofN);
				
				\draw[->, edge] (e) to (estar);
				\draw[->, edge] (b) to (estar);C
				\draw[->, edge] (estar) to (nonM);
				\draw[->, edge] (estar) to (d);
				\draw[->, edge] (e) to (eubd);
				
				\draw[->, edge] (prstar) to (d);
				\draw[->, edge] (prstar) to (pr);
				
				\draw[->, edge] (prstar) to (pr);
				\draw[->, edge] (prubd) to (pr);
				\draw[->, edge] (eubd) to (nonM);
				\draw[->, edge] (eubd) to (nonN);

			\end{tikzpicture}
			\caption{Cicho\'n's diagram with the variants of evasion/prediction numbers.}\label{fig_Cd_all}
		\end{figure}

		\begin{figure}
			\centering
			\begin{tikzpicture}
				\tikzset{
					textnode/.style={text=black}, 
				}
				\tikzset{
					edge/.style={color=black, thin, opacity=0.5}, 
				}
				\newcommand{\w}{2.4}
				\newcommand{\h}{2.0}
				
				\node[textnode] (addN) at (0,  0) {$\addn$};
				\node (t1) [fill=lime, draw, text=black, circle,inner sep=1.0pt] at (-0.2*\w, 0.8*\h) {$\theta_1$};
				
				\node[textnode] (covN) at (0,  \h*3) {$\covn$};
				\node (t2) [fill=lime, draw, text=black, circle,inner sep=1.0pt] at (0.15*\w, 2.45*\h) {$\theta_2$};

				\node[textnode] (addM) at (\w,  0) {$\cdot$};
				\node[textnode] (b) at (\w,  1.3*\h) {$\bb$};
				\node (t3) [fill=lime, draw, text=black, circle,inner sep=1.0pt] at (0.68*\w, 1.2*\h) {$\theta_3$};
				
				\node[textnode] (nonM) at (\w,  \h*3) {$\nonm$};
				\node (t5) [fill=lime, draw, text=black, circle,inner sep=1.0pt] at (1.0*\w, 3.3*\h) {$\theta_5$};
				
				\node[textnode] (covM) at (\w*2,  0) {$\covm$};
				\node (t6) [fill=lime, draw, text=black, circle,inner sep=1.0pt] at (2.0*\w, -0.3*\h) {$\theta_6$};
				
				\node[textnode] (d) at (\w*2,  1.7*\h) {$\dd$};
				\node (t8) [fill=lime, draw, text=black, circle,inner sep=1.0pt] at (2.32*\w, 1.8*\h) {$\theta_{8}$};
				\node[textnode] (cofM) at (\w*2,  \h*3) {$\cdot$};

				\node[textnode] (nonN) at (\w*3,  0) {$\nonn$};
				\node (t9) [fill=lime, draw, text=black, circle,inner sep=1.0pt] at (2.85*\w, 0.55*\h) {$\theta_{9}$};
				
				\node[textnode] (cofN) at (\w*3,  \h*3) {$\cofn$};
				\node (t10) [fill=lime, draw, text=black, circle,inner sep=1.0pt] at (3.25*\w, 2.2*\h) {$\theta_{10}$};
				
				\node[textnode] (aleph1) at (-\w,  0) {$\aleph_1$};
				\node[textnode] (c) at (\w*4,  \h*3) {$2^{\aleph_0}$};
				\node (t10) [fill=lime, draw, text=black, circle,inner sep=1.0pt] at (3.67*\w, 3.4*\h) {$\theta_\cc$};
				
				\node[textnode] (e) at (0.5*\w,  1.7*\h) {$\mathfrak{e}$}; 
				\node[textnode] (estar) at (\w,  1.7*\h) {$\ee^*$};
				\node (t4) [fill=lime, draw, text=black, circle,inner sep=1.0pt] at (0.75*\w, 1.9*\h) {$\theta_4$};

				\node[textnode] (pr) at (2.5*\w,  1.3*\h) {$\mathfrak{pr}$}; 
				\node[textnode] (prstar) at (\w*2,  1.3*\h) {$\mathfrak{pr}^*$};
				\node (t9) [fill=lime, draw, text=black, circle,inner sep=1.0pt] at (2.25*\w, 1.1*\h) {$\theta_7$};
				
				\node[textnode] (eubd) at (\w*0.75,  \h*2.45) {$\ee_{ubd}$};
				%\node[textnode] (nonE) at (\w*0.75,  \h*2.55) {$\non(\mathcal{E})$};
				%\node (t5) [fill=lime, draw, text=black, circle,inner sep=1.0pt] at (0.5*\w, 2.1*\h) {$\theta_5$};%0.5 2.1
				
				\node[textnode] (prubd) at (\w*2.25,  \h*0.55) {$\mathfrak{pr}_{ubd}$};
				%\node[textnode] (covE) at (\w*2.25,  \h*0.45) {$\cov(\mathcal{E})$};
				%\node (t8) [fill=lime, draw, text=black, circle,inner sep=1.0pt] at (2.47*\w, 0.9*\h) {$\theta_8$};

				\draw[->, edge] (addN) to (covN);
				\draw[->, edge] (addN) to (addM);
				\draw[->, edge] (covN) to (nonM);	
				\draw[->, edge] (addM) to (b);
				%\draw[->, edge] (b) to (nonM);
				\draw[->, edge] (addM) to (covM);
				\draw[->, edge] (nonM) to (cofM);
				%\draw[->, edge] (covM) to (d);
				\draw[->, edge] (d) to (cofM);
				\draw[->, edge] (b) to (prstar);
				\draw[->, edge] (covM) to (nonN);
				\draw[->, edge] (cofM) to (cofN);
				\draw[->, edge] (nonN) to (cofN);
				\draw[->, edge] (aleph1) to (addN);
				\draw[->, edge] (cofN) to (c);
				
				%\draw[->, edge] (e) to (nonM);
				\draw[->, edge] (e) to (covM);
				\draw[->, edge] (addN) to (e);
				
				\draw[->, edge] (covM) to (prstar);
				\draw[->, edge] (covM) to (prubd);
				\draw[->, edge] (nonM) to (pr);
				\draw[->, edge] (covN) to (prubd);
				\draw[->, edge] (pr) to (cofN);
				
				\draw[->, edge] (e) to (estar);
				\draw[->, edge] (b) to (estar);C
				\draw[->, edge] (estar) to (nonM);
				\draw[->, edge] (eubd) to (nonN);
				\draw[->, edge] (estar) to (d);
				\draw[->, edge] (e) to (eubd);
				
				\draw[->, edge] (prstar) to (d);
				\draw[->, edge] (prstar) to (pr);
				
				\draw[->, edge] (prstar) to (pr);
				\draw[->, edge] (prubd) to (pr);
				\draw[->, edge] (eubd) to (nonM);
				%		\
				%		
				%		\draw[->, edge] (eubd) to (nonE);
				%		\draw[->, edge] (addM) to (nonE);
				%		\draw[->, edge] (nonE) to (nonM);
				%		
				%		\draw[->, edge] (covE) to (prubd);
				%		\draw[->, edge] (covE) to (cofM);
				%		\draw[->, edge] (covM) to (covE);
				
				\draw[blue,thick] (-0.5*\w,1.5*\h)--(3.5*\w,1.5*\h);
				\draw[blue,thick] (1.5*\w,-0.5*\h)--(1.5*\w,3.5*\h);
				
				\draw[blue,thick] (-0.5*\w,-0.5*\h)--(-0.5*\w,3.5*\h);
				\draw[blue,thick] (3.5*\w,-0.5*\h)--(3.5*\w,3.5*\h);
				
				\draw[blue,thick] (0.5*\w,-0.5*\h)--(0.5*\w,1.5*\h);
				\draw[blue,thick] (2.5*\w,1.5*\h)--(2.5*\w,3.5*\h);
				
				\draw[blue,thick] (0.3*\w,2.25*\h)--(1.5*\w,2.25*\h);
				%\draw[blue,thick] (-0.3*\w,2.7*\h)--(1.5*\w,2.7*\h);
				\draw[blue,thick] (0.3*\w,1.5*\h)--(0.3*\w,2.25*\h);
				\draw[blue,thick] (0.5*\w,2.25*\h)--(0.5*\w,3.5*\h);
				
				\draw[blue,thick] (2.7*\w,0.8*\h)--(1.5*\w,0.8*\h);
				%\draw[blue,thick] (3.1*\w,0.3*\h)--(1.5*\w,0.3*\h);
				\draw[blue,thick] (2.7*\w,1.5*\h)--(2.7*\w,0.8*\h);
				\draw[blue,thick] (2.5*\w,0.8*\h)--(2.5*\w,-0.5*\h);

			\end{tikzpicture}
			\caption{Separation constellation of Theorem C.}\label{fig_thmC}
		\end{figure}

		\subsection{Structure of the paper}
		In Section \ref{sec_RS_PT}, we review the relational systems, the Tukey order and the general preservation theory of fsi (finite support iteration), such as \textit{goodness}. In Section \ref{sec_uflimit}, we present the notion of \textit{ultrafilter-limit}, which was first introduced in \cite{GMS16}. 
		Also, we introduce the new notion \textit{closed-ultrafilter-limit} 
		and prove Theorem B, which is specific for this new limit notion.
		In Section \ref{sec_separation}, we present the application of Theorem B and prove the separation results Theorem A and Theorem C.
		Finally, we conclude the paper leaving some open questions presented in Section \ref{sec_question}.
		
		\section{Relational systems and preservation theory} \label{sec_RS_PT}
		\begin{dfn}
			
			\begin{itemize}
				\item $\R=\langle X,Y,\sqsubset\rangle$ is a relational system if $X$ and $Y$ are non-empty sets and $\sqsubset \subseteq X\times Y$.
				\item We call an element of $X$ a \textit{challenge}, an element of $Y$ a \textit{response}, and ``$x\sqsubset y$''  ``$x$ is \textit{met by }$y$''.
				\item $F\subseteq X$ is $\R$-unbounded if no response meets all challenges in $F$.
				\item $F\subseteq Y$ is $\R$-dominating if every challenge is met by some response in $F$.
				\item $\R$ is non-trivial if $X$ is $\R$-unbounded and $Y$ is $\R$-dominating. For non-trivial $\R$, define
				\begin{itemize}
					\item $\bb(\R)\coloneq\min\{|F|:F\subseteq X \text{ is }\R\text{-unbounded}\}$, and
					\item $\dd(\R)\coloneq\min\{|F|:F\subseteq Y \text{ is }\R\text{-dominating}\}$.
				\end{itemize}

			\end{itemize}
			
		\end{dfn}
		In this section, we assume $\R$ is non-trivial.
		
		\begin{dfnfac}
			\label{dfnfac_RS}
			\begin{enumerate}
				\item For $\mathbf{D}\coloneq\langle \oo,\oo,\leq^*\rangle$, we get $\bb(\mathbf{D})=\bb, \dd(\mathbf{D})=\dd$.
				\item %Let $Pred$ denote the set of all predictors.
				Define $\mathbf{PR}\coloneq\langle\oo, \Pred,\sqsubset^\mathrm{p}\rangle$, where $f\sqsubset^\mathrm{p}\pi:\Leftrightarrow f$ is predicted by $\pi$.
				Also, define $\mathbf{BPR}\coloneq\langle\oo, \Pred,\sqsubset^\mathrm{bp}\rangle$, where $f\sqsubset^\mathrm{bp}\pi:\Leftrightarrow f$ is bounding-predicted by $\pi$ and $\mathbf{PR}_g\coloneq\langle\prod_{n<\omega}g(n), \Pred,\sqsubset^\mathrm{p}\rangle$  where $g\in(\omega+1\setminus2)^\omega$. We have $\bb(\mathbf{PR})=\ee, \dd(\mathbf{PR})=\pre$, $\bb(\mathbf{BPR})=\eeb, \dd(\mathbf{BPR})=\preb$, $\bb(\mathbf{PR}_g)=\ee_g, \dd(\mathbf{PR}_g)=\pre_g$.
				\item\label{item_dfnfac_RS_ideal}
				For an ideal $I$ on $X$, define two relational systems $\bar{I}\coloneq\langle I,I,\subseteq\rangle$ and $C_I\coloneq\langle X,I,\in\rangle$.
				We have $\bb(\bar{I})=\addi,\dd(\bar{I})=\cofi$ and $\bb(C_I)=\noni,~\dd(C_I)=\covi$.
				If $I$ is an ideal, then we will write $\R\lq I$ to mean $\R\lq\bar{I}$; and analogously for $\succeq_T$ and $\cong_T$.
			\end{enumerate}
			
		\end{dfnfac}
		
		\begin{dfn}
			$\R^\bot$ denotes the dual of $\R=\langle X,Y,\sqsubset\rangle$,
			i.e.,
			$\R^\bot\coloneq\langle Y,X,\sqsubset^\bot\rangle$ where $y\sqsubset^\bot x:\Leftrightarrow \lnot(x\sqsubset y)$.
		\end{dfn}

		\begin{dfn}
			
			For relational systems $\R=\langle X,Y,\sqsubset \rangle, \R^{\prime}=\langle X^{\prime},Y^{\prime},\sqsubset^{\prime}~\rangle$,
			$(\Phi_-,\Phi_+):\R\rightarrow\R^\prime$ is a Tukey connection from $\R$ into $\R^{\prime}$ if $\Phi_-:X\rightarrow X^{\prime}$ and $\Phi_+:Y^{\prime}\rightarrow Y$ are functions such that:
			\begin{equation*}
				\forall x\in X~\forall y^{\prime}\in Y^{\prime}~\Phi_-(x)\sqsubset^{\prime} y^{\prime}\Rightarrow x \sqsubset \Phi_{+} (y^{\prime}).
			\end{equation*}

			We write $\R\preceq_T\R^{\prime}$ if there is a Tukey connection from $\R$ into $\R^{\prime}$ and call $\preceq_T$ the Tukey order.
			Tukey equivalence $\R\cong_T\R^{\prime}$ is defined as: $\R\preceq_T\R^{\prime}$ and $\R^{\prime}\preceq_T\R$. 
			
		\end{dfn}
		
		\begin{fac}
			\label{Tukey order and b and d}
			\begin{enumerate}
				\item $\R\preceq_T\R^{\prime}$ implies $(\R^{\prime})^\bot\preceq_T\R^\bot$.
				\item $\R\preceq_T\R^{\prime}$ implies $\mathfrak{b}(\R^{\prime})\leq\mathfrak{b}(\R)$ and $\mathfrak{d}(\R)\leq\mathfrak{d}(\R^{\prime})$.
				\item $\bb(\R^\bot)=\dd(\R)$ and $\dd(\R^\bot)=\bb(\R^\bot)$.
			\end{enumerate}
		\end{fac}
		
		In the rest of this section, we fix an uncountable regular cardinal $\theta$ and a set $A$ of size $\geq\theta$. $[A]^{<\theta}$ is an ideal on $A$, so $C_{[A]^{<\theta}}$ is a relational system as in Definition and Fact \ref{dfnfac_RS}\eqref{item_dfnfac_RS_ideal} and $\bb(C_{[A]^{<\theta}})=\theta$ and $\dd(C_{[A]^{<\theta}})=|A|$. 
		For a relational system $\R$, we can calculate $\bb(\R)$ and $\dd(\R)$ from ``outside'' and ``inside'', using this  $C_{[A]^{<\theta}}$:
		\begin{cor}
			\textcolor{white}{a}
		%	Let $\theta$ be regular uncountable and $A$ a set of size $\geq\theta$.
			\begin{itemize}
				\item[(outside)] If $\R\lq C_{[A]^{<\theta}}$, then $\theta\leq\bb(\R)$ and $\dd(\R)\leq|A|$.
				\item[(inside)] If $C_{[A]^{<\theta}}\lq \R$, then $\bb(\R)\leq\theta$ and $|A|\leq\dd(\R)$.
			\end{itemize}
		\end{cor}
		
		Both ``$\R\lq C_{[A]^{<\theta}}$'' and ``$C_{[A]^{<\theta}}\lq \R$'' have the following characterizations:
		\begin{fac}(\cite[Lemma 1.16.]{forcing_constellations})
			\label{fac_Tukey_order_equivalence_condition}
			Assume $|X|\geq\theta$ where $\R=\langle X,Y,\sqsubset\rangle$.
			%Let $\theta$ be an infinite cardinal and $\R=\langle X,Y,\sqsubset\rangle$ a relational system.
			\begin{enumerate}
				\item $\R\lq C_{[X]^{<\theta}}$ iff $\bb(\R)\geq\theta$. %any subset of $X$ of size $<\theta$ is $\R$-bounded (i.e., not $\R$-unbounded).
				\item \label{item_small_equiv}
				$C_{[A]^{<\theta}}\lq\R$ iff there exists $\langle x_a:a\in A\rangle$ such that every $y\in Y$ meets only $<\theta$-many $x_a$.
			\end{enumerate}
		\end{fac}
		
		To separate the right side by using submodels after having separated the left side,
		``$\R\cong_T C_{[A]^{<\theta}}$'' does not work, but ``$\R\cong_T[A]^{<\theta}$'' does.
		The following fact gives a sufficient condition which implies $C_{[A]^{<\theta}}\cong_T [A]^{<\theta}$:
		\begin{fac}(\cite[Lemma 1.15.]{forcing_constellations})
			\label{fac_suff_eq_CI_and_I}
			%If $\theta$ is regular and $A$ is a set with 
			If $|A|^{<\theta}=|A|$,
			then $C_{[A]^{<\theta}}\cong_T [A]^{<\theta}$.
		\end{fac}

		\begin{fac}(\cite[Lemma 2.11.]{forcing_constellations})
			\label{fac_cap_V}
			%Let $\theta$ be uncountable regular and $A$ be a set of size $\geq\theta$.
			Every ccc poset forces $[A]^{<\theta}\cong_T[A]^{<\theta}\cap V$ and $C_{[A]^{<\theta}}\cong_T C_{[A]^{<\theta}}\cap V$.
			Moreover, $\mathfrak{x}([A]^{<\theta})=\mathfrak{x}^V([A]^{<\theta})$ where $\mathfrak{x}$ represents ``$\add$'', ``$\cov$'', ``$\non$'' or ``$\cof$''.
		\end{fac}
		
		%\begin{dfn}
		%	$\R=\langle X,Y,\sqsubset\rangle$ is a definable relational system of the reals if $X$ and $Y$ are analytic in Polish spaces $Z$ and $W$ respectively and $\sqsubset$ is analytic in $Z\times W$. When dealing with a definable relational system of the reals $\R=\langle X,Y,\sqsubset\rangle$, we interpret it depending on the model we are working in.
		%\end{dfn}

		When performing a forcing iteration, the ``outside'' direction is easily satisfied by bookkeeping, while the other one, ``inside'' direction needs more discussion and actually it is usually the main work of separating cardinal invariants. 
		
		In the context of separating cardinal invariants of the continuum by finite support iteration (fsi) of ccc forcings,
		the notions of ``Polish relational system'' and ``good'' (introduced in \cite{JS90} and \cite{Bre91}) work well.

		\begin{dfn}
			$\R=\langle X,Y,\sqsubset\rangle$ is a Polish relational system (Prs) if:
			\begin{enumerate}
				\item $X$ is a perfect Polish space.
				\item  $Y$ is analytic in a Polish space $Z$.
				\item $\sqsubset =\bigcup_{n<\omega}\sqsubset_n$ where $\langle \sqsubset_n:n<\omega\rangle$ is an ($\subseteq$-)increasing sequence of closed subsets of $X\times Z$ such that for any $n<\omega$ and any $y\in Y$, $\{x\in X:x\sqsubset_n y\}$ is closed nowhere dense.
			\end{enumerate}
			When dealing with a Prs, we interpret it depending on the model we are working in.
			%Note that a Prs is a definable relational system of the reals.
			%If $|\Omega|=1$, we call $\R$ a Polish relational system (Prs).
		\end{dfn}

		In the rest of this section, $\R=\langle X,Y,\sqsubset\rangle$ denotes a Prs. 
		%and $\theta$ denotes an infinite cardinal.
		\begin{dfn}
			A poset $\p$ is $\theta$-$\R$-good if for any $\p$-name $\dot{y}$ for a member of $Y$, there is a non-empty set $Y_0\subseteq Y$ of size $<\theta$ such that for any $x\in X$, if $x$ is not met by any $y\in Y_0$, then $\p$ forces $x$ is not met by $\dot{y}$.
			If $\theta=\aleph_1$, we say ``$\R$-good'' instead of ``$\aleph_1$-$\R$-good''.
		\end{dfn}
		
		The following two facts show that goodness works well for the ``inside'' direction of fsi of ccc forcings:
		\begin{fac}(\cite{JS90}, \cite[Corollary 4.10.]{BCMseparating})
			Any fsi of ccc $\theta$-$\R$-good posets is again $\theta$-$\R$-good.
		\end{fac}
		
		\begin{fac}(\cite{FM21},\cite[Theorem 4.11.]{BCMseparating})
			Let $\p$ be a fsi of non-trivial ccc $\theta$-$\R$-good posets of length $\gamma\geq\theta$. 
			Then, $\p$ forces $C_{[\gamma]^{<\theta}}\lq\R$.
		\end{fac}
		
		An example of a good poset is a small one:
		\begin{fac}
			%\cite[Lemma 2.16.]{forcing_constellations}
			(\cite[Lemma 4]{Mej13},
			\cite[Theorem 6.4.7]{BJ95})
			Every poset of size $<\theta$ is $\theta$-$\R$-good.
			In particular, Cohen forcing is $\R$-good.
		\end{fac}
		
		To treat goodness, we have to characterize cardinal invariants using a Prs.
		While $\mathbf{D},\mathbf{PR},\mathbf{BPR}$ and $\mathbf{PR}_g$ are canonically Prs's, the cardinal invariants on ideals need other characterizations:
		%\footnote{Also note that a goodness is not necessarily preserved through Tukey equivalence, i.e., even if $\p$ is $\theta$-$\R$-good and $\R\cong_T\R^\prime$,it does not $\p$ is not always $\theta$-$\R^\prime$-good.}.
		
		\begin{exa}
			\begin{enumerate}
				\item For $k<\omega$, let $\mathrm{id}^k\in\oo$ denote the function $i\mapsto i^k$ for each $i<\omega$ and let $\mathcal{H}\coloneq\{\mathrm{id}^{k+1}:k<\omega\}$.
				
				Let $\mathcal{S}=\mathcal{S}(\omega,\mathcal{H})$ be the set of all functions $\varphi\colon\omega\to\fin$ such that there is $h\in \mathcal{H}$ 
				with $|\varphi(i)|\leq h(i)$ for all $i<\omega$.
				Let $\mathbf{Lc}^*=\langle\oo,\mathcal{S},\in^*\rangle$ be the Prs where $x\in^*\varphi:\Leftrightarrow x(n)\in\varphi(n)$ for all but finitely many $n<\omega$.
				
				As a consequence of \cite{BartInv},
				$\mathbf{Lc}^*\cong_T\n$ holds, so $\bb(\mathbf{Lc}^*)=\addn$ and $\dd(\mathbf{Lc}^*)=\cofn$.
				
				Any $\mu$-centered poset is $\mu^+$-$\mathbf{Lc}^*$-good (\cite{Bre91,JS90}).
				
				Any Boolean algebra with a strictly positive finitely additive measure is $\mathbf{Lc}^*$-good (\cite{Kam89}). In particular, so is any subalgebra of random forcing.
				
				\item For each $n<\omega$,
				let $\Omega_n\coloneq\{a\in[\sq]^{<\omega}:\mathbf{Lb}_2(\bigcup_{s\in a}[s])\leq2^{-n}\}$ (endowed with the discrete topology) where $\mathbf{Lb}_2$ is the standard Lebesgue measure on $2^\omega$.
				Put $\Omega\coloneq\prod_{n<\omega}\Omega_n$ with the product topology, which is a perfect Polish space.
				For $x\in\Omega$, let $N_x^*\coloneq\bigcap_{n<\omega}\bigcup_{s\in x(n)}[s]$, a Borel null set in $2^\omega$.
				Define the Prs $\mathbf{Cn}\coloneq\langle\Omega,2^\omega,\sqsubset^{\mathbf{Cn}}\rangle$ where $x\sqsubset^\mathbf{Cn}z:\Leftrightarrow z\notin N^*_x$.
				Since $\langle N^*_x:x\in\Omega\rangle$ is cofinal in $\n(2^\omega)$ (the set of all null sets in $2^\omega$), $\mathbf{Cn}\cong_T C_\n^\bot$ holds, so $\bb(\mathbf{Cn})=\covn$ and $\dd(\mathbf{Cn})=\nonn$.
				
				Any $\mu$-centered poset is $\mu^+$-$\mathbf{Cn}$-good (\cite{Bre91}).
				
				\item Let $\Xi\coloneq\{f\in(\sq)^{\sq}:\forall s\in\sq, s\subseteq f(s)\}$ and define the Prs $\mathbf{Mg}\coloneq\langle2^\omega,\Xi,\in^\bullet\rangle$ where $x\in^\bullet f:\Leftrightarrow|\{s\in\sq:x\supseteq f(s)\}|<\omega$. Note that $\mathbf{Mg}\cong_T C_\m$ and hence $\bb(\mathbf{Mg})=\nonm$ and $\dd(\mathbf{Mg})=\covm$.

			\end{enumerate}
		\end{exa}
		Summarizing the properties of the ``inside'' direction and the goodness, 
		we obtain the following corollary, which will be actually applied to the iteration in Section \ref{sec_separation}:
		\begin{cor}
			\label{cor_smallness_for_addn_and_covn_and_nonm}
			%Let $\theta$ be regular uncountable and 
			Let $\p$ be a fsi of ccc forcings of length $\gamma\geq\theta$.
			
			\begin{enumerate}
				\item Assume that each iterand is either:
				\begin{itemize}
					\item of size $<\theta$,
					\item a subalgebra of random forcing, or
					\item $\sigma$-centered.
				\end{itemize}
				Then, $\p$ forces $C_{[\gamma]^{<\theta}}\lq\mathbf{Lc}^*$,
				in particular, $\addn\leq\theta$.
				\item  Assume that each iterand is either:
				\begin{itemize}
					\item of size $<\theta$, or
					\item $\sigma$-centered.
				\end{itemize}
				Then, $\p$ forces $C_{[\gamma]^{<\theta}}\lq\mathbf{Cn}$,
				in particular, $\covn\leq\theta$.
				\item Assume that each iterand is:
				\begin{itemize}
					\item of size $<\theta$.
				\end{itemize}
				Then, $\p$ forces $C_{[\gamma]^{<\theta}}\lq\mathbf{Mg}$,
				in particular, $\nonm\leq\theta$.
			\end{enumerate}

		\end{cor}

		\begin{rem}
			In \cite{GKS}, the ``outside'' direction is treated by introducing ``$\cob$'' (short for cone of bounds). For a directed partially ordered set $(S,\leq_S)$, $\cob(\R,\p,S)$ stands for ``there exists a sequence $\langle\dot{y}_s:s\in S\rangle$ of $\p$-names of responses such that for every $\p$-name $\dot{x}$ of a challenge, there exists $s\in S$ such that for any $t\geq_S s$, $\Vdash_\p\dot{x} \sqsubset \dot{y}_t$''.
			
			If $\theta\leq\lambda$ and $\p$ is ccc, 
			then $\cob(\R,\p,[\lambda]^{<\theta})$ is equivalent to $\Vdash_\p \R\lq[\lambda]^{<\theta}\cap V\cong_T[\lambda]^{<\theta}$ (\cite[Remark 2.13.]{forcing_constellations}).
			Moreover, if $\lambda^{<\theta}=\lambda$, 
			it is also equivalent to $\Vdash_\p \R\lq C_{[\lambda]^{<\theta}\cap V}\cong_TC_{[\lambda]^{<\theta}}$ by Fact \ref{fac_suff_eq_CI_and_I}.
			
		\end{rem}
		
		\begin{rem}
			In \cite{GKS}, the ``inside'' direction is treated by introducing ``$\lcu$'' (short for linearly cofinally unbounded). For a limit ordinal $\gamma$,  $\lcu(\R,\p,\gamma)$ stands for ``there exists a sequence $\langle\dot{x}_\a:\a<\gamma\rangle$ of $\p$-names of challenges such that for every $\p$-name $\dot{y}$ of a response, there exists $\a<\gamma$ such that for any $\b\geq\a$, $\Vdash_\p\lnot(\dot{x}_\b \sqsubset \dot{y})$''(hence $\lcu(\R,\p,\gamma)\Leftrightarrow\cob(\R^\bot,\p,(\gamma,\leq))$).
			If $\gamma=\lambda$ is a regular cardinal of size $\geq\theta$ and $\p$ is ccc, then $\lcu(\R,\p,\gamma)$ is equivalent to $\Vdash_\p  C_{[\lambda]^{<\theta}\cap V}\cong_TC_{[\lambda]^{<\theta}}\lq\R$.
			Moreover, if $\lambda^{<\theta}=\lambda$, 
			it is also equivalent to $\Vdash_\p [\lambda]^{<\theta}\cap V\cong_T [\lambda]^{<\theta}\lq\R$ by Fact \ref{fac_suff_eq_CI_and_I}.
		\end{rem}

		\section{ultrafilter limit and closedness} \label{sec_uflimit}
		\subsection{General Theory}
		We basically follow the presentation of \cite{Mej_Two_FAM} to describe the general theory of (closed-)ultrafilter-limits. Also, the original ideas are already in \cite{GMS16}.
		\begin{dfn}(\cite[Section 5]{Mej19})
			Let $\Gamma$ be a class for subsets of posets,
			i.e., $\Gamma\in\prod_{\p}\mathcal{P}(\mathcal{P}(\p))$, a (class) function. (e.g., $\Gamma=\Lambda({\text{centered}})\coloneq$ ``centered'' is an example of a class for subsets of poset and in this case $\Gamma(\p)$ denotes the set of all centered subsets of $\p$ for each poset $\p$.)
			\begin{itemize}

				\item A poset $\p$ is $\mu$-$\Gamma$-covered if $\p$ is a union of $\leq\mu$-many subsets in $\Gamma(\p)$.
				%A poset $\p$ is $\mu$-$\Gamma$-linked if $\p$ has a dense subposet which is a union of $\leq\mu$-many subsets in $\Gamma(\p)$. 
				%i.e., $\p=\bigcup\{\bigcup \mathcal{Q}:\mathcal{Q}\in[\Gamma(\p)]^{\leq\mu}\}$.
				As usual, when $\mu=\aleph_0$, we use ``$\sigma$-$\Gamma$-covered'' instead of ``$\aleph_0$-$\Gamma$-covered''. Moreover, we often just say ``$\mu$-$\Gamma$'' instead of ``$\mu$-$\Gamma$-covered''.
				\item Abusing notation, we write ``$\Gamma\subseteq\Gamma^\prime$'' if $\Gamma(\p)\subseteq\Gamma^\prime(\p)$ holds for every poset $\p$.
			\end{itemize}
		\end{dfn}
		
		%\begin{dfn}
		%	Let $\Lambda_{\text{centered}}\coloneq$ ``centered''.
		%\end{dfn}

		In this paper, an ``ultrafilter'' means a non-principal ultrafilter.
		\begin{dfn}
			\label{dfn_UF_linked}
			%Let $\mathcal{UF}$ be the set of all ultrafilters on $\omega$ and let 
			
			Let $D$ be an ultrafilter on $\omega$ and $\p$ be a poset.
			\begin{enumerate}
				\item $Q\subseteq \p$ is $D$-lim-linked ($\in\Lambda^\mathrm{lim}_D(\p)$) if there exist a function $\lim^D\colon Q^\omega\to\p$ and a $\p$-name $\dot{D}^\prime$ of an ultrafilter extending $D$ such that for any countable sequence 
				$\bar{q}=\langle q_m:m<\omega\rangle\in Q^\omega$, 
				\begin{equation}
					\textstyle{\lim^D\bar{q}} \Vdash \{m<\omega:q_m \in \dot{G}\}\in \dot{D}^\prime.
				\end{equation}

				Moreover, if $\ran(\lim^D)\subseteq Q$, we say $Q$ is c-$D$-lim-linked (closed-$D$-lim-linked, $\in\Lambda^\mathrm{lim}_{\mathrm{c}D}(\p)$).
				\item $Q$ is (c-)uf-lim-linked (short for (closed-)ultrafilter-limit-linked) if $Q$ is (c-)$D$-lim-linked for every ultrafilter $D$.
				\item $\Lambda^\mathrm{lim}_\mathrm{uf}\coloneq\bigcap_D\Lambda^\mathrm{lim}_D$ and 
				$\Lambda^\mathrm{lim}_\mathrm{cuf}\coloneq\bigcap_D\Lambda^\mathrm{lim}_{\mathrm{c}D}$.
				%\item $\Lambda^\mathrm{lim}_D(\p)$

				%\item For cardinal $\mu$, $\p$ is $\mu$-(c-)uf-lim-linked if there exists a dense subposet which is a $\mu$-many union of (c-)uf-lim-linked subsets.
				%As usual, when $\mu=\aleph_0$, we write $\sigma$-(c-)uf-lim-linked instead of $\aleph_0$-(c-)uf-lim-linked.
				
			\end{enumerate}

			We often say ``$\p$ has (c-)uf-limits'' instead of ``$\p$ is $\sigma$-(c-)uf-lim-linked''.

		\end{dfn}

		\begin{exa}
			\label{exa_size_linked}
			Singletons are c-uf-lim-linked and hence every poset $\p$ is $|\p|$-c-uf-lim-linked. %(The limit condition is the common condition and hence all conditions in the sequence are in the generic filter). 
			%hence Cohen forcing is $\sigma$-c-uf-lim-linked.
		\end{exa}

		%\begin{lem}
		%	Let $D$ be an ultrafilter.
		%	\begin{enumerate}
			%		\item A $\mu$-c-$D$-lim-linked poset is $\mu$-$D$-lim-linked.
			%		\item A $\mu$-$D$-lim-linked poset is $\mu$-Fr-linked.
			%	\end{enumerate}
		%\end{lem}	

		%\subsection{General iteration theory of ultrafilter-limit}
		To define ``$\langle(\p_\xi,\qd_\xi):\xi<\gamma\rangle$ is a fsi of $\mu$-$\Gamma$-covered forcings ($\mu^+$-$\Gamma$-iteration, below)'' in a general way, we have the covering of each iterand $\qd_\xi$ witnessed by some complete subposet $\p^-_\xi$ of $\p_\xi$, not necessarily by $\p_\xi$.
		\begin{dfn}
			\label{dfn_Gamma_iteration}
			\begin{itemize}
				\item A $\kappa$-$\Gamma$-iteration is a fsi $\langle(\p_\eta,\qd_\xi):\eta\leq\gamma,\xi<\gamma\rangle $ of ccc forcings, with witnesses $\langle\p_\xi^-:\xi<\gamma\rangle$, $\langle\theta_\xi:\xi<\gamma\rangle$ and $\langle\dot{Q}_{\xi,\zeta}:\zeta<\theta_\xi,\xi<\gamma\rangle$ satisfying for all $\xi<\gamma$:
				\begin{enumerate}
					\item $\p^-_\xi\lessdot\p_\xi$.
					\item $\theta_\xi$ is a cardinal of size $<\kappa$.
					%\item $\qd_\xi$ is 
					\item \label{item_Q_is_P_minus_name}
					$\qd_\xi$ and $\langle\dot{Q}_{\xi,\zeta}:\zeta<\theta_\xi\rangle$ are $\p^-_\xi$-names and $\p^-_\xi$ forces that 
					$\bigcup_{\zeta<\theta_\xi}\dot{Q}_{\xi,\zeta}=\qd_\xi$ and $\dot{Q}_{\xi,\zeta}\in\Gamma(\qd_\xi)$ for each $\zeta<\theta_\xi$.
					%$\langle\dot{Q}_{\xi,\zeta}:\zeta<\theta_\xi\rangle$ witnesses $\qd_\xi$ is $\theta_\xi$-$\Gamma$-linked.
				\end{enumerate}
				\item $\xi<\gamma$ is a trivial stage if $\Vdash_{\p^-_\xi}|\dot{Q}_{\xi,\zeta}|=1$ for all $\zeta<\theta_\xi$. $S^-$ is the set of all trivial stages and $S^+\coloneq\gamma\setminus S^-$.
				\item A guardrail for the iteration is a function $h\in\prod_{\xi<\gamma}\theta_\xi$.
				\item $H\subseteq\prod_{\xi<\gamma}\theta_\xi$ is complete if any countable partial function in $\prod_{\xi<\gamma}\theta_\xi$ is extended to some (total) function in $H$.
				\item $\p^h_\eta$ is the set of conditions $p\in\p_\eta$ following $h$, i.e., for each $\xi\in\dom(p)$, $p(\xi)$ is a $\p^-_\xi$-name 
				and $\Vdash_{\p^-_\xi}p(\xi)\in \dot{Q}_{\xi,h(\xi)}$.  	
			\end{itemize}
		\end{dfn}
		
		The notion ``$p$ follows $h$'' only depends on the values of $h$ on $\dom(p)$:
		\begin{fac}
			\label{fac_guarrail_folloing_only_domain}
			Let $p\in\p^h_\eta$ and assume that a guardrail $g$ satisfies $g\on\dom(p)=h\on\dom(p)$.
			Then, $p\in \p^g_\eta$.
		\end{fac}
		
		%	\begin{proof}
			%		Direct from the definitions.
			%	\end{proof}
		
		If every finite partial guardrail can be extended to some $h\in H$, in particular if $H$ is complete, then there are densely many conditions which follow some $h\in H$:
		\begin{lem}
			If every finite partial guardrail can be extended to some $h\in H$, then $\bigcup_{h\in H}\p^h_\eta$ is dense in $\p_\eta$ for all $\eta\leq\gamma$.
		\end{lem}
		\begin{proof}
			Induction on $\eta$.
		\end{proof}
		
		The following theorem and corollary give a sufficient cardinal arithmetic to have a complete set of guardrails of small size:
		\begin{thm}(\cite{EK65})
			Let $\theta\leq\mu\leq\chi$ be infinite cardinals with $\chi\leq2^\mu$.
			Then, there is $F\subseteq {^{\chi}\mu}$ of size $\leq\mu^{<\theta}$ such that any partial function $\chi\to\mu$ of size $<\theta$ can be extended 
			to some (total) function in $F$.
		\end{thm}
		
		\begin{cor}
			\label{cor_complete}
			Assume $\aleph_1\leq\mu\leq|\gamma|\leq2^\mu$ and $\mu^+=\kappa$. 
			Then, for any $\langle\theta_\xi<\kappa:\xi<\gamma\rangle$, there exists a complete set of guardrails  of size $\leq\mu^{\aleph_0}$ which works for each $\kappa$-$\Gamma$-iteration of length $\gamma$ using $\langle\theta_\xi:\xi<\gamma\rangle$.
		\end{cor}
		
		%We look at special case for $\Gamma=\Lambda^\mathrm{lim}_\mathrm{uf}$ or $\mathrm{lim}_\mathrm{uf}$.
		In this section, let $\Gamma_\mathrm{uf}$ represent $\Lambda^\mathrm{lim}_\mathrm{uf}$ or $\Lambda^\mathrm{lim}_\mathrm{cuf}$.
		\begin{dfn}
			\label{dfn_UF_iteration}
			%Let $\Gamma\subseteq\Lambda^\mathrm{lim}_\mathrm{uf}$.
			A $\kappa$-$\Gamma_\mathrm{uf}$-iteration has $\Gamma_\mathrm{uf}$-limits on $H$ if
			\begin{enumerate}
				\item $H\subseteq\prod_{\xi<\gamma}\theta_\xi$ is a set of guardrails.
				\item For $h\in H$, $\langle \dot{D}^h_\xi:\xi\leq\gamma \rangle$ is a sequence such that $\dot{D}^h_\xi$ is a $\p_\xi$-name of a non-principal ultrafilter on $\omega$. 
				\item If $\xi<\eta\leq\gamma$, then $\Vdash_{\p_\eta}\dot{D}^h_\xi\subseteq\dot{D}^h_\eta$.
				
				\item \label{item_D^-}
				For $ \xi\in S^+$, $\Vdash_{\p_\xi} (\dot{D}^h_\xi)^-\in V^{\p^-_\xi}$ where $(\dot{D}^h_\xi)^-\coloneq\dot{D}^h_\xi\cap V^{\p^-_\xi}$.%  if , otherwise let $(\dot{D}^h_\xi)^-$ be an arbitrary ultrafilter in $V^{\p^-_\xi}$ (hence this item is trivially satisfied in this case).
				\item Whenever $\langle \xi_m:m<\omega \rangle\subseteq \gamma$ and $\bar{q}=\langle \dot{q}_m:m<\omega\rangle$ satisfying 
				
				$\Vdash_{\p^-_{\xi_m}}\dot{q}_m\in\dot{Q}_{\xi_m,h(\xi_m)}$ for each $m<\omega$:
				\begin{enumerate}
					\item If $\langle \xi_m:m<\omega \rangle$ is constant with value $\xi$, then
					\begin{equation}
						\label{eq_constant}
						\Vdash_{\p_\xi}\textstyle{\lim^{(\dot{D}^h_\xi)^-}}\bar{q}\Vdash_{\qd_\xi}\{m<\omega:\dot{q}_m\in \dot{H}_\xi\}
						\in\dot{D}^h_{\xi+1}.
					\end{equation}
					($\dot{H}_\xi$ denotes the canonical name of $\qd_\xi$-generic filter over $V^{\p_\xi}$.)
					\item If $\langle \xi_m:m<\omega \rangle$ is strictly increasing, then
					\begin{equation}
						\label{eq_increasing}
						\Vdash_{\p_\gamma}\{m<\omega:\dot{q}_m\in \dot{G}_\gamma\}\in\dot{D}^h_\gamma.
					\end{equation}
				\end{enumerate}
				
			\end{enumerate}
		\end{dfn}
		Justification for \eqref{eq_constant} is as follows:
		
		We have:
		\begin{itemize}
			\item $(\dot{D}^h_\xi)^-$ is an ultrafilter in $ V^{\p^-_\xi}$ by \eqref{item_D^-}.
			\item $\Vdash_{\p^-_\xi}\dot{Q}_{\xi,h(\xi)}\in\Lambda^\mathrm{lim}_\mathrm{uf}(\qd_\xi)$ by \eqref{item_Q_is_P_minus_name} in Definition \ref{dfn_Gamma_iteration}.
			
			\item $\Vdash_{\p^-_\xi}\dot{q}_m\in\dot{Q}_{\xi,h(\xi)}$ for all $m<\omega$.
		\end{itemize}
		Thus, we can consider `` $\lim^{(\dot{D}^h_\xi)^-}\bar{q}$ '' in $V^{\p^-_\xi}$ and hence in $V^{\p_\xi}$.
		Moreover, abusing notation, we also use ``$\lim^{(\dot{D}^h_\xi)^-}\bar{q}$'' for a trivial stage $\xi\in S^-$ to denote the constant value of $\bar{q}$. Even though $(\dot{D}^h_\xi)^-$ is not defined for $\xi\in S^-$, $\lim^{(\dot{D}^h_\xi)^-}\bar{q}$ works like an ultrafilter-limit since it trivially forces $\{m<\omega:\dot{q}_m\in \dot{H}_\xi\}=\omega$, and particularly \eqref{eq_constant} is satisfied as long as $\dot{D}^h_{\xi+1}$ is an ultrafilter.  
		
		Justification for \eqref{eq_increasing} is that in the standard way we identify $\dot{q}_m$ with a condition $p$ in $\p_\gamma$ defined by $\dom(p)\coloneq\{\xi_m\}$ and $p(\xi_m)\coloneq\dot{q}_m$, so \eqref{eq_increasing} is a valid statement.
		
	%	Thus, Definition {dfn_UF_iteration} makes sense.
		
		It seems to be possible to extend the iteration at a successor step by the direct use of the definition of uf-lim-linkedness in Definition \ref{dfn_UF_linked} (actually the purpose of the notion is to realize this successor step),
		but actually such a simple direct use does not work:
		
		Recall that we are in a slightly complicated situation where there are two models, $V^{\p_\gamma}$ and $V^{\p^-_\gamma}$, and two ultrafilters, $\dot{D}^h_\gamma\in V^{\p_\gamma}$ and $(\dot{D}^h_\gamma)^-\in V^{\p^-_\gamma}$.
		Hence, the definition of uf-lim-linkedness in Definition \ref{dfn_UF_linked} only helps to extend $(\dot{D}^h_\gamma)^-$, not $\dot{D}^h_\gamma$, since the statement ``$\dot{Q}_{\gamma,\zeta}\in\Gamma(\qd_\gamma)$ (for each $\zeta<\theta_\gamma$)''
		%``$\langle\dot{Q}_{\xi,\zeta}:\zeta<\theta_\xi\rangle$ witnesses $\qd_\xi$ is $\theta_\xi$-$\Lambda^\mathrm{lim}_\mathrm{uf}$-linked'' 
		holds in $ V^{\p^-_\gamma}$, not in $V^{\p_\gamma}$.
		
		Thus, we need the following lemma which helps to amalgamate ultrafilters:

		\begin{lem}(\cite[Lemma 3.20.]{BCM_filter_linkedness})
			\label{lem_UF_upward_directed}
			Let 
			$M\subseteq N$ be transitive models of set theory, $\p\in M$ be a poset, $D_0\in M,D_0^\prime\in N$ be ultrafilters and $\dot{D}_1\in M^\p$ be a name of an ultrafilter.
			If $D_0\subseteq D_0^\prime$ and $\Vdash_{M,\p}D_0\subseteq \dot{D}_1$,
			then there exists $\dot{D}_1^\prime\in N^\p$, a name of an ultrafilter such that $\Vdash_{N,\p}D_0^\prime,\dot{D}_1\subseteq\dot{D}_1^\prime$.
			(Here, we write $\Vdash_{M,\p}\varphi$ for $M\vDash(\Vdash_\p\varphi)$).
		\end{lem}
		
		\begin{proof}
			It is enough to show that $\Vdash_{N,\p}\text{``}D_0^\prime\cup\dot{D}_1$ has SFIP''.
			(SFIP is short for Strong Finite Intersection Property and means ``every finite subset has infinite intersection''.) 
			We show that for any $A\in D_0^\prime$ and any $\Vdash_{M,\p}\dot{B}\in\dot{D}_1$, $\Vdash_{N,\p}A\cap\dot{B}\neq\emptyset$.
			Let $p\in\p$ be arbitrary and $B^\prime\coloneq\{b<\omega:\exists q\leq p,~q\Vdash_{M,\p}b\in\dot{B}\}$.
			Since $p\Vdash_{M,\p}\dot{B}\subseteq B^\prime$ and $\Vdash_{M,\p}D_0\subseteq \dot{D}_1$, we obtain $B^\prime\in D_0$.
			Hence in $N$, we can find $c\in A\cap B^\prime$.
			Let $q\leq p$ be a witness of $c\in B^\prime$.
			Note that an $N$-generic filter $G$ is trivially $M$-generic as well, so $q\Vdash_{M,\p}c\in\dot{B}$ implies $q\Vdash_{N,\p}c\in\dot{B}$.
			Thus, $q\Vdash_{N,\p}c\in A\cap\dot{B}$ and since $p$ is arbitrary, we have $\Vdash_{N,\p}A\cap\dot{B}\neq\emptyset$.
		\end{proof}

		\begin{lem}
			\label{lem_uf_const_succ}
			Let $\p_{\gamma+1}$ be a $\kappa$-$\Gamma_\mathrm{uf}$-iteration (of length $\gamma+1$) and suppose $\p_\gamma=\p_{\gamma+1}\on\gamma$ has $\Gamma_\mathrm{uf}$-limits on $H$.
			If $\gamma\in S^-$, or if $\gamma\in S^+$ and:
			\begin{equation}
				\label{eq_minus}
				\Vdash_{\p_\gamma} (\dot{D}^h_\gamma)^-\in V^{\p^-_\gamma}\text{ for all }h\in H,
			\end{equation}
			then we can find $\{\dot{D}^h_{\gamma+1}:h\in H\}$ witnessing that $\p_{\gamma+1}$ has $\Gamma_\mathrm{uf}$-limits on $H$.
		\end{lem}

		\begin{proof}
			If $\gamma\in S^-$, %(where \eqref{eq_minus} is trivially satisfied),
			any $\dot{D}^h_{\gamma+1}$ extending $\dot{D}^h_\gamma$ for $h\in H$ satisfies \eqref{eq_constant} since every ultrafilter contains $\omega$.
			Thus, we may assume $\gamma\in S^+$.
			By Definition \ref{dfn_UF_linked}, for each $h\in H$ we can find a $\p^-_\gamma*\qd_\gamma$-name $\dot{D}^\prime$ of an ultrafilter extending $(\dot{D}^h_\gamma)^-$ such that 
			for any $\bar{q}=\langle \dot{q}_m:m<\omega\rangle$ satisfying $\Vdash_{\p^-_\gamma}\dot{q}_m\in\dot{Q}_{\gamma,h(\gamma)}$ for all $m<\omega$:
			\begin{equation}
				\label{eq_ufconst_succ}
				\Vdash_{\p^-_\gamma}\textstyle{\lim^{(\dot{D}^h_\gamma)^-}}\bar{q}\Vdash_{\dot{Q}_\gamma}\{m<\omega:\dot{q}_m\in \dot{H}_\gamma\}
				\in\dot{D}^\prime.
			\end{equation}
			Since $(\dot{D}^h_\gamma)^-$ is extended to $\dot{D}^h_\gamma$ and $\dot{D}^\prime$, we can find a  $\p_\gamma*\qd_\gamma=\p_{\gamma+1}$-name $\dot{D}^h_{\gamma+1}$ of an ultrafilter extending $\dot{D}^h_\gamma$ and $\dot{D}^\prime$ by Lemma \ref{lem_UF_upward_directed}.
			%Since $(\dot{D}^h_{\gamma+1})^-=\dot{D}^\prime\in V^{\p^-_\gamma}$, 
			This $\dot{D}^h_{\gamma+1}$ satisfies \eqref{eq_constant} and we are done.
		\end{proof}
		
		We give a sufficient condition satisfying the assumption \eqref{eq_minus}: %(e.g. $\p_\gamma^-=\p_\gamma$ is a trivial case).
		
		\begin{lem}
			\label{lem_D_minus_in_suff}
			Let $\p$ be a ccc poset, $\dot{D}$ a $\p$-name of a set of reals, $\Theta$ a sufficiently large regular cardinal and $N\preccurlyeq H_\Theta$ a $\sigma$-closed submodel containing $\dot{D}$, i.e., $N^\omega\cup\{\dot{D}\}\subseteq N$.
			Then, $\p^-\coloneq\p\cap N$ is a complete subposet of $\p$ and $\Vdash_\p\dot{D}\cap V^{\p^-}\in V^{\p^-}$. 
		\end{lem}
		\begin{proof}
			Since $\p$ is ccc and $N$ is $\sigma$-closed, $N$ contains all maximal antichains in $\p^-$ and hence $\p^-\lessdot\p$ by elementarity.
			Moreover, we may identify a (nice) $\p^-$-name of a real and a (nice) $\p$-name of a real in $N$.
			Define a $\p^-$-name $\tau$ by $(\sigma,p)\in\tau:\Leftrightarrow\sigma$ is a nice $\p^-$-name of a real and $p\in\p^-$ satisfies $p\Vdash_\p\sigma\in \dot{D}$.
			We obtain $\Vdash_\p\tau=\dot{D}\cap V^{\p^-}$ and we are done. 
		\end{proof}
		Thus, if we are under the assumption of Lemma \ref{lem_uf_const_succ} without \eqref{eq_minus}, and additionally if $N\supseteq N^\omega\cup\{\dot{D}^h_\gamma:h\in H\}$ and $\p^-_\gamma=\p_\gamma\cap N$, then \eqref{eq_minus} is satisfied by Lemma \ref{lem_D_minus_in_suff}.

		For the limit step of the construction of the ultrafilters $\dot{D}^h_\gamma$, we use centeredness:
		
		\begin{lem}
			\label{lem_uf_const_all}
			Let $\gamma$ be limit and 
			$\p_\gamma$ be a $\kappa$-$\left(\Lambda(\mathrm{centered})\cap\Gamma_\mathrm{uf}\right) $-iteration.
			If $\langle \dot{D}^h_\xi:\xi<\gamma, h\in H\rangle$ witnesses that for any $\xi<\gamma$, $\p_\xi=\p_\gamma\on\xi$ has $\Gamma_\mathrm{uf}$-limits on $H$,  
			then we can find $\langle\dot{D}^h_\gamma:h\in H\rangle$ such that $\langle \dot{D}^h_\xi:\xi\leq\gamma, h\in H\rangle$ witnesses $\p_\gamma$ has $\Gamma_\mathrm{uf}$-limits on $H$.
		\end{lem}

		\begin{proof}
			%By Lemma \ref{lem_uf_const_succ}, we may assume that $\gamma$ is limit.
			For $\cf(\gamma)>\omega$ there is nothing to do, so we assume $\cf(\gamma)=\omega$.
			Let $h\in H$ be arbitrary and $S$ be the collection of $\bar{q}=\langle \dot{q}_m:m<\omega\rangle$ such that for some increasing $\langle \xi_m<\gamma:m<\omega \rangle$, $\Vdash_{\p^-_{\xi_m}}\dot{q}_m\in\dot{Q}_{\xi_m,h(\xi_m)}$ holds for each $m<\omega$ (Note that $\xi_m\to\gamma$ since $\xi_m$ are increasing and $\cf(\gamma)=\omega$).
			For $\bar{q}\in S$, let $\dot{A}(\bar{q})\coloneq\{m<\omega:\dot{q}_m\in\dot{G}_\gamma\}$.
			%By excluding triviality, we may assume that we are in the case $\cf(\gamma)=\omega$ (hence $\xi_m\to\gamma$) and 
			We will show:
			\begin{equation}
				\Vdash_{\p_\gamma}\text{``}\bigcup_{\xi<\gamma}\dot{D}^h_\xi\cup\{\dot{A}(\bar{q}):\bar{q}\in S\}\text{ has SFIP}\text{''}.
			\end{equation} 
			If not, there exist $p\in\p_\gamma$, $\xi<\gamma$, $\p_\xi$-name $\dot{A}$ of an element of $\dot{D}^h_\xi$, $\{\bar{q}^i=\langle \dot{q}^i_m:m<\omega\rangle:i<n\}\in[S]^{<\omega}$ and increasing ordinals $\langle \xi_m^i<\gamma:m<\omega \rangle$ for $i<n$ such that $\Vdash_{\p^-_{\xi^i_m}}\dot{q}_m^i\in\dot{Q}_{\xi^i_m,h(\xi^i_m)}$ holds for $m<\omega$ and $i<n$ and the following holds:
			\begin{equation}
				\label{eq_limit_SFIP_contra}
				p\Vdash_{\p_\gamma}\dot{A}\cap \bigcap_{i<n} \dot{A}(\bar{q}^i)=\emptyset.
			\end{equation}
			We may assume that $p\in\p_\xi$.
			Since all $\langle \xi^i_m<\gamma:m<\omega \rangle$ are increasing and converge to $\gamma$, there is $m_0<\omega$ such that $\xi_m^i>\xi$ for any $m>m_0$ and $i<n$.
			By Induction Hypothesis, $p\Vdash_{\p_\xi}\text{``}\dot{D}^h_\xi$ is  an ultrafilter'' and hence we can pick $q\leq_{\p_\xi} p$ and $m>m_0$ 
			such that $q\Vdash_{\p_\xi}m\in\dot{A}$.
			Let us reorder $\{\xi_m^i:i<n\}=\{\xi^0<\cdots<\xi^{l-1}\}$.
			Inducting on $j<l$, we construct $q_j\in\p_{\xi_j}$. % and a $\p_\xi$-name $\dot{q}_j$.
			Let $q_{-1}\coloneq q$ and $j<l$ and assume we have constructed $q_{j-1}$.
			Let $I_j\coloneq\{i<n:\xi_m^i=\xi^j\}$.
			Since $\p_{\xi^j}$ forces that all $\dot{q}^i_m$ for $i\in I_j$ are in the same centered component $\dot{Q}_{\xi^j,h(\xi^j)}$, we can pick $p_j\leq q_{j-1}$ in $\p_{\xi^j}$ and 
			a $\p_{\xi^j}$-name $\dot{q}_j$ of a condition in $\qd_{\xi^j}$ such that for each $i\in I_j$, 
			$p_j\Vdash_{\p_{\xi^j}}\dot{q}_j\leq\dot{q}^i_m$.
			Let $q_j\coloneq p_j^\frown\dot{q}_j$.
			By construction, $q^\prime\coloneq q_{l-1}$ satisfies $q^\prime\leq q\leq p$ and $q^\prime\on\xi_m^i\Vdash_{\p_{\xi_m}} q^\prime(\xi_m^i)\leq \dot{q}^i_m$ for all $i<n$, so in particular,  $q^\prime\Vdash_{\p_\gamma}m\in\dot{A}\cap \bigcap_{i<n} \dot{A}(\bar{q}^i)$, which contradicts \eqref{eq_limit_SFIP_contra}.
		\end{proof}

		%\begin{cor}
		%	\label{cor_uf_name_const}
		%	Let $\p_\gamma$ be a $\kappa$-$\left(\Lambda(\mathrm{centered})\cap\Gamma_\mathrm{uf}\right) $-iteration.
		%	By constructing $\{\dot{D}^h_\xi:\xi\leq\gamma,~h\in H\}$ (inductively), we can make $\p_\gamma$ have $\Gamma_\mathrm{uf}$-limit on H.
		%\end{cor}
		%
		%\begin{proof}
		%	Direct from Lemma \ref{lem_uf_const_succ} and Lemma \ref{lem_uf_const_limit}.
		%\end{proof}
		\subsection{Uniform $\Delta$-system}
		From now on, we always assume $\p_\gamma$ is a $\kappa$-$\Gamma_\mathrm{uf}$-iteration with $\Gamma_\mathrm{uf}$-limits on $H$.
		To define ultrafilter-limits for “refined”
		sequences of conditions in $\p_\gamma$, we introduce the notion of \textit{uniform $\Delta$-system}, which is a more refined $\Delta$-system of conditions in $\p^h_\gamma$ (see \cite[Definition 4.3.19.]{Uri}):
		\begin{dfn}
			\label{dfn_uniform_delta_system}
			Let $\delta$ be an ordinal, 
			$h\in H$ and 
			$\bar{p}=\langle p_ m: m<\delta\rangle\in(\p_\gamma^h)^\delta$.
			$\bar{p}$ is an $h$-uniform $\Delta$-system if:
			\begin{enumerate}
				%\item There is $h\in H$ such that $p_ m\in\p_\gamma^h$ for all $ m<\delta$.
				\item $\{\dom(p_ m): m<\delta\}$ forms a $\Delta$-system with some root $\nabla$.
				\item All $|\dom(p_ m)|$ are the same $n^\prime$ and $\dom(p_ m)=\{\a_{n, m}:n<n^\prime\}$ is the increasing enumeration.
				\item There is $r^\prime\subseteq n^\prime$ such that $n\in r^\prime\Leftrightarrow\a_{n, m}\in\nabla$ for $n<n^\prime$.
				\item For $n\in n^\prime\setminus r^\prime$, $\langle\a_{n, m}: m<\delta \rangle$ is (strictly) increasing.
				
			\end{enumerate}
		\end{dfn}
		
		$\Delta$-System Lemma for this uniform $\Delta$-system also holds:
		\begin{lem}(\cite[Theorem 4.3.20]{Uri})
			\label{lem_uniform_Delta_System_Lemma}
			Assume that $\theta>|H|$ is regular and $\{p_m:m<\theta\}\subseteq\bigcup_{h\in H}\p^h_\gamma$.
			Then, there exist $I\in[\theta]^\theta$ and $h\in H$ such that $\{p_m:m\in I\}$ forms an $h$-uniform $\Delta$-system.
		\end{lem}
		
		\begin{proof}
			Almost direct from $\Delta$-System Lemma. (see e.g. \cite[Theorem 4.3.20]{Uri}).
		\end{proof}
		
		\begin{dfn}
			\label{dfn_uflimit_condition_for_iteration}
			%For $h\in H$ and $\bar{p}\in(\p^h_\gamma)^\omega$ whose domains form a $\Delta$-system with some root $\nabla$, we define $p^\infty=\lim^h\bar{p}\in\p_\gamma$ as follows:
			Let $\bar{p}=\langle p_m:m<\omega\rangle\in(\p^h_\gamma)^\omega$ be an $h$-uniform (countable) $\Delta$-system with root $\nabla$.
			We define the limit condition $p^\infty=\lim^h\bar{p}\in\p_\gamma$ as follows:
			\begin{enumerate}
				\item $\dom(p^\infty)\coloneq\nabla$.
				\item For $\xi\in\nabla$,  %$\Vdash_{\p^-_\xi}p^\infty(\xi):=\lim^{(\dot{D}^h_\xi)^-}\overline{ p(\xi)}$, where $\overline{ p(\xi)}\coloneq\langle p_m(\xi):m<\omega\rangle$.
				$\Vdash_{\p^-_\xi}p^\infty(\xi):=\lim^{(\dot{D}^h_\xi)^-}\langle p_m(\xi):m<\omega\rangle$.
			\end{enumerate}
		\end{dfn}

		The ultrafilter limit condition forces that ultrafilter many conditions are in the generic filter:
		\begin{lem}
			\label{lem_UF-limit_Principle}
			$\lim^h\bar{p}\Vdash_{\p_\gamma}\{m<\omega:p_m\in \dot{G}_\gamma\}\in\dot{D}^h_\gamma$.
		\end{lem}
		
		\begin{proof}
			Induct on $\gamma$.
			
			$\textit{Successor step.}$ Let $\bar{p}=\langle(p_m,\dot{q}_m):m<\omega\rangle\in(\p^h_{\gamma+1})^\omega$ be an $h$-uniform $\Delta$-system with root $\nabla$.	
			To avoid triviality, we may assume that $\gamma\in\nabla$.
			Also we may assume that $\Vdash_{\p^-_\gamma}\dot{q}_m\in\dot{Q}_{\gamma,h(\gamma)}$ for each $m<\omega$.
			Let $p^\infty\coloneq\lim^h\langle p_m:m<\omega\rangle\in\p^-_\gamma$ and 
			$\Vdash_{\p^-_\gamma}\dot{q}^\infty\coloneq\lim^{(\dot{D}^h_\gamma)^-}\langle \dot{q}_m:m<\omega\rangle$.
			By Induction Hypothesis, $p^\infty\Vdash_{\p_\gamma}\dot{A}\coloneq\{m<\omega:p_m\in \dot{G}\}\in\dot{D}^h_\gamma$.
			By \eqref{eq_constant}, $\Vdash_{\p_\gamma}\dot{q}^\infty\Vdash_{\dot{Q}_\gamma}\dot{B}\coloneq\{m<\omega:\dot{q}_m\in \dot{H}_\gamma\}\in\dot{D}^h_{\gamma+1}$
			Thus,
			$\lim^h\bar{p}=p^\infty*\dot{q}^\infty\Vdash_{\p_{\gamma+1}}\{m<\omega:(p_m,\dot{q}_m)\in \dot{G}_\gamma*\dot{H}_\gamma\}\supseteq\dot{A}\cap\dot{B}\in \dot{D}^h_{\gamma+1}$.
			
			$\textit{Limit step.}$ Let $\gamma$ be limit and $\bar{p}=\langle p_m:m<\omega\rangle\in(\p^h_\gamma)^\omega$ be an $h$ uniform-$\Delta$-system.
			We use the same parameters as in Definition \ref{dfn_uniform_delta_system}.
			Let $\xi\coloneq\max(\nabla)+1<\gamma$ and $\bar{p}\on\xi\coloneq\langle p_m\on\xi:m<\omega\rangle$.
			Since %all $\xi_m$ are out of $\nabla$, 
			$\bar{p}\on\xi$ is also an $h$-uniform $\Delta$-system with root $\nabla$,
			$p^\infty\coloneq\lim^h\bar{p}\on\xi=\lim^h\bar{p}$.
			By Induction Hypothesis, $p^\infty\Vdash_{\p_\xi}\dot{A}\coloneq\{m<\omega:p_m\on\xi\in \dot{G}_\xi\}\in\dot{D}^h_\xi$.
			Let $n^{\prime\prime}\coloneq\max{r^\prime}+1$.
			Since $\langle \a_{n,m}:m<\omega \rangle$ is increasing for $n\in\left[n^{\prime\prime},n^\prime\right)$,
			by \eqref{eq_increasing}, 
			$\Vdash_{\p_\gamma}\dot{B}_n\coloneq\{m<\omega:p_m(\a_{n,m})\in \dot{G}_\gamma\}\in\dot{D}^h_\gamma$ for $n\in\left[n^{\prime\prime},n^\prime\right)$.
			Since $\dom(p_m)=\dom(p\on\xi)\cup\{\a_{n,m}:n\in\left[n^{\prime\prime},n^\prime\right)\}$,
			$p^\infty\Vdash_{\p_\gamma}\{m<\omega:p_m\in \dot{G}_\gamma\}=\dot{A}\cap\bigcap_{n\in\left[n^{\prime\prime},n^\prime\right)} \dot{B}_n\in \dot{D}^h_\gamma$.
		\end{proof}
		
		\begin{cor}
			\label{cor_UF-limit_Preservation_Property}
			Let $\bar{p}=\langle p_m :m<\omega\rangle $, $h\in H$ and $p^\infty=\lim^h\bar{p}$ as above
			and let $\varphi$ be a formula of the forcing language without parameter $m$.
			If all $p_m$ force $\varphi$, then $p^\infty$ also forces $\varphi$.
		\end{cor}
		\begin{proof}
			Let $G$ be any generic filter containing $p^\infty$.
			By Lemma \ref{lem_UF-limit_Principle}, in particular, there exists $p_m\in G$.
			Since $p_m$ forces $\varphi$, $V[G]\vDash\varphi$ and recall that $G$ is arbitrary containing $p^\infty$.
		\end{proof}
		
		The following lemma is specific for $\Lambda^\mathrm{lim}_\mathrm{cuf}$ and actually this is why we consider the notion of closedness:
		
		\begin{lem}
			\label{lem_cuf_guardrail_closedness}
			Consider the case $\Gamma_\mathrm{uf}=\Lambda^\mathrm{lim}_\mathrm{cuf}$.
			If $\bar{p}\in(\p^h_\gamma)^\omega$, then $\lim^h\bar{p}\in\p^h_\gamma$.
		\end{lem}
		\begin{proof}
			Direct from the definition of ``closed-uf-lim-linked'' in Definition \ref{dfn_UF_linked} and Definition \ref{dfn_uflimit_condition_for_iteration}.
		\end{proof}

		\subsection{Application to bounding-prediction}
		To control cardinal invariants using a $\kappa$-$\Gamma_\mathrm{uf}$-iteration $\p_\gamma$ with $\Gamma_\mathrm{uf}$-limits, it is useful to iterate Cohen forcings in the first half of the iteration. For this purpose, we assume the following in this subsection:
		\begin{ass}
			\label{ass_first_Cohen}
			\begin{enumerate}
				\item $\kappa<\lambda$ are uncountable regular cardinals and $\gamma=\lambda+\lambda$.
				\item $\p_\gamma$ is a $\kappa$-$\Gamma_\mathrm{uf}$-iteration with $\Gamma_\mathrm{uf}$-limits on $H$ (with the same parameters as in Definition \ref{dfn_Gamma_iteration} and \ref{dfn_UF_iteration}).
				\item $H$ is complete and $|H|<\kappa$.
				\item For $\xi<\lambda$, $\Vdash_{\p^-_\xi}\qd_\xi=\mathbb{C}$, the Cohen forcing. Note that $\mathbb{C}$ is $\kappa$-$\Gamma_\mathrm{uf}$-linked by Example \ref{exa_size_linked}.
			\end{enumerate}
		\end{ass}

		As mentioned above, in \cite{GMS16} they introduced the notion of ultrafilter-limits to keep the bounding number $\bb$ small through the iteration and separated the left side of Cicho\'n's diagram. This is described as follows using the notions we have already defined above:
		
		%when the first Cicho\'n's maximum was constructed in \cite{GKS}, they used the notion of ultrafilter-limits to keep the bounding number $\bb$ small through the iteration.
		
		\begin{thm}(\cite[Main Lemma 4.6]{GMS16}, \cite[Lemma 1.31.]{GKS})
			\label{thm_uf_limit_keeps_b_small}
			Consider the case $\Gamma_\mathrm{uf}=\Lambda^\mathrm{lim}_\mathrm{uf}$.
%			(and repeatedly, suppose Assumption \ref{ass_first_Cohen})
%			Assume:
%			\begin{itemize}
%				\item $\kappa$ is uncountable regular and $\kappa<\gamma$.
%				\item $H$ is complete and has size $<\kappa$.
%			\end{itemize} 
			Then, $\p_\gamma$ forces $C_{[\lambda]^{<\kappa}}\lq \mathbf{D}$, in particular, $\bb\leq\kappa$. %(Note that since a set of conditions following a common guardrail is centered, $\p_\gamma$ is $\kappa$-cc and hence preserves all cardinals $\geq\kappa$.) %In practical applications, we usually iterate ccc forcings and hence all cardinals and cofinalities are preseved.)
		\end{thm}
		\begin{rem}
			Note that in Theorem \ref{thm_uf_limit_keeps_b_small} and Main Lemma \ref{lem_Main_Lemma} below, we do not require centeredness. The justification is that while we need centeredness when constructing a $\kappa$-$\Gamma_\mathrm{uf}$-iteration $\p_\gamma$ with $\Gamma_\mathrm{uf}$-limits at limit steps (see Lemma \ref{lem_uf_const_all}), after having constructed the iteration we do not need centeredness anymore in the argument itself of controlling cardinal invariants. The reason why we state the theorem and the main lemma in the current way is related to Question \ref{que_last} in Section \ref{sec_question}, in order to clarify where the problem discussed in the question lies.
		\end{rem}
		
		We shall carry out a similar argument for closed-ultrafilter-limits to keep $\eeb$ small, using the lemmas in the previous subsection, which are specific for this new limit notion.
		First, we introduce some notation on the bounding-prediction.
		\begin{dfn}
			%Let $Pred$ be the set of all predictors.
			For a predictor $\pi=(A,\langle \pi_k:k\in A \rangle)\in \Pred$ and $f\in\oo$,
			we write $f\bpd_n \pi$ if $f(k)\leq\pi_k(f\on k)$ for all $k\geq n$ in $A$.
			Note that $\bpd=\bigcup_{n<\omega}\bpd_n $ (see Definition and Fact \ref{dfnfac_RS}) and we say $n$ is a starting point of $f\bpd\pi$ if $f\bpd_n \pi$ holds (we do not require the minimality of such $n$).
			%Let $\mathbf{BPred}$ be the relatinal system $\langle \oo, Pred, \bpd \rangle$, thus $\bb(\mathbf{BPred})=\eeb$ and $\dd(\mathbf{BPR})=\preb$.
			
		\end{dfn}

		By applying the general theory of (c-)uf-limit to bounding-prediction, we can exclude a possible prediction point and preserve the information of the initial segment of predicted reals:
		\begin{lem}
			\label{lem_Prediction_Point_Exclusion_Lemma}
			Let $\dot{\pi}=(\dot{A},\langle\dot{\pi}_k:k\in\dot{A}\rangle)$ be a $\p_\gamma$-name of a predictor,
			$n^*\leq j<\omega$,
			$t\in\omega^{j-1}$ and
			$\bar{p}=\langle p_m:m<\omega\rangle$ be an $h$-uniform $\Delta$-system.
			
			Assume that each $p_m$ forces:	  
			\begin{equation}
				\label{eq_PPEL_assump}
				\text{there is }c\in\oo\text{such that }c\on j=t^\frown m\text{ and }c\bpd_{n^*}\dot{\pi}.
			\end{equation}  
			Then, $\lim^h\bar{p}$ forces:
			\begin{gather}
				j-1\notin\dot{A}\text{, and}\label{eq_PPEL_conclu_1}\\
				\text{there is }c\in\oo\text{ such that }c\on (j-1)=t\text{ and }c\bpd_{n^*}\dot{\pi}.\label{eq_PPEL_conclu_2}
			\end{gather}	 
		\end{lem}
		\begin{proof}
			Let $G$ be any generic filter containing $p^\infty\coloneq\lim^h\bar{p}$ and work in $V[G]$.
			By Lemma \ref{lem_UF-limit_Principle}, $M\coloneq\{m<\omega:p_m\in G\}\in \dot{D}^h[G]$. 
			Let $c_m\in\oo$ be a witness of \eqref{eq_PPEL_assump} for each $m\in M$.
			%	By \eqref{item_qtau_initial_segment}, 
			%	$\dot{c}_{\b_{i(\tau^\frown m)}}\on(n^*+n-1)=s^\frown\tau$ 
			%	and $\dot{c}_{\b_{i(\tau^\frown m)}}\bpd_{n^*}\dot{\pi}$ 
			%	for $m\in M_\tau$.

			If $j-1\in\dot{A}[G]$, then $c_m(j-1)=m\leq\dot{\pi}[G](c_m\on (j-1))=\dot{\pi}[G](t)$ for each $m\in M$, which contradicts that $M\in\dot{D}^h[G]$ is infinite and $\dot{\pi}[G](t)$ is a natural number.
			Thus, $j-1\notin\dot{A}[G]$ and since $G$ is arbitrary containing $p^\infty$, we obtain \eqref{eq_PPEL_conclu_1}. 
			\eqref{eq_PPEL_conclu_2} is direct from Corollary \ref{cor_UF-limit_Preservation_Property}.
		\end{proof}
		\begin{rem}
			This proof highlights the difference between bounding prediction and $g$-prediction for $g\in(\omega\setminus2)^\omega$ 
			since in the case of $g$-prediction, we cannot consider such infinitely many $p_m$. Indeed, the forcing poset $\pr_g$ (defined later) has closed-ultrafilter-limits and increases $\ee_g$ and hence the limits actually do not keep $\ee_g$ small.
			
		\end{rem}
		Lemma \ref{lem_Prediction_Point_Exclusion_Lemma} tells us one limit condition excludes one possible prediction point and preserves the information of shorter initial segments of predicted reals.
		%Let us consider the situation where there are such $\omega$-many $t$'s whose values differ only at the final point $j-2$ and we have already obtained the limit conditions $p^\infty_t$ respectively.
		%Recall that Lemma \ref{lem_cuf_guardrail_closedness}, the closedness of cUF-limits, tells us that the limit conditions are still in the position where we can again take limits of them.
		%Thus, by taking the limit of the limit conditions $p^\infty_t$'s, we can exclude the next smaller point by using shorter initial segments.
		%Also, by Corollary \ref{cor_UF-limit_Preservation_Property}, we can preserve ``$j-1\notin\dot{A}$'', the unpredictability at $j-1$, and preserve the information of the shorter initial segment. We can continue this way.
		Thus, the strategy to prove Main Lemma \ref{lem_Main_Lemma} below, which states that closed-ultrafilter-limits keep $\eeb$ small, is as follows:
		\begin{enumerate}
			\item Assume the negation of the conclusion towards contradiction.
			\item By a $\Delta$-system argument and arranging Cohen reals (as witnesses of $c$ in Lemma \ref{lem_Prediction_Point_Exclusion_Lemma}), satisfy the assumption of Lemma \ref{lem_Prediction_Point_Exclusion_Lemma}.
			\item Taking limits infinitely many times (guaranteed by Lemma \ref{lem_cuf_guardrail_closedness}) in some suitable order, exclude potential elements of $\dot{A}$ ``downwards'' and ultimately obtain a condition for each $j$ which excludes the points between $n^*$ and $j$.
			\item Finally, take the limit of the ultimate conditions and exclude points ``upwards'' ($j\to\infty$), i.e., almost all points, which contradicts that there are infinitely many prediction points.
		\end{enumerate}

		\begin{mainlem}
			\label{lem_Main_Lemma}
			Consider the case $\Gamma_\mathrm{uf}=\Lambda^\mathrm{lim}_\mathrm{cuf}$.
			% (and repeatedly, suppose Assumption \ref{ass_first_Cohen})
%			
%			Assume:
%			\begin{itemize}
%				\item $\kappa$ is uncountable regular and $\kappa<\gamma$.
%				\item $H$ is complete and has size $<\kappa$.
%			\end{itemize} 
			Then, $\p_\gamma$ forces $C_{[\lambda]^{<\kappa}}\lq \mathbf{BPR}$, in particular, $\ee\leq\eeb\leq\kappa$.
		\end{mainlem}
		
		\begin{proof}
%			To avoid triviality, we may assume that each iterand is an atomless poset.
%			Since $\p_\gamma$ is a fsi, we can find the first $\kappa$-many Cohen reals (as members of $\oo$) $\bar{c}\coloneq\langle \dot{c}_\b:\b<\kappa \rangle$ added at the first $\kappa$-many limit stages whose cofinalities are countable.
%			For simplicity, we may assume that we have actually performed Cohen forcings in the first $\kappa$-many stages.
			Let $\bar{c}\coloneq\langle \dot{c}_\b:\b<\lambda \rangle$ be Cohen reals (as members of $\oo$) added in the first $\lambda$ stages.
			We shall show that $\bar{c}$ witnesses Fact \ref{fac_Tukey_order_equivalence_condition}\eqref{item_small_equiv}, an equivalent condition of $C_{[\lambda]^{<\kappa}}\lq \mathbf{BPR}$.
			
			Assume towards contradiction that there exist a condition $p\in\p_\gamma$ and a $\p_\gamma$-name of a predictor $\dot{\pi}=(\dot{A},\langle\dot{\pi}_k:k\in\dot{A}\rangle)$ such that 
			$p\Vdash |\{\b<\lambda:\dot{c}_\b\bpd \dot{\pi}\}|\geq\kappa$.
			%Let us witness the unboundedness and the starting point of the prediction $\bpd=\bigcup_{n<\omega}\bpd_n $.
			For each $i<\kappa$, inductively pick $p_i\leq p$, $\b_i<\lambda$ and $n_i<\omega$ such that $\b_i\notin\{\b_{i^\prime}:i^\prime<i\}$ and $ p_i\Vdash\dot{c}_{\b_i}\bpd_{n_i}\dot{\pi}$.
			By extending and thinning, we may assume:
			
			\begin{enumerate}
				\item $\b_i\in\dom(p_i)$. (By extending $p_i$.)
				\item All $ p_i$ follow a common guardrail $h\in H$. ($|H|<\kappa$.)
				\item $\{p_i:i<\kappa\}$ forms a uniform $\Delta$-system with root $\nabla $. (By Lemma \ref{lem_uniform_Delta_System_Lemma}.)
				%\item $\b_i\notin\nabla$, hence all $\b_i$ are distinct. (Since $\b_i\geq i$, $\b_i$ are eventually out of the finite set $\nabla$.)
				\item All $n_i$ are equal to $ n^*$.
				\item All $p_i(\b_i)$ are the same Cohen condition $s\in\seq$. 
				\item $|s|=n^*$. (By extending $s$ or increasing $n^*$.)
			\end{enumerate}
			
			In particular, we have that:
			\begin{equation}
				\label{eq_property_of_refined_pi}
				\text{For each }i<\kappa, p_i\text{ forces }\dot{c}_{\b_i}\on n^*=s\text{ and }\dot{c}_{\b_i}\bpd_{n^*}\dot{\pi}.
			\end{equation}
			
			Note that $\b_i\notin\nabla$ for $i<\kappa$ since all $\b_i$ are distinct. Pick the first $\omega$ many $p_i$ and fix some bijection $i\colon\seq\to\omega$.
			Fix any $n<\omega$. 
			(For simplicity, we assume $n\geq3$ ).
			For each $\sigma\in\omega^n$, define $q_\sigma\leq p_{i(\sigma)}$ by extending the $\b_{i(\sigma)}$-th position $q_\sigma(\b_{i(\sigma)}):=s^\frown\sigma$.
			By \eqref{eq_property_of_refined_pi}, we have:
			\begin{equation}
				\label{eq_property_of_qtau}
				\text{For each }\sigma\in\seq, q_\sigma\text{ forces }\dot{c}_{\b_{i(\sigma)}}\on(n^*+n)=s^\frown\sigma\text{ and }\dot{c}_{\b_{i(\sigma)}}\bpd_{n^*}\dot{\pi}.
			\end{equation}

			Fix $\tau\in\omega^{n-1}$ and we consider the sequence $\bar{q}_\tau\coloneq\langle q_{\tau^\frown m}:m<\omega\rangle$.
			When defining $q_\sigma$ we changed the $\b_{i(\sigma)}$-th position which is out of $\nabla$, so $\{q_{\tau^\frown m}:m<\omega\}$ forms a uniform $\Delta$-system with root $\nabla$, following some new countable partial guardrail $h^\prime$.
			Since $H$ is complete, $h^\prime$ is extended to some $h_\tau\in H$.
			Note that 
			\begin{equation}
				\label{eq_h_h_tau_conicide}
				h_\tau\on\nabla=h^\prime\on\nabla=h\on\nabla.
			\end{equation}
			Let $q^\infty_\tau:=\lim^{h_\tau}\bar{q}_\tau $. 
			
			By Lemma \ref{lem_cuf_guardrail_closedness}, $q^\infty_\tau$ follows $h_\tau$ and by Definition \ref{dfn_uflimit_condition_for_iteration}, $\dom(q^\infty_\tau)=\nabla$.
			Thus, by \eqref{eq_h_h_tau_conicide} and Fact \ref{fac_guarrail_folloing_only_domain}, $q^\infty_\tau$ also follows $h$.
			By \eqref{eq_property_of_qtau}, each $q_{\tau^\frown m}$ forces that:
			
			\begin{equation}
				\text{there exists }c\in\oo\text{ such that } c\on(n^*+n)=s^\frown\tau^\frown m\text{ and }c\bpd_{n^*}\dot{\pi}.
			\end{equation}
			
			Thus, we are under the assumption of Lemma \ref{lem_Prediction_Point_Exclusion_Lemma} and hence obtain:
			\begin{equation}
				q^\infty_\tau\text{ forces } n^*+n-1\notin\dot{A}\text{ and }\varphi_\tau,
			\end{equation}
			where:
			\begin{equation}
				\varphi_\tau:\equiv\text{``there exists }c\in\oo\text{ such that } c\on(n^*+n-1)=s^\frown\tau\text{ and }c\bpd_{n^*}\dot{\pi}\text{''.}
			\end{equation}
			
			Unfix $\tau$ and fix $\rho\in\omega^{n-2}$.
			We consider the sequence $\bar{q}_\rho\coloneq\langle q^\infty_{\rho^\frown m}:m<\omega\rangle$.
			%Since all the domains of $q^\infty_{\rho^\frown m}$ are the same $\nabla$, 
			%$\bar{q}_\rho$ is trivially a uniform $\Delta$-system with root $\nabla$.
			Since all $q^\infty_{\rho^\frown m}$ follow $h$ and have domain $\nabla$, they form a uniform $\Delta$-system with root $\nabla$ and have a limit $q^\infty_\rho:=\lim^h\bar{q}_\rho$.
			%Also, since all $q_{\rho^\frown m}$ follow $h$, we can define $q^\infty_\rho:=\lim^h\bar{q}_\rho$.
			Similarly, we have that $\dom(q^\infty_\rho)=\nabla$ and $q^\infty_\rho$ follows $h$. 
			Note that each $q^\infty_{\rho^\frown m}$ forces $n^*+n-1\notin\dot{A}$ and $\varphi_{\rho^\frown m}$.
			Thus, we are under the assumption of Lemma \ref{lem_Prediction_Point_Exclusion_Lemma} and hence obtain that $q^\infty_\rho$ forces:
			\begin{itemize}
				\item $n^*+n-1\notin\dot{A}$.
				\item $n^*+n-2\notin\dot{A}$.
				\item $\text{there exists }c\in\oo\text{ such that }c\on(n^*+n-2)=s^\frown\rho\text{ and }c\bpd_{n^*}\dot{\pi}$.
			\end{itemize}
			(The first item is direct from Corollary \ref{cor_UF-limit_Preservation_Property}.) 
			Continuing this way, we ultimately obtain $q^n\coloneq q^\infty_\emptyset$ with the following properties:
			\begin{itemize}
				\item $\dom(q^n)=\nabla$ and $q^n$ follows $h$ (hence they form an $h$-uniform $\Delta$-system).
				\item $q^n$ forces $ \left[n^*,n^*+n\right) \cap \dot{A}=\emptyset$.
			\end{itemize}
			%(The information of the initial segments is not required anymore)
			Finally, unfix $n$ and let $q^\infty$ be the limit condition of the ultimate conditions $q^\infty:=\lim^h\langle q^n:n<\omega\rangle$.
			$q^\infty$ forces that 
			for infinitely many $n<\omega$, $\left[n^*,n^*+n\right) \cap \dot{A}=\emptyset$
			, which contradicts that $\dot{A}$ is infinite.
		\end{proof}
		
		%\begin{rem}
		%	Roughly speaking, we first excluded points ``downwards'' and then ``upwards''. 
		%	The reason for ``downwards'' is that the information of an initial segment also has that of shorter ones.
		%\end{rem}
		
		\begin{rem}
			The tricks of the proof are as follows:
			\begin{itemize}
				\item By quantifying over $c$, we succeeded to define $\varphi_\tau$ without parameter $m$ and apply Lemma \ref{lem_Prediction_Point_Exclusion_Lemma}.
				\item By using the intervals $\left[n^*,n^*+n\right)$, we succeeded to capture the infinite set $\dot{A}$.
			\end{itemize}
		\end{rem}
		
		\subsection{Forcing-free characterization and concrete forcing notions}  
		We introduce a forcing-free characterization of ``$Q\subseteq\p$ is (c-)uf-lim-linked'': 
		\begin{lem}
			\label{lem_chara_UF_linked}
			Let $D$ be an ultrafilter on $\omega$, $\p$ a poset, $Q\subseteq\p$, $\lim^D\colon Q^\omega\to\p$.
			Then, the following are equivalent:
			\begin{enumerate}
				\item $\lim^D$ witnesses $Q$ is $D$-lim-linked.\label{item_lim_1}
				\item $\lim^D$ satisfies $(\star)_n$ below for all $n<\omega$:\label{item_lim_2}
				\begin{align*}
					(\star)_n:& \text{``Given }\bar{q}^j=\langle q_m^j:m<\omega\rangle\in Q^\omega\text{ for }j<n\text{ and }r\leq\textstyle{\lim^D}\bar{q}^j\text{ for all }j<n,\\
					&\text{then }\{m<\omega:r \text{ and all }q_m^j \text{ for } j<n \text{ have a common extension}\}\in D\text{''}.
				\end{align*}
			\end{enumerate}
		\end{lem}
		
%		\begin{lem}
%			\label{lem_suff_has_uf}
%			Let $D$ be an ultrafilter on $\omega$ and $Q\subseteq\p$ is suff-$D$-lim-linked witnessed by $\lim^D$.
%			Then, $Q$ is uf-lim-linked with the same witness $\lim^D$.
%			In particular, $\suff\subseteq\Lambda^\mathrm{lim}_\mathrm{uf}$ and $\suff_\mathrm{c}\subseteq\Lambda^\mathrm{lim}_\mathrm{cuf}$.
%		\end{lem}
		\begin{proof}
			$(1)\Rightarrow(2)$:
			Let $\dot{D}$ be the $\p$-name of an ultrafilter extending $D$ as in Definition \ref{dfn_UF_linked}. 
			Then, $r$ forces:
			\begin{gather*}
				\dot{A}\coloneq\bigcap_{j<n}\{m<\omega:q^j_m\in \dot{G}\}\in\dot{D},\text{ and }\\
				B\coloneq\{m<\omega:r \text{ and all }q_m^j \text{ for } j<n \text{ have a common extension}\}\supseteq \dot{A}.
			\end{gather*}
			 Thus, we have $B\in D$ since $B$ is in the ground model.
			
			$(2)\Rightarrow(1)$:
			For $\bar{q}=\langle q_m:m<\omega\rangle\in Q^\omega$, let $\Vdash_\p \dot{A}(\bar{q})\coloneq\{m<\omega:q_m\in\dot{G}\}$.
			We show $\Vdash_\p\text{``} D\cup\{\dot{A}(\bar{q}):\bar{q}\in Q^\omega\cap V,~\lim^D\bar{q}\in\dot{G}\}$ has SFIP''.
			If not, there exist $r\in\p$, $A\in D$, $n<\omega$ and $\bar{q}^j=\langle q_m^j:m<\omega\rangle\in Q^\omega$ satisfying $r\Vdash\lim^D\bar{q}^j\in\dot{G}$ for $j<n$, 
			such that $r\Vdash A\cap\bigcap_{j<n}\dot{A}(\bar{q}^j)=\emptyset$.
			We may assume that $r\leq\lim^D\bar{q}^j$ for all $j<n$.
			%By \eqref{item_1_dfn_suff} in Definition \ref{dfn_suff}, 
			%$B^j\coloneq\{m<\omega: r\text{ and }q^j_m \text{ are compatible}\}\in D$ for each $j<n$.
			By $(\star)_n$, we can find some $m\in A$ and $\tilde{r}\leq r$ extending all $q^j_m$.
			%Since we are under the assumption of \eqref{item_2_dfn_suff} in Definition \ref{dfn_suff}, there is $\tilde{r}\leq r$ extending all $q^j_m$.
			Thus, $\tilde r\Vdash m\in A\cap\bigcap_{j<n}\dot{A}(\bar{q}^j)$, which is a contradiction.
			Hence, we can take a name of an ultrafilter $\dot{D}^\prime$ extending
			$D\cup\{\dot{A}(\bar{q}):\bar{q}\in Q^\omega\cap V,~\lim^D\bar{q}\in\dot{G}\}$.   
			Let $\bar{q}=\langle q_m:m<\omega\rangle\in Q^\omega$ be arbitrary.
			Since $\lim^D\bar{q}\Vdash\lim^D\bar{q}\in\dot{G}$ trivially holds,
			$\lim^D\bar{q}\Vdash\dot{A}(\bar{q})=\{m<\omega:q_m\in\dot{G}\}\in\dot{D}^\prime$ is obtained and we are done.
		\end{proof}
		
		%あとで　GMS引用
		This characterization in Lemma \ref{lem_chara_UF_linked} enables us to investigate whether a poset $\p$ has ultrafilter-limits without considering forcings. Note that the closedness can be easily checked, since the conditions \eqref{item_lim_1} and \eqref{item_lim_2} in Lemma \ref{lem_chara_UF_linked} share the same witness $\lim^D$.

		Using this characterization, we show that some concrete forcing notions have ultrafilter-limits.
		%We will show some concrete forcing notions and 
		\begin{lem}(\cite{GKS})
			Eventually different forcing $\ev$ has ultrafilter-limits,
			where $\ev$ is defined as follows in this paper:
			\begin{itemize}
				\item $\ev\coloneq\{(s,k,\varphi):s\in\seq, k<\omega, \varphi\colon\omega\to [\omega]^{\le k}\}$.
				\item $(s^\prime,k^\prime,\varphi^\prime)\le (s,k,\varphi)$ 
				if $s^\prime\supseteq s$, 
				$k^\prime\ge k$, $\varphi^\prime(i) \supseteq \varphi(i)$ for all $i<\omega$ and $s^\prime(i)\notin\varphi(i)$ for all $i\in\dom(s^\prime\setminus s)$.
			\end{itemize}
			
		\end{lem}
		
		\begin{proof}
			For $s\in\seq$ and $k<\omega$, let $Q_{s,k}\coloneq\{(s^\prime,k^\prime,\varphi)\in\ev:s^\prime=s,k^\prime=k\}$. We show $Q\coloneq Q_{s,k}$ is uf-lim-linked. Let $D$ be any ultrafilter and define $\lim^D\colon Q^\omega\to Q$, $\bar{q}=\langle q_m=(s,k,\varphi_m):m<\omega\rangle\mapsto(s,k,\varphi^\infty)$ as follows:
			\[ 
			j\in \varphi_\infty(i) \Leftrightarrow \{ m<\omega: j\in \varphi_m(i)\}\in D.
			\]
			We check first $(\star)_1$ and then $(\star)_n$ for $n\geq2$ ($(\star)_0$ trivially holds).
			
				\textit{The case $(\star)_1$}:
				Assume that $\bar{q}=\langle q_m:m<\omega\rangle\in Q^\omega$ and $r\leq\lim^D\bar{q}$.
				Let $r=(s^\prime,k^\prime,\varphi^\prime)$ and $\lim^D\bar{q}=(s,k,\varphi^\infty)$.
				Since $r\leq\lim^D\bar{q}$, $s^\prime(i)\notin \varphi_\infty(i)$ for all $i\in\dom(s^\prime\setminus s)$. 
				That is, $\{m<\omega:s^\prime(i)\in \varphi_m(i)\}\notin D$.
				Since $D$ is an ultrafilter, $A^i :=\{m: s^\prime(i)\notin \varphi_m(i)\}\in D$ for $i\in\dom(s^\prime\setminus s)$. 
				Let $A\coloneq\bigcap\{ A^i:i\in\dom(s^\prime\setminus s)\}\in D$ and for $m\in A$,  $s^\prime(i)\notin \varphi_m(i)$ for all $i\in\dom(s^\prime\setminus s)$. 
				Thus, $q_m$ is compatible with $r$.
				
				\textit{The case $(\star)_n$}: Suppose that:
				\begin{itemize}
					\item $n\geq2$ and $\bar{q}^j=\langle q_m^j:m<\omega\rangle\in Q^\omega$ for $j<n$,
					\item $r\leq\lim^D\bar{q}^j$ for all $j<n$,
					%\item $m<\omega$ and $r$ and $q_m^j$ are compatible for all $j<n$.
				\end{itemize}
				By $(\star)_1$, $A\coloneq\bigcap_{j<n}\{m<\omega:r \text{ and }q^j_m \text{ are compatible}\}\in D$.
				Let $m\in A$ and define $r=(s^\prime,k^\prime,\varphi^\prime)$ and $q^j_m=(s,k,\varphi^j_m)$. Since $s^\prime\supseteq s$, $k^\prime\ge k$ and $r$ and $q_m^j$ are compatible,
				the condition $\tilde{r}\coloneq(s^\prime,k^\prime+nk,\tilde{\varphi})$ extends $r$ and all $q_m^j$ for $j<n$
				where $\tilde{\varphi}(i)\coloneq\varphi^\prime(i)\cup\bigcup_{j<n}\varphi^j_m(i)$ for each $i<\omega$.
				Thus, $A$ witnesses $(\star)_n$ holds.
		\end{proof}
		
		\begin{cor}
			\label{cor_E_is_centered_suff_cuf}
			$\mathbb{E}$ is $\sigma$-$\left( \Lambda(\text{centered})\cap\Lambda^\mathrm{lim}_\mathrm{cuf}\right) $-covered.
		\end{cor}
		\begin{proof}
			$Q$ in the previous proof is centered and the limit function $\lim^D$ is closed in $Q$.
		\end{proof}

		The next example of a forcing notion with ultrafilter-limits is \textit{$g$-prediction forcing} $\pr_g$, which generically adds a $g$-predictor and hence increases $\ee_g$, introduced in \cite{Bre95}\footnote{The name ``prediction forcing'' and the notation $\pr_g$ are not common and were not used in the original paper \cite{Bre95}, but we use the name and the notation in this paper.}.
		
		\begin{dfn}
			\label{dfn_PR}
			Fix $g\in\left(\omega+1\setminus2\right)^\omega$.
			$g$-prediction forcing $\pr_g$ consists of tuples $(d,\pi,F)$ satisfying:
			
			\begin{enumerate}
				\item $d\in\sq$.
				\item $\pi=\langle\pi_n:n\in d^{-1}(\{1\})\rangle$.
				\item for each $n\in d^{-1}(\{1\})$, $\pi_n$ is a finite partial function of $\prod_{k<n}g(k)\to g(n)$.
				\item $F\in[\prod_{n<\omega}g(n)]^{<\omega}$
				\item for each $f,f^\prime\in F, f\on|d|=f^\prime\on|d|$ implies $f=f^\prime$.
			\end{enumerate}
			$(d^\prime,\pi^\prime,F^\prime)\leq(d,\pi,F)$ if:
			\begin{enumerate}[(i)]

				\item $d^\prime\supseteq d$.
				\item $\forall n\in d^{-1}(\{1\}), \pi_n^\prime\supseteq\pi_n$.
				\item $F^\prime\supseteq F$.
				\item \label{item_PR_order_long}
				For all $ n\in(d^\prime)^{-1}(\{1\})\setminus d^{-1}(\{1\})$ and $f\in F$, we have $f\on n\in\dom(\pi^\prime_n)$ and $\pi^\prime_n(f\on n)=f(n)$.
				
			\end{enumerate}
			When $g(n)=\omega$ for all $n<\omega$, we write $\pr$ instead of $\pr_g$ and just call it ``prediction forcing''.
		\end{dfn}

		We introduce a useful notation:
		\begin{dfn}
			%Let $\zero$ be the set of all finite sequences whose values are always $0$.
			%Namely, $\zero\coloneq{0}^{<\omega}$.
			For $N<\omega$, let $\zero_N $ be the sequence of length $N$ whose values are all $0$.
			Namely, $\zero_N \coloneq N\times 1$.
			
		\end{dfn}
		%
		%\begin{dfn}
		%	Let $D$ be an ultrafilter on $\omega$ and $\bar{f}=\langle f^i\in\oo:i<\omega\rangle$. 
		%	
		%	If for each $n<\omega$, there (uniquely) exists $m_n\omega$ such that $\{i<\omega:f^i(n)=m_n\}$,
		%	
		%	we define $\bar{f}^\infty\in\oo$ by $ \bar{f}^\infty(n)=m_n$ for each $n<\omega$.
		%\end{dfn}
		%
		%\begin{dfn}
		%	Fix $k<\omega$,$d,\pi$ and $f^*=\{f^*_l\in\omega^{|d|}:l<k\}$.
		%	For a countable sequence of conditions $\bar{q}=\langle p^i=(d,\pi,\{f^i_l:l<k\})\in\pr:i<\omega\rangle$ with $\forall l<k,f^i_l\on|d|=f^*_l$,
		%	define:
		%	\begin{itemize}
			%		\item $\bar{f}_l\coloneq \langle f^i_l:i<\omega\rangle$.
			%		\item $A\coloneq \{l<k:{\bar{f}_l^\infty}\text{ exists }\}, B\coloneq k\setminus A$.
			%		\item $F^\infty\coloneq \{{\bar{f}_l^\infty}:l\in A\}$.
			%		\item for $l\in B$, $n_l$ is the first $n<\omega$ where no $m<\omega$ satisfies $\{i<\omega:f^i_l(n)=m\}\in D$.
			%		\item $n^\infty\coloneq \max\{n_l+1:l\in B\}$ (if $B=\emptyset$, $n^\infty\coloneq |d|$).
			%		\item $d^\infty\coloneq d^\frown \zero_{n^\infty-|d|}\in2^{n^\infty}$.
			%		\item $\lim^D\bar{q}\coloneq (d^\infty,\pi,F^\infty)$.
			%	\end{itemize}
		%\end{dfn}
		%Note that $\lim^D{\bar{q}}$ is always defined and a condition of $\pr$.
		%\begin{rem}
		%	\label{rem_pr_B}
		%	\item For every $l\in B$ and $m<\omega, \{i<\omega:f^i_l(n_l)>m\}\in D$.
		%\end{rem}

		\begin{lem}(\cite{BS_E_and_P_2})%\footnote{In \cite{BS_E_and_P_2}, Lemma \ref{lem_pr_is_uflinked} was not directly proved, but the key point of the proof of the lemma appeared there. In other words, by using the notation of this paper, the (1) part of the proof for the special case where $D$ is the cofinite filter was proved there.})
			\label{lem_pr_is_uflinked}
			$\pr_g$ has ultrafilter-limits for any $g\in\left(\omega+1\setminus2\right)^\omega$.
			If $g\in\left(\omega\setminus2\right)^\omega$, 
			then $\pr_g$ has closed-ultrafilter-limits.
		\end{lem}
		
		\begin{proof}
			Fix $k<\omega$, $d,\pi$ and $f^*=\{f^*_l\in\omega^{|d|}:l<k\}$ where each $f^*_l$ is different.
			Let $Q\subseteq\pr_g$ consist of every $(d^\prime,\pi^\prime,F^\prime)$
			such that $d^\prime=d, \pi^\prime=\pi$ and $F^\prime=\{f^\prime_l:l<k\})$ 
			where $f^\prime_l\on|d|=f^*_l$ for all  $l<k$.
			It is enough to show that $Q$ is uf-lim-linked.
			Let $D$ be any ultrafilter.
			For $\bar{f}=\langle f^i\in\oo:i<\omega\rangle$, 
			if $\bar{f}$ satisfies that for each $n<\omega$, there (uniquely) exists $a_n<\omega$ such that $\{i<\omega:f^i(n)=a_n\}\in D$,
			we define $\bar{f}^\infty\in\oo$ by $ \bar{f}^\infty(n)=a_n$ for each $n<\omega$.
			Note that ${\bar{f}}^\infty$ is always defined if $g\in\left(\omega\setminus2\right)^\omega$.
			
			For $\bar{q}=\langle q_m=(d,\pi,\{f^m_l:l<k\})\in\pr:m<\omega\rangle\in Q^\omega$,
			define $\lim^D\bar{q}$ as follows: 
			\begin{itemize}
				\item $\bar{f}_l\coloneq\langle f^m_l:m<\omega\rangle$, $A\coloneq\{l<k:{\bar{f}_l^\infty}\text{ exists}\}, B\coloneq k\setminus A$.
				\item $F^\infty\coloneq\{{\bar{f}_l^\infty}:l\in A\}$.
				\item For $l\in B$, let $n_l$ be the first $n<\omega$ where no $a<\omega$ satisfies: 
				
				$\{m<\omega:f^m_l(n)=a\}\in D$ (hence $n_l\geq|d|$).
				\item $n^\infty\coloneq\max\{n_l+1:l\in B\}$ (if $B=\emptyset$, $n^\infty\coloneq|d|$).
				\item $d^\infty\coloneq d^\frown\zero_{n^\infty-|d|}\in2^{n^\infty}$.
				\item $\lim^D\bar{q}\coloneq(d^\infty,\pi,F^\infty)$.
			\end{itemize}
			To see $\lim^D\bar{q}$ is a condition, it is enough to show that any $\bar{f}_l^\infty,\bar{f}_{l^\prime}^\infty\in F^\infty $ with $\bar{f}_l^\infty\on|d^\infty|=\bar{f}_{l^\prime}^\infty\on|d^\infty|$ satisfies $\bar{f}_l^\infty=\bar{f}_{l^\prime}^\infty$.
			Since $f^*_l=\bar{f}_{l}^\infty\on|d|=\bar{f}_{l^\prime}^\infty\on|d|=f^*_{l^\prime}$, we have $l=l^\prime$ and hence $\lim^D\bar{q}$ is a valid condition.
			
			Note that if $g\in\left(\omega\setminus2\right)^\omega$, then $A=k, B=\emptyset, n^\infty=|d|, d^\infty=d$ and hence $\lim^D\bar{q}\in Q$ (for any $D$).
			
			%Thus, since the closedness has been shown, 
			So, it is enough to show that $\lim^D$ satisfies $(\star)_n$ for $n<\omega$. We check first $(\star)_n$ for $n\geq2$ assuming $(\star)_1$ and then $(\star)_1$ since the former is easier to show.
			
				\textit{The case $(\star)_n$}: Assuming $(\star)_1$, suppose that:
			\begin{itemize}
				\item $n\geq2$ and $\bar{q}^j=\langle q_m^j:m<\omega\rangle\in Q^\omega$ for $j<n$,
				\item $r\leq\lim^D\bar{q}^j$ for all $j<n$,
				%\item $m<\omega$ and $r$ and $q_m^j$ are compatible for all $j<n$.
			\end{itemize}
			By $(\star)_1$, $A\coloneq\bigcap_{j<n}\{m<\omega:r \text{ and }q^j_m \text{ are compatible}\}\in D$.
			Let $m\in A$ and put $r=(d^\prime,\pi^\prime,F^\prime)$ and $q^j_m=(d,\pi,F^j_m)$. Then, the condition $\tilde{r}\coloneq({d^\prime}^\frown\zero_N,\pi^\prime,F^\prime\cup\bigcup_{j<n} F^j_m)$ extends $r$ and all $q_m^j$ for $j<n$ where $N$ is large enough to satisfy 
			that any distinct functions in $F^\prime\cup\bigcup_{i<n}F_i$ have different values before $|d^\prime|+N$ (hence $\tilde{r}$ is a condition and hereafter we use ``$N$ is large enough'' in this sense).
			
			\textit{The case $(\star)_1$}: Assume that $\bar{q}=\langle q_m:m<\omega\rangle\in Q^\omega$ and $r\leq\lim^D\bar{q}=(d^\infty,\pi,F^\infty)$.
			Let $r\coloneq (d^r,\pi^r,F^r)$ and $\mathrm{New}\coloneq(d^r)^{-1}(\{1\})\setminus (d^\infty)^{-1}(\{1\})$.
			
			Fix $l\in A$ and $n\in\mathrm{New}$.
			By the definition of $\bar{f}_l^\infty$,
			$ X_0\coloneq\{m<\omega: f^m_l\on(n+1)=\bar{f}_l^\infty\on(n+1)\}\in D$.
			Along with $r\leq\lim^D\bar{q}$, for $m\in X_0$, $\pi^r_n(f^m_l\on n)=\bar{f}_l^\infty(n)=f^m_l(n)$.
			
			Unfixing $l$ and $n$, we have:
			\begin{equation}
				\label{eq_pr_limit_1}
				X_1\coloneq\{m<\omega:\pi^r_n(f^m_l\on n)=f^m_l(n)\text{ for all }l\in A\text{ and } n\in\mathrm{New}\}\in D.
			\end{equation}
			Fix $l\in B$ and $n\in\mathrm{New}$.
			Let $m_n\coloneq \max\{\sigma(j):\sigma\in\dom(\pi^r_n),j<n\}$.
			Since $l\in B$, $X_2\coloneq\{m<\omega: f^m_l(n_l)>m_n\}\in D$.
			Since $n_l<n^\infty\leq n$, for all $m\in X_2, f^m_l\on n\notin\dom(\pi^r_n)$.
			Unfixing $l$ and $n$, we have:
			\begin{equation}
				\label{eq_pr_limit_2}
				X_3\coloneq\{m<\omega:f^m_l\on n\notin\dom(\pi^r_n)\text{ for all }l\in B\text{ and } n\in\mathrm{New}\}\in D.
			\end{equation}
			%By \eqref{eq_pr_limit_1} and \eqref{eq_pr_limit_2},
			It is enough to show that for all $m\in X_1\cap X_3$, $q_m$ and $r$ are compatible.
			Fix such $m$ and define (as a common extension of $q_m$ and $r$) $q^\prime\coloneq (d^\prime,\pi^\prime,F^\prime)$ as follows:
			\begin{itemize}
				\item $F^\prime\coloneq F^r\cup \{f^m_l:l<k\}$.
				\item $d^\prime=d^{r\frown} \zero_N $ where $N$ is large enough.
				\item For all $n\in (d^r)^{-1}(\{1\}), \pi^\prime_n\supseteq\pi^r_n$ and for all $l\in B$ and $n\in\mathrm{New}, \pi^\prime_n(f^m_l\on n)=f^m_l(n)$
				(This can be done by \eqref{eq_pr_limit_2}).
			\end{itemize}
			
			$q^\prime\leq r$ trivially holds since $(d^\prime)^{-1}(\{1\})\setminus (d^r)^{-1}(\{1\})=\emptyset$.
			
			To see $q^\prime\leq q_m$, we have to show:
			\begin{equation}
				\label{eq_pr_limit_3}
				\text{For all } l<k \text{ and } n\in(d^r)^{-1}(\{1\})\setminus (d)^{-1}(\{1\}), \pi^\prime_n(f^m_l\on n)=f^m_l(n).
			\end{equation}
			If $l\in A$, \eqref{eq_pr_limit_1} implies \eqref{eq_pr_limit_3}, while if $l\in B$, \eqref{eq_pr_limit_3} holds by the definition of $\pi^\prime$.
		\end{proof}
		
		\begin{cor}
			\label{cor_PR_is_sigma_centered_suff}
			$\pr_g$ is $\sigma$-$\left( \Lambda(\text{centered})\cap\Lambda^\mathrm{lim}_\mathrm{uf}\right) $-covered.
			Moreover, 
			if $g\in(\omega\setminus2)^\omega$, 
			$\pr_g$ is $\sigma$-$\left( \Lambda(\text{centered})\cap\Lambda^\mathrm{lim}_\mathrm{cuf}\right)$-covered.
		\end{cor}
		\begin{proof}
			$Q$ in the previous proof is centered since any finitely many conditions $\{(d,\pi,F_i)\in Q:i<n\}$ have a common extension $r=(d^\frown\zero_N ,\pi,\bigcup_{i<n}F_i)$ where $N$ is large enough.
			%Also recall that the limit function is closed in $Q$ whenever $g\in(\omega\setminus2)^\omega$.
		\end{proof}

		\section{Separation} \label{sec_separation}
		\subsection{Separation of the left side}
		\newcommand{\ind}{i}
		\newcommand{\pst}{\mathbb{P}}
		We are ready to prove the main theorem, \textit{Cicho\'n's maximum with evasion number}. We roughly follow the flow of the original construction of Cicho\'n's maximum in \cite{GKS} and \cite{GKMS}, i.e., we first separate the left side of the diagram by performing a fsi %like \cite{GKS}
		and then the right side by submodel method introduced in \cite{GKMS}.
		%However, there is a problem that the cardinal arithmetics we shall assume for the separation of the left (Assumption \ref{ass_for_complete_guardrail}) and that of the right (GCH\footnote{While we can actually weaken the assumption more than GCH, but we assume GCH in this paper for simplicity. See \cite{forcing_constellations} for details.}) actually \textit{conflicts}.
		%To overcome this problem, we construct three posets $\pst_{\mathrm{pre}}$, $\pst_{\mathrm{mid}}$ and $\pst_{\mathrm{fin}}$ in this order:
		%First, we construct $\pst_{\mathrm{pre}}$ which separates the left side under the cardinal arithmetic Assumption \ref{ass_for_complete_guardrail}.
		%Then, we construct under GCH $\pst_{\mathrm{mid}}$ whose separation constellation is the same as $\pst_{\mathrm{pre}}$'s, by mimicking the construction of $\pst_{\mathrm{pre}}$ (This kind of argument of recovering GCH originally appears in \cite{GKS}).
		%Finally, we construct $\pst_{\mathrm{fin}}$ which separates also the right side by using the submodel method using GCH.

		\begin{dfn}
			\begin{itemize}
				\item $\br_0\coloneq\mathbb{C}$, the Cohen forcing.
				\item $\R_1\coloneq\mathbf{Lc}^*$ and $\br_1\coloneq\mathbb{A}$, the Amoeba forcing.
				\item $\R_2\coloneq\mathbf{Cn}$ and $\br_2\coloneq\mathbb{B}$, the random forcing.
				\item $\R_3\coloneq\mathbf{D}$ and $\br_3\coloneq\mathbb{D}$, the Hechler forcing.
				\item $\R_4\coloneq\mathbf{PR}$, $\R_4^*\coloneq\mathbf{BPR}$ and $\br_4\coloneq\mathbb{PR}$.
				\item $\R_5\coloneq\mathbf{Mg}$ and $\br_5\coloneq\mathbb{E}$.
			\end{itemize}
			Let $I\coloneq 6$ and $I^+\coloneq I\setminus \{0\}$ be the index sets.
		\end{dfn}
		Hence, $\br_{\ind}$ is the poset which increases $\bb(\R_\ind)$ for each $\ind\in I^+$.
		Also note that $\R_4^*\lq\R_4$.
		\begin{ass}
			\label{ass_card_arith}
			
			\begin{enumerate}
				\item $\lambda_1<\cdots<\lambda_6$ are regular uncountable cardinals.
				\item $\lambda_3=\mu_3^+$ and $\lambda_4=\mu_4^+$ are successor cardinals and $\mu_3$ is regular.
				\item \label{item_aleph1_inacc}$\kappa<\lambda_\ind$ implies $\kappa^{\aleph_0}<\lambda_\ind$ for all $\ind\in I^+$.
				%\item $\lambda_3^{\aleph_0}<\lambda_4$ and $\lambda_4^{\aleph_0}<\lambda_5$.
				\item \label{item_ca_3}
				$\lambda_6^{<\lambda_5}=\lambda_6$, hence $\lambda_6^{<\lambda_\ind}=\lambda_6$ for all $\ind\in I^+$.

			\end{enumerate}
		\end{ass}
		
		\begin{lem}
			\label{lem_ccc_rs}
			Every ccc poset forces $C_{[\lambda_6]^{<\lambda_\ind}}\cong_T C_{[\lambda_6]^{<\lambda_\ind}}\cap V\cong_T[\lambda_6]^{<\lambda_\ind}\cap V\cong_T[\lambda_6]^{<\lambda_\ind}$ for any $\ind\in I^+$.
		\end{lem}
		
		\begin{proof}
			Direct from Fact \ref{fac_suff_eq_CI_and_I}, Fact \ref{fac_cap_V} and Assumption \ref{ass_card_arith} \eqref{item_ca_3}.
		\end{proof}
		
		Thus, we often identify the four relational systems in Lemma \ref{lem_ccc_rs} in this section.
		
		To satisfy Assumption \ref{ass_first_Cohen}, we define the following: 
		\begin{dfn}
			Put $\lambda\coloneq\lambda_6$, $S_0\coloneq\lambda$, $\gamma\coloneq\lambda+\lambda$, the length of the iteration we shall perform. Fix some cofinal partition $S_1\cup\cdots\cup S_5=\gamma\setminus S_0$ and for $\xi<\gamma$, let $\ind(\xi)$ denote the unique $\ind\in I$ such that $\xi\in S_\ind$.
		\end{dfn}
		
		We additionally assume the following cardinal arithmetic to obtain a complete set of guardrails:
		\begin{ass}
			\label{ass_for_complete_guardrail}
			$\lambda_6\leq2^{\mu_3}$.
		\end{ass}
		%Note that this Assumption \ref{ass_for_complete_guardrail} conflicts with GCH.

		\begin{lem}
			For $j=3,4$ and $\xi<\gamma$, let $\theta^j_{\xi}\coloneqq\mu_j$ if $\ind(\xi)\leq j$ and $\theta^j_{\xi}\coloneqq\omega$ if $\ind(\xi)> j$.
			Then, there exist complete sets $H$ and $H^\prime$ of guardrails of length $\gamma$ for $\lambda_3$-$\Lambda^\mathrm{lim}_\mathrm{uf}$-iteration of size $<\lambda_3$ with $\langle\theta_\xi^3:\xi<\gamma\rangle$, and $\lambda_4$-$\Lambda^\mathrm{lim}_\mathrm{cuf}$-iteration of size $<\lambda_4$ with $\langle\theta_\xi^4:\xi<\gamma\rangle$, respectively.  
		\end{lem}
		
		\begin{proof}
			Direct from Corollary \ref{cor_complete}, Assumption \ref{ass_card_arith} \eqref{item_aleph1_inacc} and \ref{ass_for_complete_guardrail}.
		\end{proof}

		\begin{con}
			\label{con_P6}
			We can construct a ccc finite support iteration $\pst_\mathrm{pre}$ of length $\gamma$ satisfying the following items:
			\begin{enumerate}
				%\item \label{item_P6_0} For all $\xi<\lambda$, $\Vdash_{\p_\xi}\dot{Q}_\xi\coloneq\mathbb{C}$, Cohen forcing.
				\item \label{item_P6_1}
				$\pst_\mathrm{pre}$ is a $\lambda_3$-$\Lambda^\mathrm{lim}_\mathrm{uf}$-iteration and has $\Lambda^\mathrm{lim}_\mathrm{uf}$-limits on $H$ 
				with the following witnesses:
				\begin{itemize}
					\item $\langle\p_\xi^-:\xi<\gamma\rangle$, the complete subposets witnessing $\Lambda^\mathrm{lim}_\mathrm{uf}$-linkedness.
					\item $\bar{Q}=\langle\dot{Q}_{\xi,\zeta}:\zeta<\theta_\xi^3,\xi<\gamma\rangle$, the $\Lambda^\mathrm{lim}_\mathrm{uf}$-linked components.
					\item $\bar{D}=\langle \dot{D}^h_\xi:\xi\leq\gamma,~h\in H \rangle$, the ultrafilters.
					\item $S^-\coloneq S_0\cup S_1\cup S_2\cup S_3$, the trivial stages and $S^+\coloneq S_4\cup S_5$, the non trivial stages.
				\end{itemize}
				\item \label{item_P6_2}
				$\pst_\mathrm{pre}$ is also a $\lambda_4$-$\Lambda^\mathrm{lim}_\mathrm{cuf}$-iteration and has $\Lambda^\mathrm{lim}_\mathrm{cuf}$-limits on $H^\prime$ 
				with the following witnesses:
				\begin{itemize}
					\item $\langle\p_\xi^-:\xi<\gamma\rangle$, the (same as above) complete subposets witnessing $\Lambda^\mathrm{lim}_\mathrm{cuf}$-linkedness.
					\item $\bar{R}=\langle\dot{R}_{\xi,\zeta}:\zeta<\theta_\xi^4,\xi<\gamma\rangle$, the $\Lambda^\mathrm{lim}_\mathrm{cuf}$-linked components.
					\item $\langle \dot{E}^{h^\prime}_\xi:\xi\leq\gamma,~h^\prime\in H^\prime\rangle$, the ultrafilters.
					\item $T^-\coloneq S_0\cup S_1\cup S_2\cup S_3\cup S_4$, the trivial stages and $T^+\coloneq S_5$, the non-trivial stages.
				\end{itemize}
				\item 
				\label{item_N}
				%\footnote{These submodels have nothing to do with the submodels which shall be used in the separation of the right side.}
				For each $\xi\in\gamma\setminus S_0$, $N_\xi\preccurlyeq H_\Theta$ is a submodel where $\Theta$ is a sufficiently large regular cardinal satisfying:
				\begin{enumerate}
					\item $|N_\xi|<\lambda_{\ind(\xi)}$.
					\item \label{item_sigma_closed}
					% $N_\xi$ is closed under $\omega$-sequences 
					$N_\xi$ is $\sigma$-closed, i.e., $(N_\xi)^\omega\subseteq N_\xi$.
					
					\item 
					\label{item_N_e}For any $\ind\in I^+$, $\eta<\gamma$ and set of (nice names of) reals $A$ in $V^{\p_\eta}$ of size $<\lambda_\ind$,
					there is some $\xi\in S_\ind$ (above $\eta$) such that $A\subseteq N_\xi$.
					\item If $\ind(\xi)=4$, then $\{\dot{D}^h_\xi:h\in H \}\subseteq N_\xi$.
					\item If $\ind(\xi)=5$, then $\{\dot{D}^h_\xi:h\in H \},\{\dot{E}^{h^\prime}_\xi:h^\prime\in H^\prime \}\subseteq N_\xi$.
					
				\end{enumerate}
				
				\item For $\xi\in S_0$, $\p_\xi^-\coloneq\p_\xi$ and for $\xi\in \gamma\setminus S_0$, $\p_\xi^-\coloneq\p_\xi\cap N_\xi$ (since $\p_\xi$ is ccc and $N_\xi$ is $\sigma$-closed, $\p_\xi^-\lessdot\p_\xi$).
				
				\item For each $\xi<\gamma$,
				$\p^-_\xi\Vdash\qd_\xi\coloneq\br_{\ind(\xi)}$. 
				
				(Here, $\br_{\ind(\xi)}$ does not denote a forcing poset in the ground model, but denotes the poset interpreted in the  $\p^-_\xi$-extension.
				Also note that $|\qd_\xi|<\lambda_{\ind(\xi)}$ holds since there are at most $|\p^-_\xi|^{\aleph_0}\leq|N_\xi|^{\aleph_0}<\lambda_{\ind(\xi)}$-many reals in the $\p^-_\xi$-extension. Moreover, for $\xi\in S_0$, $\qd_\xi$ is always (forced to be) the same Cohen forcing $\mathbb{C}$.)
				\item $\bar{Q}$ and $\bar{R}$ are determined in the canonical way:
				In the case of $\bar{Q}$, if $\ind(\xi)=0,1,2,3$, split $\qd_\xi$ into singletons and if $\ind(\xi)=4,5$, split the $\sigma$-$\Lambda^\mathrm{lim}_\mathrm{uf}$-linked iterand $\qd_\xi$ into the $\omega$-many $\Lambda^\mathrm{lim}_\mathrm{uf}$-linked components. In the case of $\bar{R}$, do it similarly.

				%\item If $\xi\in S_\ind$, then $\Vdash_{\p_\xi}|\qd_\xi|<\lambda_\ind$.

			\end{enumerate}

		\end{con}
		We explain why the construction is possible:
		
		\textit{Successor step.}
		At stage $\xi(\geq\lambda)$, we can take some $N_\xi$ satisfying Construction \ref{con_P6}\eqref{item_N},
		by Assumption \ref{ass_card_arith}\eqref{item_aleph1_inacc} for Construction \ref{con_P6}\eqref{item_N}\eqref{item_sigma_closed} and by Assumption \ref{ass_card_arith} \eqref{item_ca_3} for Construction \ref{con_P6}\eqref{item_N}\eqref{item_N_e}, the bookkeeping condition.
		%We may also require that By Lemma \ref{lem_D_minus_in_suff},
		%$\Vdash_{\p_\xi}(\dot{D}^h_\xi)^-, (\dot{E}^{h^\prime}_\xi)^-\in V^{\p_\xi^-}$ for all $h\in H$ and $h^\prime\in H^\prime$.
		By Lemma \ref{lem_uf_const_succ}, Corollary \ref{cor_E_is_centered_suff_cuf} and Corollary \ref{cor_PR_is_sigma_centered_suff}, we obtain suitable $\{\dot{D}^h_{\xi+1}:h\in H \}$ and $\{\dot{E}^{h^\prime}_{\xi+1}:h^\prime\in H^\prime \}$ (including the case $\ind(\xi)=0$): E.g., consider the most complicated case $\ind(\xi)=4$. Since $\xi\in S^+$, $(\dot{D}^h_\xi)^-=\dot{D}^h_\xi\cap V^{\p_\xi^-}$ and $\dot{D}^h_\xi\in N_\xi$. Hence by Lemma \ref{lem_D_minus_in_suff}, $(\dot{D}^h_\xi)^-\in V^{\p^-_\xi}$ holds. %On the other hand, since $\xi\in T^-$, $(\dot{E}^{h^\prime}_\xi)^-\in V^{\p^-_\xi}$ trivially holds.
		Thus, the assumption \eqref{eq_minus} for $\dot{D}^h_\xi$ is satisfied and since $\xi\in T^-$, Lemma \ref{lem_uf_const_succ} can be applied for both $\dot{D}^h_\xi$ and $\dot{E}^{h^\prime}_\xi$. The other cases are similar and simpler: If $\ind(\xi)=0,1,2,3$, then $\xi$ is trivial for both $\bar{D}$ and $\bar{E}$. If $\ind(\xi)=5$, then $\xi$ is non-trivial for both, but $\{\dot{D}^h_\xi:h\in H \},\{\dot{E}^{h^\prime}_\xi:h^\prime\in H^\prime \}\subseteq N_\xi$. Hence, Lemma \ref{lem_uf_const_succ} can be applied in any case. %Actually this is why we split trivial and nontrivial stages.)
		
		\textit{Limit step.}
		Direct from Lemma \ref{lem_uf_const_all}, Corollary \ref{cor_E_is_centered_suff_cuf} and Corollary \ref{cor_PR_is_sigma_centered_suff}.

		%To show that this construction can be done, it is enough to show that for each $\xi<\gamma$, if we have constructed all objects with index $\eta<\xi$, we can find $\langle\dot{D}_\xi^h:h\in H\rangle$, $\langle\dot{E}_\xi^h:h\in H\rangle$ and $N_\xi$, since $\langle\dot{Q}_{\xi,\zeta}:\zeta<\theta_\xi\rangle$ and $\langle\dot{R}_{\xi,\zeta}:\zeta<\theta_\xi\rangle$ are automatically determined. (In the case of $\langle\dot{Q}_{\xi,\zeta}:\zeta<\theta_\xi,\rangle$, if $\ind(\xi)=1,2,3$, split $\qd_\xi$ into singletons and if $\ind(\xi)=4,5$, split the $\sigma$-$\Lambda^\mathrm{lim}_\mathrm{uf}$-linked iterand $\qd_\xi$ into the $\omega$-many $\Lambda^\mathrm{lim}_\mathrm{uf}$-linked components. In the case of $\langle\dot{R}_{\xi,\zeta}:\zeta<\theta_\xi\rangle$, do it similarly.)
		%
		%For $\langle\dot{D}_\xi^h:h\in H\rangle$, 
		%use Lemma \ref{lem_uf_const_all}, Corollary \ref{cor_E_is_centered_suff_cuf} and Corollary \ref{cor_PR_is_sigma_centered_suff}.
		%
		%For $\langle\dot{E}_\xi^h:h\in H\rangle$, 
		%use Lemma \ref{lem_uf_const_all} and Corollary \ref{cor_E_is_centered_suff_cuf}.
		%
		%For $N_\xi$, use Assumption \ref{ass_card_arith}\eqref{item_aleph1_inacc} for Construction \ref{con_P6}\eqref{item_N}\eqref{item_sigma_closed} (and use bookkeeping argument for Construction \ref{con_P6}\eqref{item_N}\eqref{item_N_e}. Assumption \ref{ass_card_arith} \eqref{item_ca_3} enables the bookkeeping). 
		%
		Thus, we can perform the iteration construction.
		
		\begin{thm}
			\label{thm_p6_forces}
			$\pst_\mathrm{pre}$ forces for each $\ind\in I^+$,
			$\R_\ind\cong_T C_{[\lambda_6]^{<\lambda_\ind}}\cong_T[\lambda_6]^{<\lambda_\ind}$, in particular, $\bb(\R_\ind)=\lambda_\ind$ and $\dd(\R_\ind)=2^{\aleph_0}=\lambda_6$ (the same things also hold for $\R_4^*$)
			(see Figure \ref{fig_left_of_CMe}).
		\end{thm}
		
		\begin{proof}
			
			Fact \ref{fac_Tukey_order_equivalence_condition} and Construction \ref{con_P6} \eqref{item_N}\eqref{item_N_e} imply $\R_\ind\lq C_{[\lambda_6]^{<\lambda_\ind}}$ for all $\ind\in I^+$.
			
			For $\ind=1,2$ and $5$, applying Corollary \ref{cor_smallness_for_addn_and_covn_and_nonm} with $\theta=\lambda_\ind$, 
			we obtain $C_{[\lambda_6]^{<\lambda_\ind}}\lq\R_\ind$.
			
			For $\ind=3$, Theorem \ref{thm_uf_limit_keeps_b_small} implies $C_{[\lambda_6]^{<\lambda_3}}\lq\R_3$.
			
			For $\ind=4$, Main Lemma \ref{lem_Main_Lemma} implies $C_{[\lambda_6]^{<\lambda_4}}\lq\R_4^*\lq\R_4$.
		\end{proof}

		\begin{figure}
			
			\centering
			\begin{tikzpicture}
				\tikzset{
					textnode/.style={text=black}, 
				}
				\tikzset{
					edge/.style={color=black, thin, opacity=0.5}, 
				}
				\newcommand{\w}{2.4}
				\newcommand{\h}{1.5}
				
				\node[textnode] (addN) at (0,  0) {$\addn$};
				\node (l1) [fill=lime, draw, text=black, circle,inner sep=1.0pt] at (-0.2*\w, 0.7*\h) {$\lambda_1$};
				
				\node[textnode] (covN) at (0,  \h*3) {$\covn$};
				\node (l2) [fill=lime, draw, text=black, circle,inner sep=1.0pt] at (0.18*\w, 2.5*\h) {$\lambda_2$};

				\node[textnode] (addM) at (\w,  0) {$\cdot$};
				\node[textnode] (b) at (\w,  1*\h) {$\bb$};
				\node (l3) [fill=lime, draw, text=black, circle,inner sep=1.0pt] at (0.8*\w, 0.8*\h) {$\lambda_3$};
				
				\node[textnode] (nonM) at (\w,  \h*3) {$\nonm$};
				\node (l5) [fill=lime, draw, text=black, circle,inner sep=1.0pt] at (1.35*\w, 3.35*\h) {$\lambda_5$};
				
				\node[textnode] (covM) at (\w*2,  0) {$\covm$};

				\node[textnode] (d) at (\w*2,  2*\h) {$\dd$};

				\node[textnode] (cofM) at (\w*2,  \h*3) {$\cdot$};

				\node[textnode] (nonN) at (\w*3,  0) {$\nonn$};

				\node[textnode] (cofN) at (\w*3,  \h*3) {$\cofn$};

				\node[textnode] (aleph1) at (-\w,  0) {$\aleph_1$};
				\node[textnode] (c) at (\w*4,  \h*3) {$2^{\aleph_0}$};
				\node (l6) [fill=lime, draw, text=black, circle,inner sep=2.0pt] at (3.5*\w, 1.5*\h) {\scalebox{3.0}{$\lambda_6$}};%3.5 1.5 \Huge
				
				\node[textnode] (e) at (0.5*\w,  2*\h) {$\mathfrak{e}$}; 
				\node[textnode] (estar) at (\w,  2*\h) {$\ee^*$};
				\node (l4) [fill=lime, draw, text=black, circle,inner sep=1.0pt] at (1.18*\w, 2.35*\h) {$\lambda_4$};
				
				\node[textnode] (pr) at (2.5*\w,  1*\h) {$\mathfrak{pr}$}; 
				\node[textnode] (prstar) at (\w*2,  1*\h) {$\mathfrak{pr}^*$};

				\draw[->, edge] (addN) to (covN);
				\draw[->, edge] (addN) to (addM);
				\draw[->, edge] (covN) to (nonM);	
				\draw[->, edge] (addM) to (b);
				
				\draw[->, edge] (addM) to (covM);
				\draw[->, edge] (nonM) to (cofM);
				
				\draw[->, edge] (d) to (cofM);
				\draw[->, edge] (b) to (prstar);
				\draw[->, edge] (covM) to (nonN);
				\draw[->, edge] (cofM) to (cofN);
				\draw[->, edge] (nonN) to (cofN);
				\draw[->, edge] (aleph1) to (addN);
				\draw[->, edge] (cofN) to (c);

				\draw[->, edge] (e) to (covM);
				\draw[->, edge] (addN) to (e);

				\draw[->, edge] (covM) to (prstar);
				\draw[->, edge] (nonM) to (pr);
				\draw[->, edge] (pr) to (cofN);
				
				\draw[->, edge] (e) to (estar);
				\draw[->, edge] (b) to (estar);
				\draw[->, edge] (estar) to (nonM);
				\draw[->, edge] (estar) to (d);
				
				\draw[->, edge] (prstar) to (d);
				\draw[->, edge] (prstar) to (pr);
				
				\draw[->, edge] (prstar) to (pr);

				\draw[blue,thick] (0.35*\w,2.6*\h)--(1.5*\w,2.6*\h);
				\draw[blue,thick] (0.35*\w,1.5*\h)--(0.35*\w,2.6*\h);

				\draw[blue,thick] (1.5*\w,-0.5*\h)--(1.5*\w,3.5*\h);
				\draw[blue,thick] (0.5*\w,2.6*\h)--(0.5*\w,3.5*\h);

				\draw[blue,thick] (0.5*\w,-0.5*\h)--(0.5*\w,1.5*\h);
				
				\draw[blue,thick] (-0.5*\w,-0.5*\h)--(-0.5*\w,3.5*\h);
				\draw[blue,thick] (-0.5*\w,1.5*\h)--(1.5*\w,1.5*\h);

			\end{tikzpicture}
			\caption{Constellation of $\pst_\mathrm{pre}$ and $\pst_\mathrm{mid}$.}\label{fig_left_of_CMe}
		\end{figure}
		
		\subsection{Recovery of GCH}
		To apply the submodel method to separate the right side, we actually need some cardinal arithmetic (see \cite{GKMS}, \cite{forcing_constellations}), which is satisfied (particularly) under GCH, but \textit{conflicts} with Assumption \ref{ass_for_complete_guardrail}, which we have used for the completeness of the sets $H$ and $H^\prime$ of guardrails.
		To avoid the conflict, we shall reconstruct the iteration under GCH. The idea is as follows: By considering some extension model where Assumption \ref{ass_for_complete_guardrail} is satisfied, we construct $\pst_\mathrm{pre}$ there. At the same time, we mimic the construction in the \textit{ground model} and obtain an iteration $\pst_\mathrm{mid}$ (in the ground model), which actually forces the same separation.
		While the argument of recovering GCH is already described in \cite{GKS},
		we give an explanation below for the sake of completeness\footnote{There is a slight difference between their construction and ours in the sense that they use \textit{groundmodel-code-sequences} while we use submodels $N_\xi$, but the ideas are essentially the same, so are the arguments of the recovery of GCH.}.

		%We can summarize the previous subsection as follows:
		%\begin{cor}
		%	\label{cor_p6_constructible}
		%	Under Assumption \ref{ass_card_arith} and \ref{ass_for_complete_guardrail}, 
		%	we can inductively construct a fsi $\pst_\mathrm{mid}=\langle\p_\xi,\qd_\xi\rangle_{\xi<\gamma}$ and submodels $\langle N_\xi:\xi<\gamma\rangle$ such that:
		%	\begin{enumerate}
			%		\item Each $N_\xi$ is $\sigma$-closed and has size $<\lambda_{\ind(\xi)}$.
			%		%\item For any $\ind\in I$, $\zeta<\gamma$ and set of (nice names of) reals $\dot{A}$ in $V^{\p_\zeta}$ of size $<\lambda_\ind$,
			%		%there is some $\xi\in S_\ind$ (above $\eta$) such that $\dot{A}\subseteq N_\xi$.
			%		\item For each $\xi<\gamma$,
			%		$\p^-_{\xi}\Vdash \qd_\xi=\br_{\ind(\xi)}$ where $\p^-_{\xi}\coloneq\p_{\xi}\cap N_\xi$.
			%		\item $\pst_\mathrm{mid}$ forces for each $\ind\in I$,
			%		$\R_\ind\cong_T C_{[\lambda_6]^{<\lambda_\ind}}\cong_T[\lambda_6]^{<\lambda_\ind}$.
			%	\end{enumerate}
		%
		%	
		%\end{cor} 
		%We distinguish $\pst_\mathrm{pre}$ and $\pst_\mathrm{mid}$ in the sense that in Corollary \ref{cor_p6_constructible}, we do not mention $\langle\dot{Q}_{\xi,\zeta}:\zeta<\theta_\xi,\xi<\gamma\rangle$, 
		%$\langle \dot{D}^h_\xi:\xi\leq\gamma,~h\in H \rangle$,
		%$\langle\dot{R}_{\xi,\zeta}:\zeta<\theta_\xi,\xi<\gamma\rangle$ or 
		%$\langle \dot{E}^h_\xi:\xi\leq\gamma,~h\in H^\prime\rangle$.

		Assume GCH hereafter and Assumption \ref{ass_card_arith}.
		
		\begin{con}
			\label{con_mid}
			We shall construct a fsi $\pst_\mathrm{mid}=\langle\p_\xi,\qd_\xi\rangle_{\xi<\gamma}$ and submodels $\langle N_\xi:\xi\in\gamma\setminus S_0\rangle$ such that $\Vdash_{\p_\xi}\qd_\xi\coloneq\mathbb{C}$ for $\xi\in S_0$ and for $\xi\in\gamma\setminus S_0$:
			\begin{enumerate}
				\item $|N_\xi|<\lambda_{\ind(\xi)}$.
				\item %\label{item_sigma_closed}
				% $N_\xi$ is closed under $\omega$-sequences 
				$N_\xi$ is $\sigma$-closed.
				\item 
				%\label{item_N_e}
				For any $\ind\in I^+$, $\eta<\gamma$ and set of (nice names of) reals $A$ in $V^{\p_\eta}$ of size $<\lambda_\ind$,
				there is some $\xi\in S_\ind$ (above $\eta$) such that $A\subseteq N_\xi$.
				
				%is $\sigma$-closed and has size $<\lambda_{\ind(\xi)}$.
				%\item For any $\ind\in I^+$, $\zeta<\gamma$ and set of (nice names of) reals $\dot{A}$ in $V^{\p_\zeta}$ of size $<\lambda_\ind$,
				%there is some $\xi\in S_\ind$ (above $\eta$) such that $\dot{A}\subseteq N_\xi$.
				\item %For each $\xi<\gamma$,
				$\p^-_{\xi}\Vdash \qd_\xi=\br_{\ind(\xi)}$ where $\p^-_{\xi}\coloneq\p_{\xi}\cap N_\xi$.
			\end{enumerate}
		\end{con}
		
		Let $\Ad\coloneq\mathrm{Add}(\mu_3,\lambda_6)=\bigcup\{2^A:A\in[\mu_3\times\lambda_6]^{<\mu_3}\}$.
		Since $\mu_3$ is regular,
		$\Ad$ is $<\mu_3$-closed and since $\mu_3^{<\mu_3}=\mu_3$ (by GCH), 
		it is $\mu_3^+$-cc (by standard $\Delta$-system argument).
		Note that it forces $2^{\mu_3}=\lambda_6$ and hence Assumption \ref{ass_for_complete_guardrail} is satisfied in the $\Ad$-extension (and so is Assumption \ref{ass_card_arith}). Hence, Construction \ref{con_P6} and Theorem \ref{thm_p6_forces} hold in the $\Ad$-extension, even if $H_\Theta$ is replaced with $M\coloneq H_\Theta^V$ (Since $\Theta$ is sufficiently large, $M$ contains all we shall require even in the $\Ad$-extention. Hereafter whenever considering Construction \ref{con_P6} in the $\Ad$-extension, we assume this replacement).

		%Actually we can (inductively) construct $\pst_\mathrm{mid}$ instead of $\pst_\mathrm{pre}$ satisfying Corollary \ref{cor_p6_constructible} also in the \textit{ground model}, without Assumption \ref{ass_for_complete_guardrail}, relying on the $\Ad$-extension. 
		
		We explain why Construction \ref{con_mid} is possible.
		We inductively construct $\pst_\mathrm{mid}$ in the ground model and $\pst_\mathrm{pre}$ in the $\Ad$-extension simultaneously, both of which share the same $\bar{N}=\langle N_\xi:\xi\in\gamma\setminus S_0\rangle$.
		Note that since $\Ad$ is $<\mu_3$-closed and $\mu_3^+$-cc, $[\mathbb{R}]^{<\lambda_\ind}\cap V$ is cofinal in $[\mathbb{R}]^{<\lambda_\ind}\cap V^\Ad$ for all $\ind \in I^+$, ($\mathbb{R}$ denotes the (same) set of all reals) and hence the bookkeeping conditions are the same. 
		Assume that $\xi\in\gamma\setminus S_0$ and we have constructed in the ground model $\langle N_\eta:\eta\in\xi\setminus S_0\rangle$ and in the $\Ad$-extension all objects of $\pst_\mathrm{pre}$ with parameters $\eta<\xi$.
		If $\ind(\xi)=1,2,3$, since what we require for $N_\xi$ are the same, we can pick some $N_\xi$ which is suitable for both of the constructions.
		If $\ind(\xi)=4,5$, in the $\Ad$-extension $N_\xi$ has to contain names of ultrafilters which are only in the extension.
		We first obtain an $\Ad$-name $\dot{N}^0_\xi$ with the suitable properties in Construction \ref{con_P6}\eqref{item_N}.
		%If $\xi\in S_1\cup S_2$, since $\Ad$ is $<\mu_3$-closed and hence no new $<\lambda_2$-sequence is added, $\dot{N}_\xi$ is actually in the ground model, so let $N_\xi\coloneq\dot{N}_\xi$. 
		Since $\Vdash_\Ad \dot{N}^0_\xi \subseteq M\subseteq V $ and $\Ad$ is $\mu_3^+$-cc, we can obtain (in the ground model) a $\sigma$-closed submodel $N^1_\xi\supseteq\dot{N}^0_\xi$ of size $<\lambda_{\ind(\xi)}$.
		Again getting into the $\Ad$-extension, we obtain a ($\Ad$-name of) $\sigma$-closed submodel $\dot{N}^2_\xi\supseteq N_\xi^1$ of size $<\lambda_{\ind(\xi)}$ with the suitable properties in Construction \ref{con_P6}\eqref{item_N}. Again, in the ground model, we can obtain a $\sigma$-closed submodel $N^3_\xi\supseteq\dot{N}^2_\xi$ of size $<\lambda_{\ind(\xi)}$. Continuing this way $\omega_1$-many times (At limit steps, take the union of all the previous (names of) submodels.), we ultimately obtain $N_\xi\coloneq N^{\omega_1}_\xi$ in the ground model, which satisfies the suitable properties in the $\Ad$-extension.
		In the extension, $\pst_\mathrm{mid}$ and $\pst_\mathrm{pre}$ are essentially the same since they share the same $\bar{N}$, which determines their structures.
		
		We show that $\pst_\mathrm{mid}$ also forces the consequence of Theorem \ref{thm_p6_forces}.
		By a bookkeeping argument, $\R_\ind\lq C_{[\lambda_6]^{<\lambda_\ind}}$ holds for each $\ind\in I^+$.
		For $\ind=1,2$ and $5$,
		we similarly obtain $C_{[\lambda_6]^{<\lambda_\ind}}\lq\R_\ind$.
		For $\ind=3,4$ by Fact \ref{fac_Tukey_order_equivalence_condition},
		it is enough to show that the first $\lambda_6$-many Cohen reals $\langle\dot{c}_\b:\b<\lambda_6\rangle$ witness that every response $\dot{y}$ meets only $<\lambda_{\ind}$-many $\dot{c}_\b$. Given such a $\pst_\mathrm{mid}$-name $\dot{y}$, working in the $\Ad$-extension and interpreting $\dot{y}$ as a $\pst_\mathrm{pre}$-name, we obtain some $B\in[\lambda_6]^{<\lambda_\ind}$ such that for any $\b\in\lambda_6\setminus B$, $\dot{c}_\b$ is not met by $\dot{y}$. Since $\Ad$ is $\mu_3^+$-cc, we can find such $B$ in the ground model. Now, for any $\b\in\lambda_6\setminus B$, $\Vdash_\Ad\Vdash_{\pst_\mathrm{pre}}$ ``$\dot{c}_\b$ is not met by $\dot{y}$'' and by absoluteness, $\Vdash_{\pst_\mathrm{mid}}$ ``$\dot{c}_\b$ is not met by $\dot{y}$'' and we are done. Hence we obtain the following theorem:
		\begin{thm}
			\label{thm_p6_GCH}
			Assume GCH and Assumption \ref{ass_card_arith}.
			Then, there exists a ccc poset $\pst_\mathrm{mid}$ which forces for each $\ind\in I^+$,
			$\R_\ind\cong_T C_{[\lambda_6]^{<\lambda_\ind}}\cong_T[\lambda_6]^{<\lambda_\ind}$, in particular, $\bb(\R_\ind)=\lambda_\ind$ and $\dd(\R_\ind)=2^{\aleph_0}=\lambda_6$ (the same things also hold for $\R_4^*$).
		\end{thm}
		
		\subsection{Separation of the right side}
		Thanks to Theorem \ref{thm_p6_GCH}, we are now in the situation where we can apply the \textit{submodel method}, which was introduced in \cite{GKMS} and enables us to separate the right side of the diagram.
		
		%\begin{thm}
		%	Assume GCH and 
		%	$\aleph_1\leq\theta_1\leq\theta_2\leq\theta_3\leq\theta_4\leq\theta_5\leq\kappa_5\leq\kappa_4\leq\kappa_3\leq\kappa_2\leq\kappa_1$ are regular and $\kappa_\cc\geq\kappa_1$ satisfies $\kappa_\cc^{\aleph_0}=\kappa_\cc$.
		%	Then, there exists a ccc poset $\p^{\mathrm{fin}}$ which forces $\bb(\R_\ind)=\theta_\ind$, $\dd(\R_\ind)=\kappa_\ind$ and $\cc=\kappa_\cc$.
		%\end{thm}
		\begin{thm}
			\label{thm_p6fin}
			Assume GCH and 
			$\aleph_1\leq\theta_1\leq\cdots\leq\theta_{10}$ are regular and $\theta_\cc$ is an infinite cardinal such that $\theta_\cc\geq\theta_{10}$ and $\theta_\cc^{\aleph_0}=\theta_\cc$.
			Then, there exists a ccc poset $\pst_{\mathrm{fin}}$ which forces $\bb(\R_\ind)=\theta_\ind$ and $\dd(\R_\ind)=\theta_{11-\ind}$ for each $\ind\in I^+$ (the same things also hold for $\R_4^*$) and $\cc=\theta_\cc$ (see Figure \ref{fig_right_of_CMe}).
		\end{thm}
		
		\begin{proof}
			See \cite{GKMS}.
		\end{proof}
		\subsection{Controlling $\ee_{ubd}$}
		Toward the proof of Theorem C, where $g$-prediction (Definition \ref{dfn_variant}) is treated in the separation, let us consider $\ee_g$ for $g\in(\omega\setminus2)^\omega$.
		
		\begin{dfn}
			$\R_5^g\coloneq\mathbf{PR}_g$ for $g\in(\omega\setminus2)^\omega$.
		\end{dfn}
		By Corollary \ref{cor_PR_is_sigma_centered_suff}, the poset $\mathbb{PR}_g$, which increases $\ee_g$, is $\sigma$-c-uf-lim-linked.
		Thus, performing an iteration where the $\sigma$-c-uf-lim-linked forcing $\mathbb{E}$ is replaced with $\mathbb{PR}_g$ and where $g$ runs through all $g\in(\omega\setminus2)^\omega$ by bookkeeping, we obtain the following:%\footnote{$\pst^\prime_\mathrm{mid}$ is the poset constructed after the recovery of GCH.}:
		
		\begin{thm}(Corresponding to Theorem \ref{thm_p6_GCH}, the construction of $\pst_\mathrm{mid}$.)
			\label{thm_p_mid_prime_GCH}
			Assume GCH and Assumption \ref{ass_card_arith}.
			Then, there exists a ccc poset $\pst^\prime_\mathrm{mid}$ which forces for each $\ind\in I^+$,
			$\R_\ind\cong_T C_{[\lambda_6]^{<\lambda_\ind}}\cong_T[\lambda_6]^{<\lambda_\ind}$, in particular, $\bb(\R_\ind)=\lambda_\ind$ and $\dd(\R_\ind)=2^{\aleph_0}=\lambda_6$ (the same things also hold for $\R_4^*$ and $\R_5^g$ for $g\in(\omega\setminus2)^\omega$).
		\end{thm}
		
		\begin{thm}(Corresponding to Theorem \ref{thm_p6fin}, the construction of $\pst_\mathrm{fin}$.)
			\label{thm_p_fin_prime_GCH}
			Assume GCH and 
			$\aleph_1\leq\theta_1\leq\cdots\leq\theta_{10}$ are regular and $\theta_\cc$ is infinite cardinal such that $\theta_\cc\geq\theta_{10}$ and $\theta_\cc^{\aleph_0}=\theta_\cc$.
			Then, there exists a ccc poset $\pst_{\mathrm{fin}}$ which forces $\bb(\R_\ind)=\theta_\ind$ and $\dd(\R_\ind)=\theta_{11-\ind}$ for each $\ind\in I^+$ (the same things also hold for $\R_4^*$ and $\R_5^g$ for $g\in(\omega\setminus2)^\omega$) and $\cc=\theta_\cc$ (see Figure \ref{fig_cm_prime}).
		\end{thm}

		\begin{figure}
			\centering
			\begin{tikzpicture}
				\tikzset{
					textnode/.style={text=black}, 
				}
				\tikzset{
					edge/.style={color=black, thin, opacity=0.5}, 
				}
				\newcommand{\w}{2.4}
				\newcommand{\h}{1.5}
				
				\node[textnode] (addN) at (0,  0) {$\addn$};
				\node (t1) [fill=lime, draw, text=black, circle,inner sep=1.0pt] at (-0.2*\w, 0.7*\h) {$\theta_1$};
				
				\node[textnode] (covN) at (0,  \h*3) {$\covn$};
				\node (t2) [fill=lime, draw, text=black, circle,inner sep=1.0pt] at (0.18*\w, 2.5*\h) {$\theta_2$};

				\node[textnode] (addM) at (\w,  0) {$\cdot$};
				\node[textnode] (b) at (\w,  1*\h) {$\bb$};
				\node (t3) [fill=lime, draw, text=black, circle,inner sep=1.0pt] at (0.8*\w, 0.8*\h) {$\theta_3$};
				
				\node[textnode] (nonM) at (\w,  \h*3) {$\nonm$};
				\node (t5) [fill=lime, draw, text=black, circle,inner sep=1.0pt] at (1.35*\w, 3.35*\h) {$\theta_5$};
				
				\node[textnode] (covM) at (\w*2,  0) {$\covm$};
				\node (t6) [fill=lime, draw, text=black, circle,inner sep=1.0pt] at (1.65*\w, -0.4*\h) {$\theta_6$};
				
				\node[textnode] (d) at (\w*2,  2*\h) {$\dd$};
				\node (t8) [fill=lime, draw, text=black, circle,inner sep=1.0pt] at (2.2*\w, 2.2*\h) {$\theta_8$};
				
				\node[textnode] (cofM) at (\w*2,  \h*3) {$\cdot$};

				\node[textnode] (nonN) at (\w*3,  0) {$\nonn$};
				\node (t9) [fill=lime, draw, text=black, circle,inner sep=1.0pt] at (2.82*\w, 0.5*\h) {$\theta_9$};
				
				\node[textnode] (cofN) at (\w*3,  \h*3) {$\cofn$};
				\node (t10) [fill=lime, draw, text=black, circle,inner sep=1.0pt] at (3.2*\w, 2.3*\h) {$\theta_{10}$};
				
				\node[textnode] (aleph1) at (-\w,  0) {$\aleph_1$};
				\node[textnode] (c) at (\w*4,  \h*3) {$2^{\aleph_0}$};
				\node (tcc) [fill=lime, draw, text=black, circle,inner sep=1.0pt] at (3.65*\w, 3.5*\h) {$\theta_\cc$};
				
				\node[textnode] (e) at (0.5*\w,  2*\h) {$\mathfrak{e}$}; 
				\node[textnode] (estar) at (\w,  2*\h) {$\ee^*$};
				\node (t4) [fill=lime, draw, text=black, circle,inner sep=1.0pt] at (1.18*\w, 2.35*\h) {$\theta_4$};
				
				\node[textnode] (pr) at (2.5*\w,  1*\h) {$\mathfrak{pr}$}; 
				\node[textnode] (prstar) at (\w*2,  1*\h) {$\mathfrak{pr}^*$};
				\node (t7) [fill=lime, draw, text=black, circle,inner sep=1.0pt] at (1.82*\w, 0.65*\h) {$\theta_7$};

				\draw[->, edge] (addN) to (covN);
				\draw[->, edge] (addN) to (addM);
				\draw[->, edge] (covN) to (nonM);	
				\draw[->, edge] (addM) to (b);
				
				\draw[->, edge] (addM) to (covM);
				\draw[->, edge] (nonM) to (cofM);
				
				\draw[->, edge] (d) to (cofM);
				\draw[->, edge] (b) to (prstar);
				\draw[->, edge] (covM) to (nonN);
				\draw[->, edge] (cofM) to (cofN);
				\draw[->, edge] (nonN) to (cofN);
				\draw[->, edge] (aleph1) to (addN);
				\draw[->, edge] (cofN) to (c);

				\draw[->, edge] (e) to (covM);
				\draw[->, edge] (addN) to (e);

				\draw[->, edge] (covM) to (prstar);
				\draw[->, edge] (nonM) to (pr);
				\draw[->, edge] (pr) to (cofN);
				
				\draw[->, edge] (e) to (estar);
				\draw[->, edge] (b) to (estar);
				\draw[->, edge] (estar) to (nonM);
				\draw[->, edge] (estar) to (d);
				
				\draw[->, edge] (prstar) to (d);
				\draw[->, edge] (prstar) to (pr);
				
				\draw[->, edge] (prstar) to (pr);

				\draw[blue,thick] (1.5*\w,0.4*\h)--(2.65*\w,0.4*\h);
				\draw[blue,thick] (2.65*\w,0.4*\h)--(2.65*\w,1.5*\h);
				\draw[blue,thick] (-0.5*\w,1.5*\h)--(3.5*\w,1.5*\h);
				\draw[blue,thick] (0.35*\w,2.6*\h)--(1.5*\w,2.6*\h);
				\draw[blue,thick] (0.35*\w,1.5*\h)--(0.35*\w,2.6*\h);

				\draw[blue,thick] (1.5*\w,-0.4*\h)--(1.5*\w,3.5*\h);
				\draw[blue,thick] (0.5*\w,2.6*\h)--(0.5*\w,3.5*\h);
				\draw[blue,thick] (2.5*\w,1.5*\h)--(2.5*\w,3.5*\h);

				\draw[blue,thick] (0.5*\w,-0.4*\h)--(0.5*\w,1.5*\h);
				\draw[blue,thick] (2.5*\w,-0.4*\h)--(2.5*\w,0.4*\h);
				\draw[blue,thick] (-0.5*\w,-0.4*\h)--(-0.5*\w,3.5*\h);
				\draw[blue,thick] (3.5*\w,-0.4*\h)--(3.5*\w,3.5*\h);

			\end{tikzpicture}
			\caption{Constellation of $\pst_\mathrm{fin}$.}\label{fig_right_of_CMe}
		\end{figure}
		
		\begin{figure}
			\centering
			\begin{tikzpicture}
				\tikzset{
					textnode/.style={text=black}, 
				}
				\tikzset{
					edge/.style={color=black, thin, opacity=0.5}, 
				}
				\newcommand{\w}{2.4}
				\newcommand{\h}{2.0}
				
				\node[textnode] (addN) at (0,  0) {$\addn$};
				\node (t1) [fill=lime, draw, text=black, circle,inner sep=1.0pt] at (-0.2*\w, 0.8*\h) {$\theta_1$};
				
				\node[textnode] (covN) at (0,  \h*3) {$\covn$};
				\node (t2) [fill=lime, draw, text=black, circle,inner sep=1.0pt] at (0.15*\w, 2.45*\h) {$\theta_2$};

				\node[textnode] (addM) at (\w,  0) {$\cdot$};
				\node[textnode] (b) at (\w,  1.3*\h) {$\bb$};
				\node (t3) [fill=lime, draw, text=black, circle,inner sep=1.0pt] at (0.68*\w, 1.2*\h) {$\theta_3$};
				
				\node[textnode] (nonM) at (\w,  \h*3) {$\nonm$};
				\node (t5) [fill=lime, draw, text=black, circle,inner sep=1.0pt] at (1.0*\w, 3.3*\h) {$\theta_5$};
				
				\node[textnode] (covM) at (\w*2,  0) {$\covm$};
				\node (t6) [fill=lime, draw, text=black, circle,inner sep=1.0pt] at (2.0*\w, -0.3*\h) {$\theta_6$};
				
				\node[textnode] (d) at (\w*2,  1.7*\h) {$\dd$};
				\node (t8) [fill=lime, draw, text=black, circle,inner sep=1.0pt] at (2.32*\w, 1.8*\h) {$\theta_{8}$};
				\node[textnode] (cofM) at (\w*2,  \h*3) {$\cdot$};

				\node[textnode] (nonN) at (\w*3,  0) {$\nonn$};
				\node (t9) [fill=lime, draw, text=black, circle,inner sep=1.0pt] at (2.85*\w, 0.55*\h) {$\theta_{9}$};
				
				\node[textnode] (cofN) at (\w*3,  \h*3) {$\cofn$};
				\node (t10) [fill=lime, draw, text=black, circle,inner sep=1.0pt] at (3.25*\w, 2.2*\h) {$\theta_{10}$};
				
				\node[textnode] (aleph1) at (-\w,  0) {$\aleph_1$};
				\node[textnode] (c) at (\w*4,  \h*3) {$2^{\aleph_0}$};
				\node (t10) [fill=lime, draw, text=black, circle,inner sep=1.0pt] at (3.67*\w, 3.4*\h) {$\theta_\cc$};
				
				\node[textnode] (e) at (0.5*\w,  1.7*\h) {$\mathfrak{e}$}; 
				\node[textnode] (estar) at (\w,  1.7*\h) {$\ee^*$};
				\node (t4) [fill=lime, draw, text=black, circle,inner sep=1.0pt] at (0.75*\w, 1.9*\h) {$\theta_4$};

				\node[textnode] (pr) at (2.5*\w,  1.3*\h) {$\mathfrak{pr}$}; 
				\node[textnode] (prstar) at (\w*2,  1.3*\h) {$\mathfrak{pr}^*$};
				\node (t9) [fill=lime, draw, text=black, circle,inner sep=1.0pt] at (2.25*\w, 1.1*\h) {$\theta_7$};
				
				\node[textnode] (eubd) at (\w*0.75,  \h*2.45) {$\ee_{ubd}$};
				%\node[textnode] (nonE) at (\w*0.75,  \h*2.55) {$\non(\mathcal{E})$};
				%\node (t5) [fill=lime, draw, text=black, circle,inner sep=1.0pt] at (0.5*\w, 2.1*\h) {$\theta_5$};%0.5 2.1
				
				\node[textnode] (prubd) at (\w*2.25,  \h*0.55) {$\mathfrak{pr}_{ubd}$};
				%\node[textnode] (covE) at (\w*2.25,  \h*0.45) {$\cov(\mathcal{E})$};
				%\node (t8) [fill=lime, draw, text=black, circle,inner sep=1.0pt] at (2.47*\w, 0.9*\h) {$\theta_8$};

				\draw[->, edge] (addN) to (covN);
				\draw[->, edge] (addN) to (addM);
				\draw[->, edge] (covN) to (nonM);	
				\draw[->, edge] (addM) to (b);
				%\draw[->, edge] (b) to (nonM);
				\draw[->, edge] (addM) to (covM);
				\draw[->, edge] (nonM) to (cofM);
				%\draw[->, edge] (covM) to (d);
				\draw[->, edge] (d) to (cofM);
				\draw[->, edge] (b) to (prstar);
				\draw[->, edge] (covM) to (nonN);
				\draw[->, edge] (cofM) to (cofN);
				\draw[->, edge] (nonN) to (cofN);
				\draw[->, edge] (aleph1) to (addN);
				\draw[->, edge] (cofN) to (c);
				
				%\draw[->, edge] (e) to (nonM);
				\draw[->, edge] (e) to (covM);
				\draw[->, edge] (addN) to (e);
				
				\draw[->, edge] (covM) to (prstar);
				\draw[->, edge] (covM) to (prubd);
				\draw[->, edge] (nonM) to (pr);
				\draw[->, edge] (covN) to (prubd);
				\draw[->, edge] (pr) to (cofN);
				
				\draw[->, edge] (e) to (estar);
				\draw[->, edge] (b) to (estar);C
				\draw[->, edge] (estar) to (nonM);
				\draw[->, edge] (eubd) to (nonN);
				\draw[->, edge] (estar) to (d);
				\draw[->, edge] (e) to (eubd);
				
				\draw[->, edge] (prstar) to (d);
				\draw[->, edge] (prstar) to (pr);
				
				\draw[->, edge] (prstar) to (pr);
				\draw[->, edge] (prubd) to (pr);
				\draw[->, edge] (eubd) to (nonM);
				%		\
				%		
				%		\draw[->, edge] (eubd) to (nonE);
				%		\draw[->, edge] (addM) to (nonE);
				%		\draw[->, edge] (nonE) to (nonM);
				%		
				%		\draw[->, edge] (covE) to (prubd);
				%		\draw[->, edge] (covE) to (cofM);
				%		\draw[->, edge] (covM) to (covE);
				
				\draw[blue,thick] (-0.5*\w,1.5*\h)--(3.5*\w,1.5*\h);
				\draw[blue,thick] (1.5*\w,-0.5*\h)--(1.5*\w,3.5*\h);
				
				\draw[blue,thick] (-0.5*\w,-0.5*\h)--(-0.5*\w,3.5*\h);
				\draw[blue,thick] (3.5*\w,-0.5*\h)--(3.5*\w,3.5*\h);
				
				\draw[blue,thick] (0.5*\w,-0.5*\h)--(0.5*\w,1.5*\h);
				\draw[blue,thick] (2.5*\w,1.5*\h)--(2.5*\w,3.5*\h);
				
				\draw[blue,thick] (0.3*\w,2.25*\h)--(1.5*\w,2.25*\h);
				%\draw[blue,thick] (-0.3*\w,2.7*\h)--(1.5*\w,2.7*\h);
				\draw[blue,thick] (0.3*\w,1.5*\h)--(0.3*\w,2.25*\h);
				\draw[blue,thick] (0.5*\w,2.25*\h)--(0.5*\w,3.5*\h);
				
				\draw[blue,thick] (2.7*\w,0.8*\h)--(1.5*\w,0.8*\h);
				%\draw[blue,thick] (3.1*\w,0.3*\h)--(1.5*\w,0.3*\h);
				\draw[blue,thick] (2.7*\w,1.5*\h)--(2.7*\w,0.8*\h);
				\draw[blue,thick] (2.5*\w,0.8*\h)--(2.5*\w,-0.5*\h);

			\end{tikzpicture}
			\caption{Constellation of $\p^\prime_\mathrm{fin}$.}\label{fig_cm_prime}
		\end{figure}

		\section{Questions} \label{sec_question}
		%\subsection{Question on closed-ultrafilter-limit}
		\begin{que}
			Are there other cardinal invariants which are not below $\eeb$ and kept small through forcings with closed-ultrafilter-limits?
		\end{que}
		In the left side of Cicho\'n's diagram, many cardinal invariants are either below $\eeb$ or above $\ee_{ubd}$ (hence closed-ultrafilter-limits do not keep them small) and a remaining candidate is 
		$\covn$. However, not only it is unclear whether c-uf-limits keep it small, but also even if they did, it would be unclear whether there would be an application since most of the known forcings with c-uf-limits are either $\sigma$-centered or sub-random, which keep $\covn$ small without resorting to c-uf-limits.
		
		\begin{que}
			Can the closedness argument in Main Lemma \ref{lem_Main_Lemma} be generalized to some fact such as ``closed-Fr-Knaster posets preserve strongly $\kappa$-$\mathbf{PR}^*$-unbounded families from the ground model'' described in \cite{BCM_filter_linkedness}?
		\end{que}
		
		The fact that ultrafilter limits keep $\bb$ small (Theorem \ref{thm_uf_limit_keeps_b_small}) is generalized to the fact ``For regular uncountable $\kappa$, $\kappa$-Fr-Knaster posets preserve strongly $\kappa$-$\mathbf{D}$-unbounded families from the ground model'' (\cite[Theorem 3.12.]{BCM_filter_linkedness}),
		where:
		\begin{itemize}
			\item Fr denotes the following linkedness notion: ``$Q\subseteq\p$ is Fr-linked if there is $\lim\colon Q^\omega\to\p$ such that $\lim\langle q_m\rangle_{m<\omega}\Vdash_\p |\{m<\omega:q_m\in\dot{G}\}|=\omega$ for any $\langle q_m\rangle_{m<\omega}\in Q^\omega$'',
			\item a $\kappa$-Fr-Knaster poset is a poset such that any family of conditions of size $\kappa$ has a Fr-linked subfamily of size $\kappa$, and
			\item a strongly $\kappa$-$\mathbf{D}$-unbounded family is a family of size $\geq\kappa$ such that any real can only dominate $<\kappa$-many reals in the family.
		\end{itemize}
		We are naturally interested in the possibility of this kind of generalization for closed-ultrafilter-limits. However, the proof of Main Lemma \ref{lem_Main_Lemma} basically depends on the freedom to arrange the initial segment of Cohen reals, and it seems to be hard to reflect the freedom to the reals in the ground model.

		%\subsection{Question on FAM-limit}
		\begin{que}
			In addition to Theorem \ref{thm_p_fin_prime_GCH} (Figure \ref{fig_cm_prime}), can we separate $\ee$ and $\eeb$ (and the dual numbers $\pre$ and $\preb$)?
		\end{que}
		
		In fact, even the consistency of $\max\{\ee,\bb\}<\eeb$ is not known ($\ee<\eeb$ and $\bb<\eeb$ are known to be consistent: Brendle \cite{Bre95} proved the consistency of $\ee<\bb(\leq\eeb)$, while he and Shelah \cite{BS_E_and_P_2} proved that of $\bb<\ee(\leq\eeb)$. Also, the latter is obtained as a corollary of Theorem \ref{thm_p6fin}.).
		We can naively define a poset $\pr^*$ which generically adds a bounding-predictor and hence increase $\eeb$, by changing the Definition \ref{dfn_PR} \eqref{item_PR_order_long} to ``For all $ n\in(d^\prime)^{-1}(\{1\})\setminus d^{-1}(\{1\})$ and $f\in F$, we have $f\on n\in\dom(\pi^\prime_n)$ and $\pi^\prime_n(f\on n)=f(n)$''. 
		We can also show that $\pr^*$ has ultrafilter-limits by a similar proof to that of Lemma \ref{lem_pr_is_uflinked} and hence it keeps $\bb$ small. However, it is unclear whether it also keeps $\ee$ small.

		\begin{que}
			\label{que_last}
			Can we additionally separate $\ee_{ubd}$ and $\nonm$ (and the dual numbers $\pre_{ubd}$ and $\covm$)?
		\end{que}
		
		This question has a deep background.
		
		%\textit{FAM-limit} is tabled for discussion.
		After the first construction of Cicho\'n's maximum in \cite{GKS}, Kellner, Shelah and T{\u{a}}nasie \cite{KST} constructed\footnote{It was constructed under the same large cardinal assumption as \cite{GKS}. Later, the assumption was eliminated in \cite{GKMS} introducing the submodel method.} Cicho\'n's maximum for another order %\footnote{Thus, all the two possible models of point-symmetrical orders were constructed at that time.} 
		illustrated in Figure \ref{fig_CM_KST}, introducing the \textit{FAM-limit}\footnote{While they did not use the name ``FAM-limit'' but ``strong FAM-limit for intervals'', we use ``FAM-limit'' in this paper. %``FAM'' is short for ``finitely additive measure'' (on $\omega$). 
			Also, the original idea of the notion is from \cite{She00}.} method, which focuses on (and actually is short for) finitely additive measures on $\omega$ and keeps the bounding number $\bb$ small as the ultrafilter-limit does.
		\begin{figure}
			\centering
			\begin{tikzpicture}
				\tikzset{
					textnode/.style={text=black}, 
				}
				\tikzset{
					edge/.style={color=black, thin}, 
				}
				\newcommand{\w}{2.4}
				\newcommand{\h}{1.5}
				
				\node[textnode] (addN) at (0,  0) {$\addn$};
				\node (t1) [fill=lime, draw, text=black, circle,inner sep=1.0pt] at (-0.25*\w, 0.5*\h) {$\theta_1$};
				
				\node[textnode] (covN) at (0,  2*\h) {$\covn$};
				\node (t3) [fill=lime, draw, text=black, circle,inner sep=1.0pt] at (0.25*\w, 1.5*\h) {$\theta_3$};

				\node[textnode] (addM) at (\w,  0) {$\cdot$};
				\node[textnode] (b) at (\w,  1*\h) {$\bb$};
				\node (t2) [fill=lime, draw, text=black, circle,inner sep=1.0pt] at (0.75*\w, 1.2*\h) {$\theta_2$};
				
				\node[textnode] (nonM) at (\w,  2*\h) {$\nonm$};
				\node (t4) [fill=lime, draw, text=black, circle,inner sep=1.0pt] at (1.25*\w, 2.4*\h) {$\theta_4$};
				
				\node[textnode] (covM) at (\w*2,  0) {$\covm$};
				\node (t5) [fill=lime, draw, text=black, circle,inner sep=1.0pt] at (1.75*\w, -0.4*\h) {$\theta_5$};
				
				\node[textnode] (d) at (\w*2,  1*\h) {$\dd$};
				\node (t7) [fill=lime, draw, text=black, circle,inner sep=1.0pt] at (2.25*\w, 0.8*\h) {$\theta_7$};
				
				\node[textnode] (cofM) at (\w*2,  2*\h) {$\cdot$};

				\node[textnode] (nonN) at (\w*3,  0) {$\nonn$};
				\node (t6) [fill=lime, draw, text=black, circle,inner sep=1.0pt] at (2.75*\w, 0.5*\h) {$\theta_6$};
				
				\node[textnode] (cofN) at (\w*3,  2*\h) {$\cofn$};
				\node (t8) [fill=lime, draw, text=black, circle,inner sep=1.0pt] at (3.25*\w, 1.5*\h) {$\theta_{8}$};
				
				\node[textnode] (aleph1) at (-\w,  0) {$\aleph_1$};
				\node[textnode] (c) at (\w*4,  2*\h) {$2^{\aleph_0}$};
				\node (tcc) [fill=lime, draw, text=black, circle,inner sep=1.0pt] at (3.7*\w, 2.5*\h) {$\theta_\cc$};

				\draw[->, edge] (addN) to (covN);
				\draw[->, edge] (addN) to (addM);
				\draw[->, edge] (covN) to (nonM);	
				\draw[->, edge] (addM) to (b);
				
				\draw[->, edge] (addM) to (covM);
				\draw[->, edge] (nonM) to (cofM);
				
				\draw[->, edge] (d) to (cofM);
				\draw[->, edge] (b) to (d);
				\draw[->, edge] (b) to (nonM);
				\draw[->, edge] (covM) to (nonN);
				\draw[->, edge] (covM) to (d);
				\draw[->, edge] (cofM) to (cofN);
				\draw[->, edge] (nonN) to (cofN);
				\draw[->, edge] (aleph1) to (addN);
				\draw[->, edge] (cofN) to (c);

				\draw[blue,thick] (-0.5*\w,1*\h)--(0.5*\w,1*\h);
				\draw[blue,thick] (2.5*\w,1*\h)--(3.5*\w,1*\h);

				\draw[blue,thick] (0.5*\w,1.5*\h)--(1.5*\w,1.5*\h);
				\draw[blue,thick] (1.5*\w,0.5*\h)--(2.5*\w,0.5*\h);

				\draw[blue,thick] (1.5*\w,-0.5*\h)--(1.5*\w,2.5*\h);
				%\draw[blue,thick] (0.5*\w,2.5*\h)--(0.5*\w,2.5*\h);
				\draw[blue,thick] (2.5*\w,-0.5*\h)--(2.5*\w,2.5*\h);

				\draw[blue,thick] (0.5*\w,-0.5*\h)--(0.5*\w,2.5*\h);
				%\draw[blue,thick] (2.5*\w,-0.5*\h)--(2.5*\w,0.5*\h);
				\draw[blue,thick] (-0.5*\w,-0.5*\h)--(-0.5*\w,2.5*\h);
				\draw[blue,thick] (3.5*\w,-0.5*\h)--(3.5*\w,2.5*\h);

			\end{tikzpicture}
			\caption{Cicho\'n's maximum constructed in \cite{KST}. Compared with the original one in \cite{GKS} (Figure \ref{fig_CM_GKS}), $\bb$ and $\covn$, $\dd$ and $\nonn$ are exchanged.} \label{fig_CM_KST}
		\end{figure}
		
		Later, Goldstern, Kellner, Mej\'{\i}a and Shelah \cite{GKMS_CM_and_evasion} proved that the FAM-limit keeps the evasion number $\ee$ small.
		Recently, Uribe-Zapata \cite{Uri} formalized the theory of the FAM-limits and he, Cardona and Mej\'{\i}a generalized the result above as follows:
		
		\begin{thmm}(\cite{Car23RIMS})
			FAM-limits keep $\non(\mathcal{E})$ small.
		\end{thmm}
		\begin{figure}
			\centering
			\begin{tikzpicture}
				\tikzset{
					textnode/.style={text=black}, 
				}
				\tikzset{
					edge/.style={color=black, thin}, 
				}
				\tikzset{cross/.style={preaction={-,draw=white,line width=7pt}}}
				\newcommand{\w}{2.4}
				\newcommand{\h}{2.0}
				
				\node[textnode] (addN) at (0,  0) {$\addn$};
				\node[textnode] (covN) at (0,  \h*3) {$\covn$};

				\node[textnode] (addM) at (\w,  0) {$\addm$};
				\node[textnode] (addE) at (\w,  -0.5*\h) {$\add(\mathcal{E})$};
				
				\node[textnode] (b) at (\w,  1.3*\h) {$\bb$};
				\node[textnode] (nonM) at (\w,  \h*3) {$\nonm$};
				
				\node[textnode] (covM) at (\w*2,  0) {$\covm$};
				\node[textnode] (d) at (\w*2,  1.7*\h) {$\dd$};
				\node[textnode] (cofM) at (\w*2,  \h*3) {$\cofm$};
				\node[textnode] (cofE) at (\w*2,  3.5*\h) {$\cof(\mathcal{E})$};

				\node[textnode] (nonN) at (\w*3,  0) {$\nonn$};
				\node[textnode] (cofN) at (\w*3,  \h*3) {$\cofn$};
				
				\node[textnode] (aleph1) at (-\w,  0) {$\aleph_1$};
				\node[textnode] (c) at (\w*4,  \h*3) {$2^{\aleph_0}$};
				
				\node[textnode] (e) at (0.5*\w,  1.7*\h) {$\mathfrak{e}$}; 
				\node[textnode] (pr) at (2.5*\w,  1.3*\h) {$\mathfrak{pr}$}; 
				
				\node[textnode] (estar) at (\w,  1.7*\h) {$\ee^*$};
				\node[textnode] (prstar) at (\w*2,  1.3*\h) {$\mathfrak{pr}^*$};

				\node[textnode] (eubd) at (\w*0.2,  \h*2.1) {$\ee_{ubd}$};
				\node[textnode] (prubd) at (\w*2.8,  \h*0.9) {$\mathfrak{pr}_{ubd}$};
				
				\node[textnode] (nonE) at (\w*0.75,  \h*2.55) {$\non(\mathcal{E})$};
				\node[textnode] (covE) at (\w*2.25,  \h*0.45) {$\cov(\mathcal{E})$};

				\draw[->, edge] (addM) to (nonE);
				\draw[->, edge] (covE) to (cofM);		
				\draw[->, edge] (covM) to (covE);		
				\draw[->, edge] (e) to (estar);
				\draw[->, edge] (e) to (covM);
				\draw[->, edge] (addM) to (b);
				\draw[->, edge] (e) to (estar);
				\draw[->, edge] (b) to (prstar);
				\draw[->, edge] (nonM) to (pr);
				\draw[->, edge] (prstar) to (pr);
				
				\draw[->, edge] (addN) to (covN);
				\draw[->, edge] (addN) to (addM);
				\draw[->, edge] (covN) to (nonM);	
				
				%\draw[->, edge] (b) to (nonM);
				\draw[->, edge] (addM) to (covM);
				\draw[->, edge] (nonM) to (cofM);
				%\draw[->, edge] (covM) to (d);
				\draw[->, edge] (d) to (cofM);
				
				\draw[->, edge] (covM) to (nonN);
				\draw[->, edge] (cofM) to (cofN);
				\draw[->, edge] (nonN) to (cofN);
				\draw[->, edge] (aleph1) to (addN);
				\draw[->, edge] (cofN) to (c);
				
				%\draw[->, edge] (e) to (nonM);
				
				\draw[->, edge] (addN) to (e);
				
				\draw[->, edge] (covM) to (prstar);
				
				\draw[->, edge] (pr) to (cofN);

				\draw[->, edge] (b) to (estar);C
				\draw[->, edge] (estar) to (nonM);
				\draw[->, edge] (estar) to (d);
				\draw[->, edge] (e) to (eubd);
				
				\draw[->, edge] (prstar) to (d);
				\draw[->, edge] (prstar) to (pr);

				\draw[->, edge] (prubd) to (pr);

				\draw[->, edge,cross] (eubd) to (nonE);
				\draw[->, edge] (nonE) to (nonM);
				
				\draw[->, edge] (covE) to (prubd);
				\draw[->, edge] (nonE) to (nonN);
				\draw[->, edge] (covN) to (covE);

				\draw[double distance=2pt, edge] (addM) to (addE);
				\draw[double distance=2pt, edge] (cofM) to (cofE);

			\end{tikzpicture}
			\caption{Cicho\'n's diagram and other cardinal invariants.}\label{fig_Cd_E}
		\end{figure}
		$\mathcal{E}$ denotes the $\sigma$-ideal generated by closed null sets and the four numbers related to the ideal $\mathcal{E}$ have the relationship illustrated in Figure \ref{fig_Cd_E} (Bartoszynski and Shelah \cite[Theorem 3.1]{BS92} proved $\add(\mathcal{E})=\addm$ and $\cof(\mathcal{E})=\cofm$. The other new arrows are obtained by easy observations: e.g., a $g$-predictor predicts only $F_\sigma$ null many reals and hence $\ee_{ubd}\leq\non(\mathcal{E})$ and $\cov(\mathcal{E})\leq\pre_{ubd}$ hold). Hence, FAM-limits seem to work for Question \ref{que_last} by keeping $\non(\mathcal{E})$ and $\ee_{ubd}$ small.
		%Since $\non(\mathcal{E})$ is above $\ee$, FAM-limits keeps $\ee$ small.
		
		In \cite{KST}, they also introduced the poset $\widetilde{\mathbb{E}}$ instead of $\mathbb{E}$, which increases $\nonm$ as $\mathbb{E}$ does, but has FAM-limits,  which $\mathbb{E}$ does not have. (Cardona, Mej\'{\i}a and Uribe-Zapata \cite{Car23RIMS} proved that $\mathbb{E}$ increases $\non(\mathcal{E})$ and hence does not have FAM-limits.)
		Thus, if $\widetilde{\mathbb{E}}$ also had closed-ultrafilter-limits, by mixing all the three limit methods as in Theorem \ref{thm_p6fin} and \ref{thm_p_fin_prime_GCH}, it would seem to be possible to additionally separate $\ee_{ubd}$ and $\nonm$ by replacing $\mathbb{E}$ with $\widetilde{\mathbb{E}}$ through the iteration.
		In fact, Goldstern, Kellner, Mej\'{\i}a and Shelah \cite{GKMS_CM_and_evasion} proved that \textit{$\widetilde{\mathbb{E}}$ actually has (closed-)ultrafilter-limits}.

		However, recall that when constructing the names of ultrafilters at limit steps in the proof of Lemma \ref{lem_uf_const_all}, we resorted to the \textit{centeredness}, which $\widetilde{\mathbb{E}}$ \textit{does not have}. %\footnote{$\widetilde{\mathbb{E}}$ is known to be between $\sigma$-linked and $\sigma$-centered: for $f,g\in\oo$ going to infinity, a poset $\p$ is $(f,g)$-linked if $\p=\bigcup_{m<\omega}\bigcap_{i>m}\bigcup_{j<f(i)}Q^i_j$ where each $Q^i_j$ is $g(i)$-linked (see e.g. \cite{OK}). Thus, $(f,g)$-linked lies between $\sigma$-linked and $\sigma$-centered (and hence it is ccc).
			%$\widetilde{\mathbb{E}}$ is proved to be $(\rho,\pi)$-linked for some (concrete) $\rho,\pi\in\oo$ (\cite[Lemma 1.19.]{KST}).}.
		We have no idea on how to overcome this problem without resorting to centeredness\footnote{This is the point where Goldstern, Kellner, Mej\'{\i}a and Shelah \cite{GKMS_CM_and_evasion} found a gap in their proof, mentioned in footnote \ref{foot_MR} in Section \ref{sec_int}.\label{foot_q}}.

		\begin{acknowledgements}
			The author thanks his supervisor J\"{o}rg Brendle for his invaluable comments and Diego~A. Mej\'{\i}a for his helpful advice. The author is also grateful to the anonymous referee for his/her corrections and suggestions.
		\end{acknowledgements}

	\end{document}